\documentclass[12pt]{amsart}

\usepackage{amsfonts,amssymb,enumerate,color,bm,hhline,makecell,array}
\usepackage{array}
\usepackage{mathrsfs}
\usepackage[all,ps,cmtip]{xy} 
\usepackage[normalem]{ulem}
\usepackage[colorlinks=true,citecolor=blue, urlcolor=blue, linkcolor=blue]{hyperref}

\usepackage{enumitem}
\setlist[itemize]{noitemsep,nolistsep}
\setlist[enumerate]{noitemsep,nolistsep}
\usepackage{colortbl}
\usepackage{enumerate}

   \allowdisplaybreaks

\let\mathcal\mathscr

\def\P{\mathbb{P}}
\def\Z{\mathbb{Z}}
\def\R{\mathbb{R}}
\def\Q{\mathbb{Q}}

\def\C{\mathbb{C}}

\def\phi{{\varphi}}

\def\vv{{\mathbf v}}

\def\cA{\mathcal{A}}
\def\cB{\mathcal{B}}
\def\cC{\mathcal{C}}
\def\cD{\mathcal{D}}
\def\cE{\mathcal{E}}
\def\cF{\mathcal{F}}

\def\cH{\mathcal{H}}
\def\cI{\mathcal{I}}

\def\cO{\mathcal{O}}
\def\cP{\mathcal{P}}

\def\cT{\mathcal{T}}
\def\cU{\mathcal{U}}
\def\cW{\mathcal{W}}
\def\cX{\mathcal{X}}

\newcommand{\mor}[1][]{\xrightarrow{#1}}

\def\lra{\longrightarrow}
\def\llra{\hbox to 10mm{\rightarrowfill}}
\def\lllra{\hbox to 15mm{\rightarrowfill}}

\def\llla{\hbox to 10mm{\leftarrowfill}}
\def\lllla{\hbox to 15mm{\leftarrowfill}}

\def\thra{\twoheadrightarrow}
\def\hra{\hookrightarrow}

\def\K{\mathbb K}

\DeclareMathOperator{\isomto}{\stackrel{{}_{\scriptstyle\sim}}{\to}}

\DeclareMathOperator{\ch}{ch}

\DeclareMathOperator{\coh}{coh}

\DeclareMathOperator{\coker}{Coker}

\DeclareMathOperator{\Db}{D\textsuperscript{\rm b}}

\DeclareMathOperator{\Ext}{Ext}

\DeclareMathOperator{\HH}{HH}

\DeclareMathOperator{\Hom}{Hom}
\DeclareMathOperator{\Id}{Id}

\def\Im{\mathop{\rm Im}\nolimits}

\DeclareMathOperator{\Ker}{Ker}

\DeclareMathOperator{\NS}{NS}

\DeclareMathOperator{\rk}{rk}

\newcommand{\HO}{\mathrm{H}\Omega}
\newcommand{\HT}{\mathrm{HT}}
\def\Dd{\Delta}

\def\onto{\ensuremath{\twoheadrightarrow}}
 
\def\llra{\hbox to 10mm{\rightarrowfill}}
\def\lllra{\hbox to 15mm{\rightarrowfill}}
\def\Ku{\mathcal{K}\!u}

\newtheorem{lemm}{Lemma}[section]
\newtheorem{theo}[lemm]{Theorem}
\newtheorem{coro}[lemm]{Corollary}
\newtheorem{prop}[lemm]{Proposition}

\theoremstyle{definition}
\newtheorem{defi}[lemm]{Definition}
\newtheorem{rema}[lemm]{Remark}

\newtheorem{exam}[lemm]{Example}
\newtheorem{ques}[lemm]{Question}
\newtheorem{setup}[lemm]{Setup}

\theoremstyle{remark}
\newtheorem*{remark*}{Remark}
\newtheorem*{note*}{Note}

\def\blank{\underline{\hphantom{A}}}
\def\llambda{\ensuremath{\boldsymbol{\lambda}}}

\newcommand{\changelocaltocdepth}[1]{%
  \addtocontents{toc}{\protect\setcounter{tocdepth}{#1}}%
  \setcounter{tocdepth}{#1}%
}

\makeatletter
\def\@tocline#1#2#3#4#5#6#7{\relax
  \ifnum #1>\c@tocdepth 
  \else
    \par \addpenalty\@secpenalty\addvspace{#2}%
    \begingroup \hyphenpenalty\@M
    \@ifempty{#4}{%
      \@tempdima\csname r@tocindent\number#1\endcsname\relax
    }{%
      \@tempdima#4\relax
    }%
    \parindent\z@ \leftskip#3\relax \advance\leftskip\@tempdima\relax
    \rightskip\@pnumwidth plus4em \parfillskip-\@pnumwidth
    #5\leavevmode\hskip-\@tempdima
      \ifcase #1
       \or\or \hskip 1em \or \hskip 2em \else \hskip 3em \fi%
      #6\nobreak\relax
    \dotfill\hbox to\@pnumwidth{\@tocpagenum{#7}}\par
    \nobreak
    \endgroup
  \fi}
\makeatother

\begin{document}

\title[Non-commutative K3 surfaces, Bridgeland stability, and moduli spaces]{Lectures on non-commutative K3 surfaces, Bridgeland stability, and moduli spaces}

\author[E.~Macr\`i]{Emanuele Macr\`i}
\address{\parbox{0.9\textwidth}{Northeastern University
\\[1pt] 
Department of Mathematics
\\[1pt]
360 Huntington Avenue, Boston, MA 02115, USA\vspace{4pt}}}
\email{{e.macri@northeastern.edu}}
\urladdr{\url{https://web.northeastern.edu/emacri/}\vspace{0,1cm}}

\author[P.~Stellari]{Paolo Stellari}
\address{\parbox{0.9\textwidth}{Universit\`a degli Studi di Milano\\[1pt]
Dipartimento di Matematica ``F.~Enriques''\\[1pt]
Via Cesare Saldini 50, 20133 Milano, Italy\vspace{4pt}
}}
\email{paolo.stellari@unimi.it}
\urladdr{\url{http://users.unimi.it/stellari}}


\thanks{E.~M.~ was partially supported by the NSF grant DMS-1700751. P.~S.~ was partially supported by the ERC Consolidator Grant ERC-2017-CoG-771507-StabCondEn and by the research projects FIRB 2012 ``Moduli Spaces and Their Applications'' and PRIN 2015 ``Geometria delle variet\`a proiettive''.}

\begin{abstract}
We survey the basic theory of non-commutative K3 surfaces, with a particular emphasis to the ones arising from cubic fourfolds.
We focus on the problem of constructing Bridgeland stability conditions on these categories and we then investigate the geometry of the corresponding moduli spaces of stable objects. We discuss a number of consequences related to cubic fourfolds including new proofs of the Torelli theorem and of the integral Hodge conjecture, the extension of a result of Addington and Thomas and various applications to hyperk\"ahler manifolds.

These notes originated from the lecture series by the
first author at the school on \emph{Birational Geometry of Hypersurfaces}, Palazzo Feltrinelli - Gargnano del Garda (Italy), March 19-23, 2018.
\end{abstract}

\maketitle
\tableofcontents

\setcounter{tocdepth}{1}

\section{Introduction}
\label{sec:intro}

K3 surfaces have been extensively studied during the last decades of the previous century. The techniques used to understand their geometry include Hodge theory, lattice theory and homological algebra. In fact the lattice and Hodge structures on their second cohomology groups determine completely their geometry in the following sense:

\medskip

\begin{enumerate}
\item[(K3.1)] {\bf Torelli Theorem:} Two K3 surfaces $S_1$ and $S_2$ are isomorphic if and only if there is a Hodge isometry $H^2(S_1,\Z)\cong H^2(S_2,\Z)$.
\end{enumerate}

\medskip
 
The result is originally due to Pijatecki\u{\i}-\u{S}apiro and \u{S}afarevi\u{c} in the algebraic case \cite{PS} and to Burns and Rapoport in the analytic case \cite{BR} (see also \cite{LP,Fri} for other proofs).
Another way to rephrase this is in terms of fibers of the period map for K3 surfaces: the period map is generically injective. Such a map turns out to be surjective as well. Roughly speaking this means that one can get complete control on which weight-$2$ Hodge structures on the abstract lattice correspond to an actual surface.

One can go further and construct other compact \emph{hyperk\"ahler manifolds} out of K3 surfaces. Following Beauville \cite{Be}, one could first consider Hilbert schemes (or Douady spaces) of points on such surfaces. More generally, for a \emph{projective} K3 surface $S$ and a primitive vector $\vv$ in the algebraic part of the total cohomology $H^*(S,\Z)$, one can construct the moduli space $M_H(S,\vv)$ of stable sheaves $E$ on $S$ with Mukai vector $v(E):=\mathrm{ch}(E)\cdot\sqrt{\mathrm{td}(S)}=\vv$, for the choice of an ample line bundle $H$ which is generic with respect to $\vv$.

\medskip

\begin{enumerate}
	\item[(K3.2)] The moduli space $M_H(S,\vv)$ is a smooth projective hyperk\"ahler manifold of dimension $\vv^2+2$ which is deformation equivalent to a Hilbert scheme of points on a K3 surface.
\end{enumerate}

\medskip

Such moduli spaces deform together with a polarized deformation of the K3 surface $S$ and they yield $19$-dimensional families of hyperk\"ahler manifolds. Here the square of $\vv$ is taken with respect to the so called \emph{Mukai pairing} which is defined on the total cohomology $H^*(S,\Z)$ of $S$. The above statement is the result of many different contributions, starting with the foundational work of Mukai \cite{Muk:Sympl,Muk:K3}, and it is due to Huybrechts \cite{Huy:bir}, O'Grady \cite{O'G:HK}, and Yoshioka. The statement in its final form is \cite[Theorems 0.1 and 8.1]{Yo:ModAbVar}). It also gives a precise non-emptyness statement: if the Mukai vector is \emph{positive}, then the moduli space is non-empty if and only if it has non-negative expected dimension.

More recently, after the works of Mukai \cite{Muk:K3} and Orlov \cite{Orl:FM}, the bounded derived categories of coherent sheaves of K3 surfaces and their autoequivalence groups have been extensively studied. The Torelli theorem (K3.1) has a homological counterpart:

\medskip

\begin{enumerate}
\item[(K3.3)] {\bf Derived Torelli Theorem:} Given two K3 surfaces $S_1$ and $S_2$, then $\Db(S_1)\cong\Db(S_2)$ are isomorphic if and only if there is an isometry between the total cohomology lattices of $S_1$ and $S_2$ preserving the Mukai weight-$2$ Hodge structure.
\end{enumerate}

\medskip

The lattice and Hodge structure mentioned above will be explained in Section \ref{subsec:Mukai}. In the geometric setting considered in (K3.3), it might be useful to keep in mind that the pairing coincides with the usual Euler form on the algebraic part of the total cohomologies of $S_1$ and $S_2$. On the other hand, the $(2,0)$-part in the Hodge decomposition mentioned above is nothing but $H^{2,0}(S_i)$.

Thus the existence of an equivalence can be detected again by looking at the Hodge and lattice structure of the total cohomology (and not just of $H^2$). Pushing this further, one can observe that $\Db(S)$ carries additional structures: \emph{Bridgeland stability conditions}. Indeed, after the works by Bridgeland \cite{Bri:stability,Bri:K3} we have the following fact:

\medskip

\begin{enumerate}
	\item[(K3.4)] If $S$ is a K3 surface, then the manifold parameterizing stability conditions on $\Db(S)$ is non-empty and a connected component is a covering of a generalized period domain.
\end{enumerate}

\medskip

Following the pattern for stability of sheaves, one can take a stability condition $\sigma$ and  a primitive vector $\vv$ as in (K3.2) such that  $\sigma$ is generic with respect to $\vv$. If one considers the moduli space $M_\sigma(\Db(S),\vv)$ of $\sigma$-stable objects in $\Db(S)$ with Mukai vector $\vv$, then by \cite{Toda:K3,BM:projectivity,BM:walls,MYY}, we have the analogue of (K3.4):

\medskip

\begin{enumerate}
	\item[(K3.5)] The moduli space $M_\sigma(\Db(S),\vv)$ is a smooth projective hyperk\"ahler manifold of dimension $\vv^2+2$ which is deformation equivalent to a Hilbert scheme of points on a K3 surface.
\end{enumerate}

\bigskip

Quite surprisingly, a similar picture appears in a completely different setting when the canonical bundle is far from being trivial: Fano varieties.
The first key example is the case of smooth \emph{cubic fourfolds}.

\medskip

\begin{enumerate}
\item[(C.1)] {\bf Torelli Theorem:} Two cubic fourfolds $W_1$ and $W_2$ defined over $\C$ are isomorphic if and only if there is a Hodge isometry $H^4(W_1,\Z)\cong H^4(W_2,\Z)$ that preserves the square of a hyperplane class.
\end{enumerate}

\medskip

The result is due to Voisin \cite{Voi:Torelli} (other proofs were given in \cite{Looijenga:cubics,Charles:Torelli,HR:cubics}).
The striking similarity with K3 surfaces is confirmed by the fact that $H^4(W,\Z)$  carries actually a weight-$2$ Hodge structure of K3 type. Moreover, it was observed by Hassett \cite{Has:special} that the $20$-dimensional moduli space $\cC$ of cubic fourfolds contain divisors of Noether-Lefschetz type and some of them parametrize cubic fourfolds $W$ with a \emph{Hodge-theoretically associated K3 surface}. Roughly, this simply means that the orthogonal in $H^4(W,\Z)$ of a rank $2$ sublattice, generated by the self-intersection $H^2$ of a hyperplane class and by some special surface which is not homologous to $H^2$, looks like the primitive cohomology of a polarized K3 surface. 

The hyperk\"ahler geometry that one can associate to a cubic fourfold is actually quite rich. Indeed, the Fano variety of lines in a cubic fourfold is a $4$-dimensional projective hyperk\"ahler manifold which is deformation equivalent to a Hilbert scheme of points on a K3 surface (see \cite{BD:cubic}). More recently, it was proved by Christian Lehn, Manfred Lehn, Sorger and van Straten \cite{LLSvS:TwistedCubics} that the moduli space of generalized twisted cubics inside a cubic fourfold also gives rise to a $8$-dimensional projective hyperk\"ahler manifold which is again deformation equivalent to a Hilbert scheme of points on a K3 surface. All these constructions work in families and thus provide $20$-dimensional locally complete families of $4$ or $8$ dimensional polarized smooth projective hyperk\"ahler manifolds.

From the homological point of view, it was observed by Kuznetsov \cite{Kuz:CubicFourfolds} that the derived category $\Db(W)$ of a cubic fourfold $W$ contains an admissible subcategory $\Ku(W)$ which is the right orthogonal of the category generated by the three line bundles $\cO_W$, $\cO_W(H)$ and $\cO_W(2H)$. We will refer to $\Ku(W)$ as the \emph{Kuznetsov component} of $W$. The category $\Ku(W)$ has the same homological properties of $\Db(S)$, for $S$ a K3 surface: it is an indecomposable category with Serre functor which is the shift by $2$ and the same Hochschild homology as $\Db(S)$. But, for $W$ very general, there cannot be a K3 surface $S$ with an equivalence $\Ku(W)\cong\Db(S)$. This is the reason why we should think of Kuznetsov components as \emph{non-commutative K3 surfaces}.

The study of non-commutative varieties was started more than thirty years ago. Artin and Zhang \cite{AZ:noncomm} investigated the case of non-commutative projective spaces (see also the book in preparation \cite{Yek1}). In these notes we will follow closely the approach developed by Kuznetsov \cite{Kuz:CubicFourfolds} and Huybrechts \cite{Huy:cubics}. One important feature is that the Kuznetzov component comes with a naturally associated lattice $\widetilde{H}(\Ku(W),\Z)$ with a weight-$2$ Hodge structure. The lattice is actually isometric to the extended K3 lattice mentioned in (K3.3). Hence, as for K3 surfaces, it is natural to expect that (C.1) has a homological counterpart in the same spirit as (K3.3):

\medskip

\begin{enumerate}
	\item[(C.2)] {\bf (Conjectural) Derived Torelli Theorem for cubic fourfolds:} We expect that, if $W_1$ and $W_2$ are cubic fourfolds, then $\Ku(W_1)\cong\Ku(W_2)$ if and only if there is a Hodge isometry $\widetilde{H}(\Ku(W_1),\Z)\cong\widetilde{H}(\Ku(W_2),\Z)$ which preserves the orientation of $4$-positive directions.
\end{enumerate}

\medskip

Such a conjecture will be explained later in this paper but it is worth pointing out that it has been proved generically by Huybrechts \cite{Huy:cubics}.

One of the aims of these notes is to show how, following \cite{BLMS:inducing,BLMNPS:families}, properties (K3.4) and (K3.5) generalize to this setting. Namely, let $W$ be a cubic fourfold and let $\vv$ be a primitive vector in the algebraic part of the total cohomology $\widetilde{H}(\Ku(W),\Z)$ of $\Ku(W)$.

\medskip

\begin{enumerate}
	\item[(C.3)] The manifold parameterizing stability conditions on $\Ku(W)$ is non-empty and a connected component is a covering of a generalized period domain.

\medskip

	\item[(C.4)] The moduli space $M_\sigma(\Ku(W),\vv)$ is a smooth projective hyperk\"ahler manifold of dimension $\vv^2+2$ which is deformation equivalent to a Hilbert scheme of points on a K3 surface, if $\sigma$ is a stability condition as in (C.3) which is generic with respect to $\vv$.

\end{enumerate}

\medskip

Again, both constructions work in families (in an appropriate sense) and thus we get $20$-dimensional locally complete families of smooth projective hyperk\"ahler manifolds. Also, we recover the Fano variety of lines and the manifold related to twisted cubics as instances of (C.4).
The geometric applications of (C.3) and (C.4) are even more and include a reproof of the Torelli Theorem (C.1), a new simple proof of the integral Hodge conjecture for cubic fourfolds and the extension of a result by Addington and Thomas \cite{AT:cubic} saying that for a cubic $W$ having a Hodge-theoretically associated K3 surface is the same as having an equivalence $\Ku(W)\cong\Db(S)$, for a smooth projective K3 surface $S$. All these results will be discussed in this paper.

The interesting point is that we expect (C.3) and (C.4) to hold for other interesting classes of Fano manifolds with naturally associated non-commutative K3 surfaces. This is the case of Gushel-Mukai and Debarre-Voisin manifolds. We will discuss this later on in the paper.

Although we do our best to make this paper as much self-contained as possible, there are various Hodge-theoretical aspects of the theory of cubic fourfolds that are not discussed here and that can be found in the lecture notes of Huybrechts \cite{Huy:LectureNotesCubic} that appear in the same volume as the present ones. Moreover, we only touch  very briefly in our treatment the fundamental theory of Homological Projective Duality: we refer to the original articles \cite{Kuz:HPD,Perry:NCHP,KP:Joins}. In \cite{Kuz:survey}, the relation between rationality questions and geometric properties of Kuznetsov components is widely discussed. Finally, an excellent survey of several aspects of the theory of semiorthogonal decompositions is \cite{Kuz:ICM}.

\subsection*{Plan of the paper}

Let us briefly sketch how the paper is organized. In Section \ref{sec:CY} we introduce the general material concerning non-commutative varieties with an emphasis on those that are of Calabi-Yau type. After presenting semiorthogonal decompositions and exceptional objects (see Section \ref{subsec:SO}), we discuss in Section \ref{subsec:NC} the notion of non-commutative variety in general. In our framework this just means an admissible subcategory $\cD$ of the bounded derived category $\Db(X)$, for $X$ a smooth projective variety. As such, it comes with a Serre functor that we compute and with a well-defined notion of product, which we use  to define the Hochschild homology and cohomology of $\cD$. A non-commutative Calabi-Yau variety is a non-commutative variety whose Serre functor is the shift by an integer $n$. In Section \ref{subsec:NCCY} we discuss a general framework to construct examples of such non-commutative varieties.

Section \ref{sec:examples} has a more geometric flavor. It deals with the constructions of non-commutative K3 categories out of the derived categories of Fano manifolds. The first case we analyze is the one of cubic fourfolds $W$ (see Section \ref{subsec:CubicFourfolds}) where the non-commutative K3 surface is the K3 category $\Ku(W)$ mentioned above. As a result, we state a (generalized version of a) result of Addington and Thomas (see Theorem \ref{thm:AT}) which characterizes completely the loci in the moduli space of cubic fourfolds parameterizing the cubics whose Kuznetsov component is actually equivalent to the bounded derived category of an actual (twisted) K3 surface. Our proof is provided in Section \ref{subsec:applications} and it is based on the use of stability conditions and of moduli spaces of stable objects. In Sections \ref{subsec:GM} and \ref{subsec:DV} we study two other classes of Fano manifolds: Gushel-Mukai and Debarre-Voisin manifolds. Finally, in Section \ref{subsec:Torelli}, by using techniques developed in Section \ref{subsec:Mukai}, we present derived variants of the Torelli theorem for cubic fourfolds and discuss conjectural relations to the birational type of these fourfolds.

Section \ref{sec:Bridgeland} is about Bridgeland stability conditions on Kuznetsov components. After recalling the basic definitions and properties in Section \ref{subsec:BridgelandMain}, we outline a few techniques: the tilting procedure (see Section \ref{subsec:Tilt}) and the way this induces stability conditions on semiorthogonal components (see Section \ref{subsec:inducing}).

We state and prove our first main results in Section \ref{sec:CubicFourfolds}, where we deal with stability conditions and moduli spaces of stable objects in the Kuznetsov component of cubic fourfolds. Theorem \ref{thm:main1} and Theorem \ref{thm:stabconn} are the main results. They prove essentially what we claim in (C.3). The proof requires the techniques in Section \ref{sec:Bridgeland} and a non-commutative Bogomolov inequality proved in Theorem \ref{thm:Bogomolov}. In Section \ref{subsec:ModuliSpaces} we study moduli spaces and prove Theorem \ref{thm:YoshiokaMain} which yields (C.4). In the rest of the paper (see Section \ref{subsec:applications}) we discuss several applications of these theorems.
These include the complete proof of Theorem \ref{thm:AT}, the proof of the integral Hodge Conjecture for cubic fourfolds (see Proposition \ref{prop:intHodge}) originally due to Voisin, and various constructions of moduli spaces associated to low degree curves in cubic fourfolds (see Theorems \ref{thm:Fano} and \ref{thm:twistcub}). We also briefly discuss a homological approach to the Torelli theorem (C.1) (see Theorem \ref{thm:Torelli}).

\subsection*{Notation and conventions}

We work over an algebraically closed field $\K$; often, when the characteristic of the field is not zero, we will assume it to be sufficiently large. When $\K=\C$, the field of complex numbers, we set $\mathfrak{i}=\sqrt{-1}$.
All categories will be $\K$-linear, namely the morphism sets have the structure of $\K$-vector space and the compositions of morphisms are $\K$-bilinear maps, and all varieties will be over $\K$.

We assume a basic knowledge of abelian, derived, and triangulated categories.
Basic references are, for example, \cite{GM:HomologicalAlgebra,Huy:FM}.
Given an algebraic variety $X$, we will denote by $\Db(X):=\Db(\coh(X))$ the bounded derived category of coherent sheaves on $X$. All derived functors will be denoted as if they were underived.
If $X$ and $Y$ are smooth projective varieties, a functor $F\colon\Db(X)\to\Db(Y)$ is called a \emph{Fourier-Mukai functor} if it is isomorphic to $\Phi_{P}(\blank):=p_{Y*}(P\otimes p_X^*(\blank))$, for some object $P\in\Db(X\times Y)$.

We also expect the reader to have some familiarity with K3 surfaces \cite{Huy:BookK3} and projective hyperk\"ahler manifolds \cite{GHJ:book,Debarre:survey}.

\subsection*{Acknowledgements}

Our warmest thank goes to Alexander Kuznetsov: his many suggestions, corrections and observations helped us very much to improve the quality of this article.
It is also our great pleasure to thank Arend Bayer, Andreas Hochenegger, Mart\'i Lahoz, Howard Nuer, Alex Perry, Laura Pertusi, and Xiaolei Zhao for their insightful comments on the subject of these notes and for carefully reading a preliminary version of this paper. We are very grateful to Nick Addington, Enrico Arbarello and Daniel Huybrechts for many useful conversations and for patiently answering our questions, and to Amnon Yekutieli for pointing out the references \cite{AZ:noncomm} and \cite{Yek1}. We would also like to thank Andreas Hochenegger and Manfred Lehn for their collaboration in organizing the school these notes originated from, and the audience for many comments, critiques, and suggestions for improvements. Part of this paper was written while the second author was visiting Northeastern University. The warm hospitality is gratefully acknowledged.

\changelocaltocdepth{2}

\section{Non-commutative Calabi-Yau varieties}
\label{sec:CY}

In this section, we follow closely the presentation and main results in \cite{Kuz:CY}; foundational references are also \cite{BO:semiorthogonal,BvdB:NC,Kuz:HPD,Kuz:ICM,Perry:NCHP}.
We start with a very short review on semiorthogonal decompositions and exceptional collections in Section \ref{subsec:SO}. We then introduce the notion of non-commutative variety in Section \ref{subsec:NC} and study basic facts about Serre functors and Hochschild (co)homology. Finally, in Section \ref{subsec:NCCY} we sketch the proof of a result by Kuznetsov (Theorem \ref{thm:KuznetsovMain}), which provides a general method to construct non-commutative Calabi-Yau varieties.

\subsection{Semiorthogonal decompositions}
\label{subsec:SO}

Let $\cD$ be a triangulated category.

\begin{defi}
\label{def:SO}
A \emph{semiorthogonal decomposition} 
\begin{equation*}
\cD = \langle \cD_1, \dots, \cD_m \rangle
\end{equation*}
is a sequence of full triangulated subcategories $\cD_1, \dots, \cD_m$ of $\cD$ 
--- called the \emph{components} of the decomposition --- such that: 
\begin{enumerate}
\item \label{definition-sod-1} 
$\Hom(F, G) = 0$ for all $F \in \cD_i$, $G \in \cD_j$ and $i>j$.
\item \label{definition-sod-2}
For any $F \in \cD$, there is a sequence of morphisms
\begin{equation*}
0 = F_m \to F_{m-1} \to \cdots \to F_1 \to F_0 = F,
\end{equation*}
such that $\mathrm{cone}(F_i \to F_{i-1}) \in \cD_i$ for $1 \leq i \leq m$. 
\end{enumerate}
\end{defi}

\begin{rema}
\label{remark-projection-functors}
Condition~\eqref{definition-sod-1} of the definition implies that the 
``filtration'' in~\eqref{definition-sod-2} and its ``factors'' are unique and functorial. 
The functor $\delta_i \colon \cD \to \cD_i$ given by the $i$-th ``factor'', i.e., 
\begin{equation*}
\delta_i(F) = \mathrm{cone}(F_i \to F_{i-1}) ,
\end{equation*}
is called the \emph{projection functor} onto $\cD_i$.
In the special case $\cD=\langle \cD_1,\cD_2\rangle$, the functor $\delta_1$ is the left adjoint of the inclusion $\cD_1\hra\cD$, while the functor $\delta_2$ is the right adjoint of $\cD_2\hra\cD$.
\end{rema}

Examples of semiorthogonal decompositions generally arise from exceptional collections, together with the concept of admissible subcategory.

\begin{defi}\label{def:admissible}
A full triangulated subcategory $\cC\subset \cD$ is called \emph{admissible} if the inclusion functor admits left and right adjoints.
\end{defi}

For a subcategory $\cC\subset \cD$, we define its \emph{left} and \emph{right orthogonals} as
\begin{align*}
{}^{\perp}\cC & :=  \{ \, G \in \cD  \mid  \Hom(G, F) = 0 \text{ for all } F \in \cC \, \}, \\ 
\cC^{\perp} & :=  \{ \, G \in \cD  \mid  \Hom(F, G) = 0 \text{ for all } F \in \cC \, \}.
\end{align*}

The following is well-known (see \cite[Lemma 2.3]{Kuz:survey} for a more general statement).

\begin{prop}\label{prop:admissible=>SO}
Let $\cC\subset \cD$ be an admissible subcategory.
Then there are semiorthogonal decompositions
\begin{equation}\label{eq:admissible=>SO}
\cD = \langle \cC, {^\perp}\cC \rangle \qquad \text{ and } \qquad \cD = \langle \cC^{\perp}, \cC \rangle.
\end{equation}
\end{prop}

\begin{proof}
Let us denote by $\Psi_L\colon\cD \to \cC$ the left adjoint to the inclusion functor. Then, for all $G\in\cD$, we have a morphism $G\to \Psi_L(G)$ in $\cD$.
The cone of this morphism is in $\cC^\perp$, thus giving the first semiorthogonal decomposition. The second one is analogous, by using the right adjoint.
\end{proof}

Proposition \ref{prop:admissible=>SO} also has a converse statement.
Namely, if a category $\cC$ arises as semiorthogonal component in both decompositions as in \eqref{eq:admissible=>SO}, then it must be admissible.
This follows immediately from Remark \ref{remark-projection-functors}.

Examples of admissible categories are given by exceptional collections.

\begin{defi}\label{def:ExceptionalCollection}
(i) An object $E\in\cD$ is \emph{exceptional} if $\Hom(E,E[p])=0$, for all integers $p\neq0$, and $\Hom(E,E)\cong \K$.

(ii) A set of objects $\{E_1,\ldots,E_m\}$ in $\cD$ is an \emph{exceptional collection} if $E_i$ is an exceptional object, for all $i$, and $\Hom(E_i,E_j[p])=0$, for all $p$ and all $i>j$.

(iii) An exceptional collection $\{E_1,\ldots,E_m\}$ in $\cD$ is \emph{full} if the smallest full triangulated subcategory of $\cD$ containing the exceptional collection is equivalent to $\cD$.

(iv) An exceptional collection $\{E_1,\ldots,E_m\}$ in $\cD$ is \emph{strong} if $\Hom(E_i,E_j[p])=0$, for all $p\neq0$ and all $i<j$.
\end{defi}

\begin{prop}\label{exceptional=>admissible}
Let $\cC:=\langle E\rangle$ be the smallest full triangulated subcategory of a proper\footnote{A triangulated category $\cD$ is \emph{proper} over $\K$ if, for all $F,G\in\cD$, $\dim_\K (\oplus_p \Hom_{\cD}(F,G[p]))<+\infty$.} triangulated category $\cD$  and containing the exceptional object $E$. Then $\cC$ is admissible.
\end{prop}

\begin{proof}
This can be seen directly by constructing explicitly the left and right adjoints to the functor $\Db(\K)\to\cD$, $V\mapsto E\otimes_\K V$ (inducing an equivalence $\Db(\K)\cong\cC$). 
More explicitly, since $E$ is exceptional, given an object $C$ in $\cD$, consider the (canonical) evaluation morphism
\[
\bigoplus_k\Hom(E,C[k])\otimes E[-k]\to C,
\]
where the first object is in $\cC$.
Complete it to a distinguished triangle
\[
\bigoplus_k\Hom(E,C[k])\otimes E[-k]\to C\to D.
\]
Since $E$ is exceptional, we get $\Hom(E,D[p])=0$, for all integers $p$. Hence $D\in\langle E\rangle^\perp$.

Therefore, we constructed a semiorthogonal decomposition
\[
\cD = \langle \langle E\rangle^\perp, \langle E \rangle \rangle. 
\]
Similarly, we can construct a decomposition $\cD = \langle \langle E\rangle, {^\perp}\langle E \rangle \rangle$, and so the category $\cC$ is admissible.
\end{proof}

By using Proposition \ref{prop:admissible=>SO}, if $\{E_1,\ldots,E_m\}$ is an exceptional collection in a proper triangulated category $\cD$ and $\cC$ denotes the smallest full triangulated subcategory of $\cD$ containing the $E_i$'s, then we get two semiorthogonal decompositions
\[
\cD=\langle\cC^\perp,E_1,\ldots,E_m\rangle=\langle E_1,\ldots,E_m,{}^\perp\cC\rangle.
\]
Here, for sake of simplicity, we write $E_i$ for the category $\langle E_i\rangle$. 

\begin{exam}[Beilinson]\label{ex:Pn}
A beautiful example is provided by the $n$-dimensional projective space $\P^n$. In this case, by a classical result of Beilinson \cite{Bei:Pn}, the line bundles
\[
\{\cO_{\P^n}(-n),\cO_{\P^n}(-n+1),\ldots,\cO_{\P^n}\}
\]
form a full strong exceptional collection and so they yield a semiorthogonal decomposition
\[
\Db(\P^n)=\langle \cO_{\P^n}(-n),\cO_{\P^n}(-n+1),\ldots,\cO_{\P^n}\rangle.
\]

The exceptionality and strongness of the 
collection follow from Bott's theorem.
The basic idea for the proof of fullness is to use the resolution of the diagonal $\Delta\subset\P^n\times\P^n$ given by the Koszul resolution
\[
0\to \cO_{\P^n}(-n)\boxtimes \Omega^n(n) \to \ldots \to \cO_{\P^n}(-1)\boxtimes \Omega^1(1) \to \cO_{\P^n\times\P^n} \to \cO_\Delta \to 0
\]
associated to a natural section $\cO_{\P^n}(-1)\boxtimes \Omega^1(1) \to \cO_{\P^n\times\P^n}$ (see, for example, \cite[Section 8.3]{Huy:FM} for the details).

Actually, any set of line bundles $\{\cO_{\P^n}(k),\cO_{\P^n}(k+1),\ldots,\cO_{\P^n}(k+n)\}$, for $k$ any integer, is a full strong exceptional collection in $\Db(\P^n)$.
\end{exam}

\begin{exam}[Kapranov et al.]\label{ex:Grass}
Consider now the Grassmannian $\mathrm{Gr}(k,n)$ of $k$-dimensional subspaces in an $n$-dimensional $\K$-vector space. Let $\mathcal{U}$ be the tautological subbundle and $\mathcal{Q}$ the tautological quotient. If $\mathrm{char}(\K)=0$ (or sufficiently large), Kapranov \cite{Kap:Grass} has constructed a full strong exceptional collection, and so a semiorthogonal decomposition
\[
\Db(\mathrm{Gr}(k,n))=\langle \Sigma^\alpha\mathcal{U}^\vee\rangle_{\alpha\in R(k,n-k)},
\]
where $R(k,n-k)$ is the $k \times (n-k)$ rectangle, $\alpha$ is a Young diagram, and $\Sigma^\alpha$ is the associated Schur functor. 
The basic idea of the proof is similar to the projective space case, by using the Borel-Bott-Weil Theorem for proving that the above collection is exceptional and a Koszul resolution of the diagonal associated to a canonical section of $\mathcal{U}^\vee\boxtimes \mathcal{Q}$.
When $\mathrm{char}(\K)>0$, the situation is more complicated and described in \cite{BLV:Grass}. In mixed characteristic, it is worth mentioning \cite{Efimov} where  semiorthogonal decompositions for the derived categories of Grassmannians over the integers are studied.

For later use, we need a different exceptional collection, described more recently in \cite{Fon:KP} (see also \cite{KP:ExceptionalHomogeneous} for related results).
We assume that $\mathrm{gcd}(k,n)=1$.
Then we have a semiorthogonal decomposition in a form which is more similar to the case of the projective space
\[
\Db(\mathrm{Gr}(k,n))=\langle \cB,\cB(1),\ldots,\cB(n-1)\rangle,
\]
with the category $\cB$ generated by the exceptional collection formed by the vector bundles $\Sigma^\alpha\mathcal{U}^\vee$ associated to the corresponding Schur functors, where the Young diagram $\alpha$ has at most $k-1$ rows and whose $p$-th row is of length at most $(n-k)(k-p)/k$, for $p=1,\ldots,k-1$.
Explicitly, in the case $\mathrm{Gr}(2,5)$, the category $\cB$ is generated by $2$ exceptional objects:
\[
\cB = \langle \cO_{\mathrm{Gr}(2,5)}, \cU^\vee \rangle.
\]
In the case of $\mathrm{Gr}(3,10)$, the category $\cB$ is generated by $12$ exceptional objects.
\end{exam}

\begin{exam}[Kapranov]\label{ex:quadrics}
Let $Q$ be an $n$-dimensional quadric in $\P^{n+1}$ defined by an equation $\{ q=0\}$; we assume that $\mathrm{char}(\K)\neq 2$. By \cite{Kap:Grass}, the category $\Db(Q)$ has one of the following semiorthogonal decompositions, given by full strong exceptional collections, according to the parity of $n$. If $n=2m+1$ is odd, we have
\[
\Db(Q)=\langle S,\cO_{Q},\cO_{Q}(1),\ldots,\cO_{Q}(n-1)\rangle,
\]
where $S$ is the spinor bundle on $Q$ defined as $\coker(\phi|_{Q})(-1)$ and $\phi\colon \cO_{\P^{n+1}}(-1)^{2^{m+1}}\to \cO_{\P^{n+1}}^{2^{m+1}}$ is such that  $\phi\circ(\phi(-1))=q\cdot \Id\colon \cO_{\P^{n+1}}(-2)^{2^{m+1}}\to \cO_{\P^{n+1}}^{2^{m+1}}$.

If $n=2m$ is even, we have
\[
\Db(Q)=\langle S^-,S^+,\cO_{Q},\cO_{Q}(1),\ldots,\cO_{Q}(n-1)\rangle,
\]
where $S^-:=\coker(\phi|_{Q})(-1)$, $S^+:=\coker(\psi|_{Q})(-1)$, and $\phi,\psi\colon \cO_{\P^{n+1}}(-1)^{2^{m}}\to \cO_{\P^{n+1}}^{2^{m}}$ are such that  $\phi\circ(\psi(-1))=\psi\circ(\phi(-1))=q\cdot \Id$. The reader can have a look at \cite{Ottaviani} for more results about spinor bundles.
\end{exam}

It is a very interesting question to determine which varieties admit a full exceptional collection. A classical result of Bondal \cite{Bon:Exceptional} shows that a smooth projective variety has a full strong exceptional collection if and only if its bounded derived category is equivalent to the bounded derived category of finitely generated right modules over an associative finite-dimensional algebra.
For example, there is a vast literature on homogeneous spaces (which conjecturally should have a full strong exceptional collection; see \cite{KP:ExceptionalHomogeneous}, and the references therein) or in relation to Dubrovin's Conjecture \cite{Dub:ICM}.
Another example is the moduli space of stable rational curves with $n$ punctures: by using results of Kapranov (\cite{Kap:M0n}) and Orlov (\cite{Orl:projbun}; see also Example \ref{ex:BlowUps} below), it is easy to show that a full exceptional collection exists, but conjecturally there exists a full strong exceptional collection which is $S_n$-invariant (see \cite{MS:M0n,CT:M0n}). We also mention the following question by Orlov relating the existence of a full exceptional collection to rationality; we will study further conjectural relations between non-commutative K3 surfaces and rationality in Section \ref{sec:examples}.

\begin{ques}[Orlov]\label{ques:orlov}
Let $X$ be a smooth projective variety over $\K$.
If $\Db(X)$ admits a full exceptional collection, then $X$ is rational.
\end{ques}

For later use, we want to show that the semiorthogonal decompositions in Examples \ref{ex:Pn} and \ref{ex:quadrics} can be made relative to a positive dimensional base.

\begin{exam}[Orlov]\label{ex:ProjBundles}
Let $F$ be a vector bundle of rank $r+1$ over a smooth projective variety $X$.
Consider the projective bundle $\pi\colon\P_X(F)\to X$.
Then, by \cite{Orl:projbun}, the pull-back functor $\pi^*$ is fully faithful and we have a semiorthogonal decomposition
\[
\Db(\P_X(F))=\langle \pi^*(\Db(X))\otimes\cO_{\P_X(F)/X}(-r),\ldots,\pi^*(\Db(X))\rangle.
\]
\end{exam}

\begin{exam}[Kuznetsov]\label{ex:QuadricFibrations}
Let $X,S$ be smooth projective varieties, and let $f\colon X\to S$ be a (flat) quadric fibration. In other words, there is a vector bundle $E$ on $S$ of rank $r+2$ such that $X$ is a divisor in $\P_S(E)$ of relative degree $2$ corresponding to a line bundle $L\subseteq S^2E^\vee$.
We assume that $\mathrm{char}(\K)\neq 2$.
Such a quadric fibration comes with a sheaf $\cB$ of Clifford algebras on $S$. The corresponding sheaf $\cB_0$ of even parts of Clifford algebras can be described as an $\cO_S$-module in the following terms:
\[
\cB_0\cong\cO_S\oplus(\wedge^2E\otimes L)\oplus(\wedge^4E\otimes L^2)\oplus\ldots
\]
We can then take the abelian category $\coh(S,\cB_0)$ of coherent $\cB_0$-modules on $S$ and the corresponding derived category $\Db(S,\cB_0):=\Db(\coh(S,\cB_0))$. One can also consider the sheaf $\cB_1$ of odd parts of the Clifford algebras, which is a coherent $\cB_0$-module. Actually, for all $i\in\Z$, we have the sheaves
\[
\cB_{2i}:=\cB_0\otimes_{\cB_0}L^{-i}\qquad\cB_{2i+1}:=\cB_1\otimes_{\cB_0}L^{-i}.
\]

The pull-back functor $f^*$ is fully-faithful.
Moreover, there is a fully-faithful functor $\Phi\colon \Db(S,\cB_0)\hra \Db(X)$, and a semiorthogonal decomposition
\[
\Db(X)=\langle\Phi(\Db(S,\cB_0)),f^*(\Db(S)),\ldots,f^*(\Db(S))\otimes\cO_{X/S}(r-1)\rangle.
\]
The functor $\Phi$ has a left adjoint that we denote by $\Psi$. The case when $\mathrm{char}(\K)=0$ was treated in \cite[Theorem 4.2]{Kuz:quadrics}. This was generalized in \cite[Theorem 2.2.1]{ABB:quadrics} to fields of arbitrary odd characteristic. The case $\mathrm{char}(\K)=2$ is discussed in the same paper and it can also be described in a similar way, but the definition of $\cB_0$ is different.
\end{exam}

Finally, the last case we review is the blow-up along smooth subvarieties.

\begin{exam}[Orlov]\label{ex:BlowUps}
Let $Y\subseteq X$ be a smooth subvariety of codimension $c$.
Consider the blow-up $f\colon\widetilde{X}\to X$ of $X$ along $Y$. Let $i\colon E\hookrightarrow\widetilde{X}$ be the exceptional divisor. The restriction $\pi:=f|_{E}\colon E\to Y$ is a projective bundle, as $E=\P_Y(\mathcal{N}_{Y/X})$. By \cite{Orl:projbun}, the pullback functor $f^*$ is fully faithful and, for any integer $j$, the functors
\[
\Psi_j\colon\Db(Y)\to\Db(\widetilde{X})\qquad G\mapsto i_*(\pi^*(G)\otimes\cO_{E/Y}(j)),
\]
are fully-faithful as well, giving the semiorthogonal decomposition
\[
\Db(\widetilde{X})=\langle f^*\Db(X),\Psi_0(\Db(Y)),\ldots,\Psi_{c-2}(\Db(Y))\rangle.
\]
\end{exam}

\subsection{Non-commutative smooth projective varieties}
\label{subsec:NC}

We will work with the following definition of non-commutative smooth projective variety; while this is not the most general notion (see \cite{Orl:SmoothNC,Perry:NCHP,KP:Joins}), it will suffice for these notes.\footnote{Note that in \cite[Definition 4.3]{Orl:SmoothNC} it is used the terminology \emph{geometric noncommutative scheme} for what we call non-commutative smooth projective variety.}

\begin{defi}
Let $\cD$ be a triangulated category linear over $\K$.
We say that $\cD$ is a \emph{non-commutative smooth projective variety} if there exists a smooth projective variety $X$ over $\K$ and a fully faithful $\K$-linear exact functor $\cD\hra\Db(X)$ having left and right adjoints.
\end{defi}

By identifying $\cD$ with its essential image in $\Db(X)$, then the definition is only asking that $\cD$ is an admissible subcategory. Note also that, as a consequence of the main result in \cite{BvdB:NC}, being admissible is a notion which is intrinsic to the category $\cD$, namely every fully faithful functor from a non-commutative smooth projective variety into the bounded derived category of a smooth projective variety will have both adjoints.

\subsubsection*{Products}
Non-commutative smooth projective varieties are closed under products (or more generally gluing of categories; see \cite{KL:gluings,Orl:SmoothNC}). More precisely, if $\cD_1\subset\Db(X_1)$ and $\cD_2\subset\Db(X_2)$, we can define $\cD_1 \boxtimes \cD_2$ as the smallest triangulated subcategory of $\Db(X_1\times X_2)$ which is closed under taking direct summands and contains all objects of the form $F_1\boxtimes F_2$, with $F_1\in\cD_1$ and $F_2\in\cD_2$.\footnote{This is different from the definition which appears in \cite[Equation (10)]{Kuz:BaseChange}, but in our context of smooth projective varieties it is equivalent (see also the beginning of the proof of \cite[Theorem 5.8]{Kuz:BaseChange}).}

\begin{prop}\label{prop:product}
The subcategory $\cD_1 \boxtimes \cD_2 \subset \Db(X_1\times X_2)$ is admissible.
\end{prop}

\begin{proof}
This is a special case of \cite[Theorem 5.8]{Kuz:BaseChange}.
In our context of smooth projective varieties, it becomes quite simple.

In fact, we claim that there is a semiorthogonal decomposition
\[
\Db(X_1\times X_2) = \langle \cD_1\boxtimes\cD_2, {^\perp}\cD_1\boxtimes\cD_2, \cD_1\boxtimes{^\perp}\cD_2, {^\perp}\cD_1\boxtimes{^\perp}\cD_2\rangle.
\]
As remarked in the previous section, this immediately implies that $\cD_1\boxtimes\cD_2$ is admissible, and thus a non-commutative smooth projective variety.

To prove the claim, we need to check conditions \eqref{definition-sod-1} and \eqref{definition-sod-2} of Definition \ref{def:SO}.
Condition \eqref{definition-sod-1} follows immediately from the K\"unneth formula (see, for example, \cite[Section 3.3]{Huy:FM}):
\[
\Hom(F_1\boxtimes G_1,F_2\boxtimes G_2) = \bigoplus_{i\in\Z} \Hom(F_1,F_2[i])\otimes \Hom(G_1,G_2[-i]).
\]

Condition \eqref{definition-sod-2} follows directly from the following elementary but very useful fact.
Let $F\in\Db(X_1\times X_2)$.
Then since $X_1\times X_2$ is projective, we can find a bounded above locally-free resolution $P^\bullet$ of $F$, such that, for all $i$, $P^i=P_1^i\boxtimes P_2^i$.
Since $X_1\times X_2$ is smooth, by truncating this resolution at sufficiently large $n$ (with respect to the stupid truncation), we obtain a split exact triangle
\[
G \to \sigma^{\geq -n}P^\bullet \to F,
\]
for some $G\in\Db(X_1\times X_2)$, namely $F$ is a direct factor of $\sigma^{\geq -n}P^\bullet$.
Since $\cD_1$ and $\cD_2$ are part of semiorthogonal decompositions of respectively $\Db(X_1)$ and $\Db(X_2)$, this concludes the proof.
\end{proof}

Having a notion of product, we can define the slightly technical notion of Fourier-Mukai functor between non-commutative varieties (see \cite{Kuz:HPD,HR:cubics}, and \cite{CS:FMsurvey} for a survey on Fourier-Mukai functors).

\begin{defi}\label{def:FM}
Let $X_1,X_2$ be algebraic varieties.
Let $\cD_1\hra\Db(X_1)$ and $\cD_2\hra\Db(X_2)$ be admissible categories.
A functor $F\colon\cD_1\to\cD_2$ is called a \emph{Fourier-Mukai functor} if the composite functor
\[
\Db(X_1)\xrightarrow{\delta_1}\cD_1\to\cD_2\hra\Db(X_2)
\]
is of Fourier-Mukai type.
\end{defi}

In the previous definition, $\delta_1$ denotes the projection functor of Remark \ref{remark-projection-functors} with respect to the semiorthogonal decomposition $\Db(X)=\langle \cD_1, \cD_2\rangle$; it coincides with the left adjoint to the inclusion functor.
By using Proposition \ref{prop:product}, it is not hard to see that if a functor between non-commutative varieties is of Fourier-Mukai type, then the kernel $P_F\in\Db(X_1\times X_2)$ actually lives in $\cD_1^{\perp\perp}\boxtimes \cD_2$; for example, the proof in \cite[Lemma 1.5]{HR:cubics} works in general.

\begin{rema}
By \cite[Theorem 6.4]{Kuz:BaseChange}, the definition of $\cD_1 \boxtimes \cD_2$ does not depend on the embedding.
More precisely, given a Fourier-Mukai equivalence $\cD_1\isomto\cD_1'$, then there is a Fourier-Mukai equivalence of triangulated categories $\cD_1 \boxtimes \cD_2\cong \cD_1' \boxtimes \cD_2$.
\end{rema}

An immediate corollary of Proposition \ref{prop:product} is that the projection functor $\delta$ is of Fourier-Mukai type.
We will use this to define Hochschild (co)homology for a non-commutative variety.

\begin{lemm}\label{lem:projectionFM}
Let $X$ be a smooth projective variety, and let $\cD\subset\Db(X)$ be an admissible category.
Then the projection functor $\delta\colon\Db(X)\to\cD$ is of Fourier-Mukai type.
\end{lemm}

\begin{proof}
This is a special case of \cite[Theorem 7.1]{Kuz:BaseChange}.
Again, in our smooth projective context, the proof is very simple.
Indeed, by Proposition \ref{prop:product}, the subcategory $\Db(X)\boxtimes\cD\subset \Db(X\times X)$ is admissible.
The kernel of the projection functor is simply the projection of the structure sheaf of the diagonal $\cO_{\Delta}\in\Db(X\times X)$ onto the category $\Db(X)\boxtimes\cD$.
\end{proof}

\begin{rema}
Not all functors $\cD_1\to\cD_2$ are of Fourier-Mukai type (see \cite{RvdB:noFM,Vol:noFM}). Nevertheless, fully faithful functors are expected to be of Fourier-Mukai type (as they are in the commutative case, by Orlov's Representability Theorem \cite{Orl:FM}). This goes under the name of Splitting Conjecture (see \cite[Conjecture 3.7]{Kuz:HPD}).
\end{rema}

\subsubsection*{The numerical Grothendieck group}
Given a triangulated category $\cD$, we denote by $K(\cD)$ its Grothendieck group. If $\cD$ is a non-commutative smooth projective variety, then the Euler characteristic
\[
\chi(F,G):=\sum_i (-1)^i \dim_k \Hom_{\cD}(F,G[i])
\]
is well-defined for all $F,G\in\cD$ and it factors through $K(\cD)$.

\begin{defi}\label{def:NumericalGrothendieck}
Let $\cD$ be a non-commutative smooth projective variety.
We define the \emph{numerical Grothendieck group} as $K_\mathrm{num}(\cD):=K(\cD)/\ker\chi$.
\end{defi}

The numerical Grothendieck group is a free abelian group of finite rank. Indeed, since $\cD$ is an admissible subcategory of $\Db(X)$, for $X$ a smooth projective variety, then $K_\mathrm{num}(\cD)$ is a subgroup of the numerical Grothendieck group $N(X)$ of $X$ which is a free abelian group of finite rank.

\subsubsection*{The Serre functor}
Serre duality for non-commutative varieties is studied via the notion of Serre functor, introduced and studied originally in \cite{BK:dg}.

\begin{defi}\label{def:SerreFunctor}
Let $\cD$ be a triangulated category.
A \emph{Serre functor} in $\cD$ is an autoequivalence $S_{\cD}\colon \cD \to\cD$ with a bi-functorial isomorphism
\[
\Hom(F,G)^\vee = \Hom(G,S_{\cD}(F))
\]
for all $F,G\in\cD$.
\end{defi}

If a Serre functor exists then it is unique up to a canonical isomorphism.
If $X$ is a smooth projective variety, the Serre functor is given by $S_{\Db(X)}(\blank)=\blank\otimes\omega_X [\dim X]$.
In general, Serre functors exist for non-commutative smooth projective varieties as well. If $\cD$ has a Serre functor and $\Phi\colon\cC\hookrightarrow\cD$ is an admissible subcategory, then the Serre functor of $\cC$ is given by the following formula
\[
S_{\cC} = \Psi_R \circ S_{\cD}\circ \Phi\qquad \text{ and }\qquad S_{\cC}^{-1}=\Psi_L\circ S_{\cD}^{-1}\circ \Phi,
\]
where $\Psi_R$ and $\Psi_L$ denote respectively the right and left adjoint to the inclusion functor.
Notice also that the Serre functor is of Fourier-Mukai type.

Serre functors behave nicely with respect to products.
Given two non-commutative smooth projective varieties $\cD_1\subset \Db(X_1)$ and $\cD_2\subset\Db(X_2)$, let us denote by $P_{S_{\cD_1}}\in\Db(X_1\times X_1)$, respectively $P_{S_{\cD_2}}\in\Db(X_2\times X_2)$, kernels representing the Serre functors.
Then the Serre functor of the product $\cD_1\boxtimes\cD_2\subset \Db(X_1\times X_2)$ is representable by $P_{S_{\cD_1}}\boxtimes P_{S_{\cD_2}}$.
This can be proved directly (for example, the argument in \cite[Corollary 1.4]{HR:cubics} works in general), since it is true for products of varieties.


We can now define non-commutative Calabi-Yau smooth projective varieties.

\begin{defi}\label{def:NCCY}
Let $\cD$ be a non-commutative smooth projective variety.
We say that $\cD$ is a non-commutative \emph{Calabi-Yau} variety of dimension $n$ if $S_{\cD}=[n]$.
\end{defi}

By what we observed before, the product of two non-commutative Calabi-Yau varieties of dimension $n$ and $m$ is again Calabi-Yau of dimension $n+m$.

\subsubsection*{Hochschild (co)homology}
The analogue of Hodge cohomology for non-commutative varieties is Hochschild homology. This can be defined naturally in the context of dg-categories.
In our smooth and projective case, this can also be done directly at the level of triangulated categories, in a similar way to \cite{Mar:Hochschild,CW:Hochschild,Cal:Hochschild2} for bounded derived categories of smooth projective varieties.\footnote{Hochschild (co)homology for schemes was originally defined and studied in \cite{Lod:Cyclic,Swa:Hochschild,Wei:Hochschild}.}
The main reference is \cite{Kuz:Hochschild}, and we are content here to briefly sketch the basic properties (see also \cite{Per:Hochschild}).

Let $\cD$ be a non-commutative smooth projective variety.
We also fix an embedding $\cD\hra\Db(X)$, for $X$ a smooth projective variety; the definition will of course be independent of this choice.
We let $P_\delta\in\Db(X\times X)$ be a kernel of the projection functor $\delta$ which is of Fourier-Mukai type, by Lemma \ref{lem:projectionFM}.
We also let $P_{\delta!}\in\Db(X\times X)$ be the kernel of $S_{\cD}^{-1}$, the inverse of the Serre functor of $\cD$.
For $i\in\Z$, we define
\begin{align*}
&\HH^i(\cD):= \Hom_{\Db(X\times X)}(P_\delta,P_\delta[i])& \text{(Hochschild cohomology)}\\
&\HH_i(\cD):= \Hom_{\Db(X\times X)}(P_{\delta!}[i],P_\delta) & \text{(Hochschild homology)}
\end{align*}
Our choice of degree for Hochschild homology is different from \cite{Kuz:Hochschild}.
It is coherent, though, with the definition of Hochschild homology for varieties. Indeed, when $\cD=\Db(X)$, this gives the usual definitions
\begin{align*}
&\HH^i(\Db(X))= \Hom_{\Db(X\times X)}(\cO_{\Delta},\cO_{\Delta}[i]),\\
&\HH_i(\Db(X))= \Hom_{\Db(X\times X)}(S_{\Delta}^{-1}[i],\cO_{\Delta}),
\end{align*}
where $S_{\Delta}^{-1}:=\Delta_*\omega_X^{-1}[-\mathrm{dim}(X)]$.

As in the commutative case, Hochschild cohomology has the structure of a graded algebra, and Hochschild homology is a right module over it.
We can also define a \emph{Mukai pairing}
\[
\left(\blank,\blank\right)\colon \HH_i(\cD)\otimes \HH_{-i}(\cD)\to \K
\]
which is a non-degenerate pairing, induced by Serre duality. Roughly, by reasoning as in \cite[Section 4.9]{Cal:Hochschild1}, one first considers the isomorphism $\tau\colon\Hom(P_{\delta!},P_\delta[-i])\to\Hom(P_\delta,P_\delta\otimes p_2^*\omega_X[n-i])$. The latter vector space is Serre dual to $\HH_{-i}(\cD)$. Thus $\left(v,w\right)$ is defined as the trace of the composition $\tau(v)\circ w$.

The triple consisting of Hochschild (co)homology and the Mukai pairing is called the \emph{Hochschild structure} associated to $\cD$.

\begin{exam}
Let $E$ be an exceptional object and let $\cD=\langle E \rangle$.
Then $\HH_\bullet(\cD)=\HH^\bullet(\cD)=\K$, both concentrated in degree $0$.
\end{exam}

\begin{rema}\label{rmk:HKR}
The \emph{Hochschild--Kostant--Rosenberg isomorphism}
$\Dd_X^*\cO_{\Dd_X}\to\bigoplus_i\Omega^i_X[i]$ yields the graded isomorphisms
\[
I^X_\mathrm{HKR}:\HH_*(X)\rightarrow\HO_*(X):=\bigoplus_i\HO_i(X)\quad\text{and}\quad I_X^\mathrm{HKR}:\HH^*(X)\rightarrow\HT^*(X):=\bigoplus_i\HT^i(X)
\]
where $\HO_i(X):=\bigoplus_{q-p=i}H^p(X,\Omega_X^q)$ and $\HT^i(X):=\bigoplus_{p+q=i}H^p(X,\wedge^q T_X)$.

The existence of such isomorphisms is due to Swan \cite{Swa:Hochschild} for fields of characteristic zero (see also \cite{CW:Hochschild}). In \cite{Yek2} it is proved that if $X$ is a smooth scheme defined over a field of characteristic $p>\dim X$, then the same result holds. In the recent paper \cite{AV:HKR}, this was extended (in a slightly weaker form) to smooth proper schemes over fields of characteristic $p\geq\dim X$.
\end{rema}

A key property of Hochschild homology is that it behaves well with respect to semiorthogonal decompositions.

\begin{prop}\label{prop:HochschildSO}
Let $\cD$ be a non-commutative variety, and let
\[
\cD=\langle \cD_1,\ldots,\cD_r\rangle
\]
be a semiorthogonal decomposition.
Then, for all $i\in\Z$,
\[
\HH_i (\cD) \cong \bigoplus_{j=1}^r \HH_i(\cD_j).
\]
Moreover, this decomposition is orthogonal with respect to the Mukai pairing.
\end{prop}

\begin{proof}
The first statement is \cite[Corollary 7.5]{Kuz:Hochschild}.
Here, for simplicity, we sketch the proof for the case in which $\cD=\Db(X)$ and we have a semiorthogonal decomposition $\Db(X)=\langle \cC, {^\perp}\cC \rangle$.

We observed above that the projection functors $\delta_1$ and $\delta_2$ for the admissible subcategories $\cC$ and ${^\perp}\cC$ respectively are of Fourier-Mukai type with kernels $P_1$ and $P_2$. These objects sit in the following distinguished triangle
\[
P_2\to\cO_\Delta\to P_1.
\]
Similarly, we have a triangle
\[
P_{2!}\to S_\Delta^{-1} \to P_{1!}.
\]
Since, by semiorthogonality, $\Hom(P_{2!}, P_1[j])=0$, for all $j\in\Z$, we have an induced map
\[
\Hom(S_\Delta^{-1}[i],\cO_\Delta)\lra \Hom(P_{1!}[i],P_1)\oplus \Hom(P_{2!}[i],P_2).
\]
By using Serre duality, it is not hard to see that $\Hom(P_{1!}, P_2[j])=0$, for all $j\in\Z$ as well; hence, the above map is an isomorphism.
Finally, by definition, $\HH_i(\cD_1)=\Hom(P_{1!}[i],P_1)$; regarding the second factor, it is a small argument (see \cite[Corollary 3.12]{Kuz:Hochschild}) to see that $\HH_i(\cD_2)\cong \Hom(P_{2!}[i],P_2)$.

For the second statement, one needs to use the previous construction of the morphism, together with the isomorphism $\Hom(P_{\delta!},P_\delta[-i])\cong\Hom(P_\delta,P_\delta\otimes p_2^*\omega_X[n-i])$ mentioned before.
\end{proof}

\begin{exam}
Let $\{E_1,\ldots,E_m\}$ be an exceptional collection, and let $\cD=\langle E_1,\ldots,E_m\rangle$. Then $\HH_\bullet(\cD)=\K^{\oplus m}$, concentrated in degree $0$.
\end{exam}

From Proposition \ref{prop:HochschildSO}, we can deduce all other properties of the Hochschild structure.

\begin{theo}\label{thm:HochschildHomologyProperties}
Let $\cC,\cD,\cE$ be non-commutative smooth projective varieties.
\begin{enumerate}
\item Any Fourier-Mukai functor $\Phi\colon\cC\to\cD$ induces a morphism of graded $k$-vector spaces $\Phi_{\HH}\colon\HH_\bullet(\cC)\to\HH_\bullet(\cD)$ such that $\Id_{\HH}=\Id$ and, given another functor $\Psi\colon\cD\to\cE$, we have $(\Psi\circ\Phi)_{\HH}=\Psi_{\HH}\circ\Phi_{\HH}$.
\item If $(\Psi,\Phi)$ is a pair of adjoint Fourier-Mukai functors, then
\[
\left(\blank,\Phi_{\HH}(\blank)\right) = \left(\Psi_{\HH}(\blank),\blank\right).
\]
\item There is a Chern character $\ch\colon K(\cD)\to \HH_0(\cD)$ such that, for all $F,G\in\cD$,
\[
\left(\ch(F),\ch(G) \right) = - \chi(F,G).
\]
\item\label{enum:Hochschild4} The Hochschild structure is invariant under exact equivalences of Fourier-Mukai type.
\end{enumerate}
\end{theo}

\begin{proof}
As for (1), since $\Phi$ is a Fourier-Mukai functor, it is induced by a Fourier-Mukai functor $\Psi\colon\Db(X_1)\to\Db(X_2)$, where $\cC$ and $\cD$ are admissible subcategories of $\Db(X_1)$ and $\Db(X_2)$ respectively. Let $E$ be the Fourier-Mukai kernel of $\Psi$. It is not difficult to see that $\Psi$ induces a morphism $\Psi_{\HH}\colon\HH_\bullet(X_1)\to\HH_\bullet(X_2)$. For a given $i$ consider $\mu\in\HH_i(X_1)$ and define $\Psi_{\HH}(\mu)\in\HH_i(X_2)$ in the following way. We consider the composition
\[
S^{-1}_{\Delta_{X_2}}[i]\mor[\gamma]E\circ E^\vee[i]\mor[\mathrm{id}\circ\eta\circ\mathrm{id}] E\circ S^{-1}_{\Delta_{X_1}}[i]\circ S_{\Delta_{X_1}}\circ E^\vee\mor[\mathrm{id}\circ\mu\circ\mathrm{id}\circ\mathrm{id}]E\circ S_{\Delta_{X_1}}\circ E^\vee\mor[h]\cO_{\Delta_{X_2}},
\]
where $\eta:\cO_{\Delta_{X_1}}\to S^{-1}_{\Delta_{X_1}}\circ S_{\Delta_{X_1}}$ is the isomorphism coming from the easy fact that $\Phi_{S^{-1}_{\Delta_{X_1}}}\circ\Phi_{S_{\Delta_{X_1}}}=\mathrm{id}$ and
the morphisms $\gamma$ and $h$ are the natural ones.

By Proposition \ref{prop:HochschildSO}, $\HH_\bullet(\cD)$ is an orthogonal factor of $\HH_\bullet(X_2)$. So we set $\Phi_{\HH}$ in (1) to be the composition of $\Psi_{\HH}|_{\HH_\bullet(\cC)}$ with the orthogonal projection onto $\HH_\bullet(\cD)$.

With this definition, (2) follows from a straightforward computation. As for (3), following \cite{Cal:Hochschild1}, let us recall that, given $E\in\cD$, we can think of $E$ as an object on $\mathrm{pt}\times X_2$ and just set $\ch(E):=(\Phi_E)_{\HH}(1)$, where $(\Phi_E)_{\HH}\colon\HH_0(\mathrm{pt})(\cong\K)\to\HH_0(X_2)$. Since $E\in\cD$, we have $\ch(E)\in\HH_0(\cD)$. The compatibility between the Mukai pairing and $\chi$ is proven as in \cite[Theorem 7.6]{Cal:Hochschild1} (the change of sign is harmless here).

For the last statement and more details on the first three, the reader can consult \cite[Section 7]{Kuz:Hochschild}.
\end{proof}

Notice, in particular, that by Theorem \ref{thm:HochschildHomologyProperties},\eqref{enum:Hochschild4}, the Hochschild structure does not depend on the choice of the embedding $\cD\hra\Db(X)$.

For Calabi-Yau varieties, Hochschild homology and cohomology coincide, up to shift.
This is the content of the following result (see \cite[Section 5.2]{Kuz:CY}).
With our definition, the proof is immediate.

\begin{prop}\label{prop:HochschildCY}
Let $\cD$ be a non-commutative Calabi-Yau variety of dimension $n$.
Then, for all $i\in\Z$, we have
\[
\HH^i(\cD) \cong \HH_{n-i}(\cD).
\]
In particular, $\HH_n(\cD)\neq0$.
\end{prop}

\begin{proof}
We only need to show the last statement. But, by definition, $\HH^0(\cD)\neq0$, since $\Id_{P_\delta}\in\Hom(P_\delta,P_\delta)$.
\end{proof}

We can use Hochschild cohomology to define the notion of connectedness for non-commutative varieties.

\begin{defi}\label{def:connected}
Let $\cD$ be a non-commutative smooth projective variety.
We say that $\cD$ is \emph{connected} if $\HH^0(\cD)=\K$.
\end{defi}

\begin{lemm}[Bridgeland's trick]\label{lem:BridgelandTrick}
Let $\cD$ be a non-commutative Calabi-Yau variety of dimension $n$.
If $\cD$ is connected, then $\cD$ is indecomposable, namely it does not admit any non-trivial semiorthogonal decomposition.
\end{lemm}

\begin{proof}
Let $\cD=\langle\cD_1,\cD_2\rangle$ be a non-trivial semiorthogonal decomposition.
Since $\cD$ is a non-commutative Calabi-Yau variety of dimension $n$, then both $\cD_1$ and $\cD_2$ are non-commutative Calabi-Yau varieties of dimension $n$ as well.
In particular, by Proposition \ref{prop:HochschildCY}, $\HH_n(\cD_i)\neq0$, for $i=1,2$. But then, by Proposition \ref{prop:HochschildSO}, we have
\[
\K = \HH^0(\cD) = \HH_n(\cD) = \HH_n(\cD_1)\oplus \HH_n(\cD),
\]
which is a contradiction.
\end{proof}

We use Hochschild (co)homology to finally define non-commutative K3 surfaces. Recall that for a K3 surface $S$, Hochschild (co)homology can be easily computed, by using the Hochschild-Kostant-Rosenberg Theorem:
\[
\HH_\bullet(S) = \K[-2] \oplus \K^{\oplus 22} \oplus \K[2].
\]

\begin{defi}\label{def:NCK3}
Let $\cD$ be a non-commutative smooth projective variety.
We say that $\cD$ is a \emph{non-commutative K3 surface} if $\cD$ is a non-commutative connected Calabi-Yau variety of dimension $2$ and its Hochschild (co)homology coincides with the Hochschild (co)homology of a K3 surface.
\end{defi}

\subsection{Constructing non-commutative Calabi-Yau varieties}
\label{subsec:NCCY}

The goal of this section is to present a result by Kuznetsov which covers essentially all currently known examples of non-commutative Calabi-Yau varieties.
We will then study examples of non-commutative K3 surfaces in details in Section \ref{sec:examples}, and then concentrate on the case of those non-commutative K3 surfaces associated to cubic fourfolds in Section \ref{sec:CubicFourfolds}.

We will need first to introduce a bit more terminology.
First of all, we recall the notion of spherical functor (\cite{Rou:Spherical,Ann:Spherical,AL:Spherical,Add:Spherical,Mea:Spherical}).
We will follow the definition in \cite[Section 2.5]{Kuz:CY}.

Given a functor $\Phi\colon\cC\to\cD$, we keep following the convention of denoting by $\Psi_R$ and $\Psi_L$ respectively its right and left adjoint functors (if they exist).

\begin{defi}\label{def:SphericalFunctor}
Let $X,Y$ be smooth projective varieties and let $\Phi\colon\Db(X)\to\Db(Y)$ be a Fourier-Mukai functor.
We say that $\Phi$ is \emph{spherical} if
\begin{enumerate}
\item the natural transformation $\Psi_L\oplus\Psi_R\to \Psi_R\circ\Phi\circ\Psi_L$, induced by the  sum of the two units of the adjunction, is an isomorphism.
\item the natural transformation $\Psi_L\circ\Phi\circ\Psi_R\to\Psi_L\oplus\Psi_R$, induced by the  sum of the two counits of the adjunction, is an isomorphism.
\end{enumerate}
\end{defi}

Given a spherical functor, we can define the associated \emph{spherical twist} functors:
\begin{align*}
&T_X \colon\Db(X)\to\Db(X),& T_X(F):=\mathrm{cone}\left(\Psi_L\circ\Phi(F)\to F \right)\\
&T_Y\colon\Db(Y)\to\Db(Y), & T_Y(G):=\mathrm{cone}\left(G\to \Phi\circ\Psi_L(G) \right)[-1]
\end{align*}

The key result about spherical functors is the following.

\begin{prop}\label{prop:SphericalTwist}
Let $\Phi\colon\Db(X)\to\Db(Y)$ be a spherical functor.
Then the twist functors $T_X$ and $T_Y$ are autoequivalences.
Moreover, we have
\[
\Phi\circ T_X = T_Y \circ \Phi \circ [2].
\]
\end{prop}

\begin{proof}
With the above definition, this is proved in detail in \cite[Proposition 2.13 \& Corollary 2.17]{Kuz:CY}.
The idea of the proof is not hard: if we define
\begin{align*}
&T_X'\colon\Db(X)\to\Db(X),& T_X'(F):=\mathrm{cone}\left(F\to \Psi_R\circ\Phi(F) \right)[-1]\\
&T_Y'\colon\Db(Y)\to\Db(Y), & T_Y'(G):=\mathrm{cone}\left(\Phi\circ\Psi_R(G)\to G \right),
\end{align*}
we can show with a direct argument that $T_X$ and $T_X'$ (respectively, $T_Y$ and $T_Y'$) are mutually inverse autoequivalences.
The last formula follows then easily from the definitions.
\end{proof}

Secondly, we recall the important notion of rectangular Lefschetz decomposition (which is fundamental for Homological Projective Duality; see \cite{Kuz:HPD,Perry:NCHP}).

\begin{defi}\label{def:RectangularLeftschez}
Let $X$ be a smooth projective variety, and let $L$ be a line bundle on $X$.
A \emph{Lefschetz decomposition} of $\Db(X)$ with respect to $L$ is a semiorthogonal decomposition of the form
\[
\Db(X) = \langle \cB_0, \cB_1\otimes L,\ldots, \cB_{m-1}\otimes L^{\otimes (m-1)}\rangle, \qquad \cB_0\supseteq \cB_1\supseteq \ldots\supseteq\cB_{m-1}.
\]
A Lefschetz decomposition is called \emph{rectangular} if $\cB_0=\cB_1= \ldots=\cB_{m-1}$.
\end{defi}

To simplify notation, given a line bundle $L$ on a smooth projective variety $X$,  we denote by $L\colon\Db(X)\isomto\Db(X)$ the autoequivalence given by tensoring by $L$.

We can now state the main result for this section.
This is our setup.

\begin{setup}\label{setup:KuznetsovCY}
Let $M$ and $X$ be smooth projective varieties and let $L_M$ and $L_X$ be line bundles on $M$ and $X$ respectively.
Let $m,d\in\Z$ be such that $1\leq d <m$. 
We assume:
\begin{enumerate}
\item $\Db(M)$ has a rectangular Lefschetz decomposition with respect to $L_M$
\[
\Db(M) = \langle \cB_M, \cB_M\otimes L_M,\ldots, \cB_M\otimes L_M^{\otimes (m-1)}\rangle,
\]
\item there is a spherical functor $\Phi\colon\Db(X)\to\Db(M)$
\end{enumerate}
 which satisfy the following compatibilities:
\begin{enumerate}
\item[(i)] $L_M\circ \Phi = \Phi \circ L_X$;
\item[(ii)] $L_X\circ T_X = T_X \circ L_X$;
\item[(iii)] $T_M(\cB_M\otimes L_M^{\otimes i}) = \cB_M\otimes L_M^{\otimes i-d}$, for all $i\in\Z$.
\end{enumerate}
\end{setup}

\begin{prop}\label{prop:KuzCYff}
The left adjoint functor $\Psi_L$ induces a fully faithful functor $\cB_M\hra\Db(X)$.
\end{prop}

\begin{proof}
This is a direct check; see \cite[Lemma 3.10]{Kuz:CY} for the details.
\end{proof}

We set $\cB_X:=\Psi_L(\cB_M)$.
By Proposition \ref{prop:KuzCYff}, and by using properties (i) and (iii), we have a semiorthogonal decomposition
\begin{equation}\label{eq:DefKuznetsovComponent}
\Db(X) = \langle \Ku(X), \cB_X, \cB_X\otimes L_X,\ldots, \cB_X\otimes L_X^{\otimes (m-d-1)} \rangle,
\end{equation}
where $\Ku(X)$ is defined as
\[
\Ku(X):=\langle \cB_X,\cB_X\otimes L_X,\ldots, \cB_X\otimes L_X^{\otimes (m-d-1)} \rangle^\perp.
\]

\begin{defi}\label{def:KuznetsovComponent}
We say that $\Ku(X)$ is the \emph{Kuznetsov component} of $X$ associated to our data: $M$, $\Phi$, and the rectangular Lefschetz decomposition of $\Db(M)$.
\end{defi}

Let us define the two autoequivalences:
\begin{align*}
&\rho\colon\Db(X)\isomto\Db(X) &\rho:= T_X\circ L_X^{\otimes d} \\
&\sigma\colon\Db(X)\isomto\Db(X) &\sigma:= S_{\Db(X)}\circ T_X\circ L_X^{\otimes m}
\end{align*}

\begin{lemm}\label{lem:PropertiesRhoSigma}
We keep our assumptions as in Setup \ref{setup:KuznetsovCY}.
We have:
\begin{enumerate}
\item\label{enum:PropertiesRhoSigma1} $\sigma\circ\rho=\rho\circ\sigma$;
\item\label{enum:PropertiesRhoSigma2} $S_{\Db(X)}^{-1}=L_X^{\otimes m}\circ T_X\circ \sigma^{-1}$;
\item\label{enum:PropertiesRhoSigma3} $\sigma$ and $\rho$ respect the semiorthogonal decomposition \eqref{eq:DefKuznetsovComponent}.
\end{enumerate}
\end{lemm}

\begin{proof}
The first statement follows immediately from the property (ii) and the fact that the Serre functor commutes with any autoequivalence.
The second statement follows then immediately.
For the last one, one can check it directly; see \cite[Lemma 3.11]{Kuz:CY} for the details.
\end{proof}

\begin{theo}[Kuznetsov]\label{thm:KuznetsovMain}
Let $c:=\mathrm{gcd}(d,m)$.
The Serre functor of the Kuznetsov component can be expressed as
\[
S_{\Ku(X)}^{d/c}=\rho^{-m/c}\circ \sigma^{d/c}.
\]
\end{theo}

To prove the theorem, we introduce a fundamental functor for the Kuznetsov component, the \emph{degree shift} functor:\footnote{In \cite{Kuz:CY} the functor $O$ is defined on the whole derived category $\Db(X)$ and it is called \emph{rotation functor}.}
\[
O_{\Ku(X)}\colon\Ku(X)\to\Ku(X) \qquad O_{\Ku(X)}:=\delta_{\Ku(X)}\circ L_X,
\]
where $\delta_{\Ku(X)}$ denotes as usual the projection onto the Kuznetsov component (or equivalently, the left adjoint functor of the inclusion).

\begin{lemm}\label{lem:DegreeShiftProperties}
We keep our assumptions as in Setup \ref{setup:KuznetsovCY}.
We have:
\begin{enumerate}
\item\label{enum:DegreeShiftProperties1} $O_{\Ku(X)}$ is an autoequivalence;
\item\label{enum:DegreeShiftProperties2} $O_{\Ku(X)}\circ \rho=\rho \circ O_{\Ku(X)}$ and $O_{\Ku(X)}\circ \sigma= \sigma \circ O_{\Ku(X)}$;
\item\label{enum:DegreeShiftProperties3} $O_{\Ku(X)}^i=\delta_{\Ku(X)}\circ L_X^{\otimes i}$, for all $0\leq i\leq m-d$;
\item\label{enum:DegreeShiftProperties4} $S_{\Ku(X)}^{-1}=O_{\Ku(X)}^{m-d}\circ\rho\circ \sigma^{-1}$;
\item\label{enum:DegreeShiftProperties5} $O_{\Ku(X)}^d=\rho$.
\end{enumerate}
\end{lemm}

\begin{proof}
This is the summary of various results in \cite[Section 3]{Kuz:CY}.
Property \eqref{enum:DegreeShiftProperties1} follows by either \eqref{enum:DegreeShiftProperties4} or \eqref{enum:DegreeShiftProperties5}.
Property \eqref{enum:DegreeShiftProperties2} follows by a direct check.
Property \eqref{enum:DegreeShiftProperties4} follows from \eqref{enum:DegreeShiftProperties3} and by Lemma \ref{lem:PropertiesRhoSigma},\eqref{enum:PropertiesRhoSigma2}.
The key results are \eqref{enum:DegreeShiftProperties3} and \eqref{enum:DegreeShiftProperties5}.

To prove \eqref{enum:DegreeShiftProperties3}, observe that the formula is true for $i=0$. Let us assume the formula is true for $0\leq i< m-d$; we want to show it is true for $i+1$ as well.
Let $F\in\Ku(X)$.
We can consider the composition
\begin{align*}
F\otimes L_X^{\otimes (i+1)} &\to \delta_{\Ku(X)}(F\otimes L_X^{\otimes i})\otimes L_X = O^i_{\Ku(X)}(F)\otimes L_X\to\\
&\to \delta_{\Ku(X)}(O^i_{\Ku(X)}(F)\otimes L_X)=O_{\Ku(X)}^{i+1}(F).
\end{align*}
We need to show that $\delta_{\Ku(X)}(F\otimes L_X^{\otimes (i+1)})=O_{\Ku(X)}^{i+1}(F)$, or equivalently that the cone $G$ of the above composition is in $\langle \cB_X,\cB_X\otimes L_X,\ldots, \cB_X\otimes L_X^{\otimes (m-d-1)} \rangle$.
But, by the octahedral axiom, the cone is an extension of two objects
\[
G_1\otimes L_X \to G \to G_2,
\]
where $G_1\in\langle \cB_X,\cB_X\otimes L_X,\ldots, \cB_X\otimes L_X^{\otimes (i-1)} \rangle$ and $G_2\in\cB_X$, which is what we wanted.

We will only show \eqref{enum:DegreeShiftProperties5} under the assumption $d\leq m-d$. Since we are interested in the Calabi-Yau case (where, in particular, we will have that $c=d$ divides $m$), this is enough for us.
By \eqref{enum:DegreeShiftProperties3}, we have that $O_{\Ku(X)}^d=\delta_{\Ku(X)}\circ L_X^{\otimes d}$.
By definition, we also have that $\rho=T_X\circ L_X^{\otimes d}$.

Let $F\in\Ku(X)$.
Then $\rho(F)$ is defined by the following triangle
\[
\Psi_L(\Phi(F\otimes L_X^{\otimes d})) \to F\otimes L_X^{\otimes d} \to \rho(F).
\]
By Lemma \ref{lem:PropertiesRhoSigma},\eqref{enum:PropertiesRhoSigma3}, we know that $\rho(F)\in\Ku(X)$.
Hence, we only need to show that $\Psi_L(\Phi(F\otimes L_X^{\otimes d}))\in {^\perp}\Ku(X)$.

By using adjointness, it is easy to see that $\Phi(F\otimes L_X^{\otimes d})\in\langle \cB_M,\ldots, \cB_M\otimes L_M^{\otimes d-1}\rangle$.
The adjoint of property (i) shows that, for all $i\in\Z$,
\[
\Psi_L(\cB_M\otimes L_M^{\otimes i})=\Psi_L(\cB_M)\otimes L_X^{\otimes i}=\cB_X\otimes L_X^{\otimes i}.
\]
Hence
\[
\Psi_L(\Phi(F\otimes L_X^{\otimes d}))\in\langle \cB_X,\ldots, \cB_X\otimes L_X^{\otimes d-1}\rangle\subset {^\perp}\Ku(X),
\]
as we wanted.
\end{proof}

We can now prove Kuznetsov's theorem.

\begin{proof}[Proof of Theorem \ref{thm:KuznetsovMain}.]
By Lemma \ref{lem:DegreeShiftProperties},\eqref{enum:DegreeShiftProperties4}, we can express the (inverse of the) Serre functor in terms of the functors $O_{\Ku(X)}$, $\rho$, and $\sigma$.
By Lemma \ref{lem:PropertiesRhoSigma},\eqref{enum:PropertiesRhoSigma1}, all these functors commute.
By raising everything to the power $d/c$, and by using Lemma \ref{lem:DegreeShiftProperties},\eqref{enum:DegreeShiftProperties5}, the statement follows.
\end{proof}

\section{Fano varieties and their Kuznetsov components: examples}
\label{sec:examples}

In this section we present a few examples of non-commutative K3 surfaces.
The basic references are \cite{Kuz:CubicFourfolds,Kuz:survey,KP:GM}.
There are very interesting examples of non-commutative Calabi-Yau varieties in higher dimension as well; we will not cover them in these notes and we refer to \cite{IM:cubics,IM:CY,Kuz:CY} and references therein.

The general goal could be stated as follows.

\begin{ques}\label{ques:existenceK3}
How to construct examples of non-commutative K3 surfaces?
Is there a generalized period map and a (derived) Torelli Theorem?
What is the image of the period map?
\end{ques}

Already the first part of Question \ref{ques:existenceK3} is not easy to answer.
The main issue is that the only way we have to construct non-commutative K3 surfaces is by embedding them in some (commutative) Fano variety.
Currently, there are a few families which have been studied.
We will present three of them in this section: cubic fourfolds, Gushel-Mukai manifolds, and Debarre-Voisin manifolds.
All of them arise indeed as Kuznetsov components in the derived category of a certain smooth Fano variety, as in Theorem \ref{thm:KuznetsovMain}.
Other examples, not covered in these notes, are K\"uchle manifolds \cite{Kuz:Kuchle}; in such examples, though, a rectangular Lefschetz decomposition is not yet known and Theorem \ref{thm:KuznetsovMain} not yet applicable directly.

There is a common theme and expectation that such non-commutative K3 surfaces should contain deep birational properties on the Fano variety itself.
We can then formulate the following question.

\begin{ques}[Kuznetsov]\label{ques:rationality}
Let $W$ be a Fano fourfold.
Assume that $W$ has a Kuznetsov component $\Ku(W)$ which is a non-commutative K3 surface.
If $W$ is rational, then is $\Ku(W)$ equivalent to the derived category of a K3 surface?
\end{ques}

The above question is not well-defined in general, since there is no invariant definition yet of what a Kuznetsov component is for a general fourfold $W$ (see \cite[Section 3]{Kuz:survey}).
On the other hand, in the examples we will see (cubic fourfolds and Gushel-Mukai fourfolds), there is an evident choice for it, and the above question therefore makes sense.

Question \ref{ques:rationality} also motivates the understanding of when such a Kuznetsov component is actually equivalent to the derived category of a K3 surface. We can then formulate the following question, which we are going to answer in the case of cubic fourfolds (see Theorem \ref{thm:AT}).

\begin{ques}[Addington-Thomas, Huybrechts]\label{ques:K3}
Let $W$ be a Fano variety.
Assume that $W$ has a Kuznetsov component $\Ku(W)$ which is a non-commutative K3 surface.
Is it true that $\Ku(W)$ is equivalent to the derived category of a K3 surface if and only if there is a primitive embedding of the hyperbolic lattice $U\hra K_\mathrm{num}(\Ku(W))$ in the numerical Grothendieck group of $W$?
\end{ques}

The hyperbolic lattice $U$ in Question \ref{ques:K3} has rank $2$, is even unimodular and is defined by the bilinear form
\[
\left(\begin{array}{cc}0&1\\1&0\end{array}\right).
\]
It has a very neat interpretation in terms of moduli spaces of objects in $\Ku(W)$.
Indeed, the two square-zero classes correspond to skyscraper sheaves and ideal sheaves of points on the K3 surface. Also, the fact that the two classes have intersection $1$ corresponds to the fact that both are fine moduli spaces.
This is the way we will approach this question in Section \ref{sec:CubicFourfolds}: we will recover the K3 surface as moduli space of stable objects in $\Ku(W)$, and use a universal family to induce the derived equivalence.

In this section we will discuss the above questions in three examples.
Cubic fourfolds will be in Section \ref{subsec:CubicFourfolds}; Gushel-Mukai manifolds in Section \ref{subsec:GM}; Debarre-Voisin manifolds in Section \ref{subsec:DV}. In Section \ref{subsec:Mukai} we review the integral Mukai structure in these examples, by using topological K-theory, as in \cite{AT:cubic}. Finally, in Section \ref{subsec:Torelli} we discuss Torelli-type statements.

Finally, for sake of completeness, we mention that conjecturally all non-commutative smooth projective varieties are expected to be admissible subcategories in a Fano manifold (we refer to \cite{BBF:Determinantal,KKLL:CompleteIntersections,KL:Fano,FK:Curves,Nar:Curves} for recent advances on this conjecture).

\begin{ques}[Bondal]\label{ques:FanoVisitor}
Let $\cD$ be a non-commutative smooth projective variety.
Does there exists a Fano manifold $W$ and a fully faithful functor $\cD\hra\Db(W)$?
\end{ques}

\subsection{Cubic fourfolds}
\label{subsec:CubicFourfolds}

Let $\iota\colon W\hra\P^5$ be a cubic fourfold over $\K$; we assume 
$\mathrm{char}(\K)\neq 2,3$.
We denote by $\cO_W(1)$ the hyperplane section $\cO_{\P^5}(1)|_W$.

\subsubsection*{The Kuznetsov component} We start by using Theorem \ref{thm:KuznetsovMain} to show that the Kuznetsov component of a cubic fourfold is a non-commutative K3 surface.

\begin{lemm}\label{lem:CubicFourfoldSpherical}
The functor $\iota_*\colon\Db(W)\to\Db(\P^5)$ is spherical.
The associated spherical twists are $T_W=\cO_W(-3)[2]$ and $T_{\P^5}=\cO_{\P^5}(-3)$.
\end{lemm}

\begin{proof}
All the statements can be checked by using the exact sequence
\[
0\to \cO_{\P^5}(-3) \to \cO_{\P^5} \to \cO_{W} \to 0
\]
and a direct computation.
\end{proof}

By using Lemma \ref{lem:CubicFourfoldSpherical} and the semiorthogonal decomposition
\[
\Db(\P^5) = \langle \cO_{\P^5},\cO_{\P^5}(1),\cO_{\P^5}(2),\cO_{\P^5}(3),\cO_{\P^5}(4),\cO_{\P^5}(5)\rangle,
\]
it is immediate to check that the compatibilities in Setup \ref{setup:KuznetsovCY} are met ($d=3$ and $m=6$).
Hence, we have a semiorthogonal decomposition
\[
\Db(W) = \langle \Ku(W), \cO_W,\cO_W(1),\cO_W(2) \rangle
\]
and, by Theorem \ref{thm:KuznetsovMain}, the Serre functor $S_{\Ku(W)}=[2]$.
Hence, $\Ku(W)$ is a non-commutative $2$-Calabi-Yau category.

We can also compute Hochschild homology of $W$ directly, by using the Hochschild-Kostant-Rosenberg Theorem, and the Hodge diamond for $W$:
\[
\HH_\bullet(W)= \K[-2] \oplus \K^{\oplus 25} \oplus \K[2].
\]
By Proposition \ref{prop:HochschildSO}, the Hochschild homology of $\Ku(W)$ is therefore isomorphic to the one of a K3 surface.
Therefore, $\Ku(W)$ is an example of a non-commutative K3 surface.

\subsubsection*{Pfaffian cubic fourfolds}

Toward understanding both Question \ref{ques:rationality} and Question \ref{ques:K3} for cubic fourfolds, the first example to analyze in detail is the case of Pfaffian cubic fourfolds.
They are all contained in a special divisor in the moduli space of cubic fourfolds.

Let $V$ be a $\K$-vector space of dimension $6$.
For $i=2,4$, let $\mathrm{Pf}(i,V)$ be the closed subset of $\P(\Lambda^2V)$ consisting of those forms having rank $\leq i$.
Let $L\subset \P(\Lambda^2V)$ be a linear subspace of dimension $8$.
We set
\[
S := \mathrm{Pf}(2,V) \cap L\qquad \text{ and }\qquad W:= \mathrm{Pf}(4,V^\vee) \cap L^\perp.
\]
For a general $L$, both $S$ and $W$ are smooth: $S$ is a K3 surface of degree 14 and $W$ a cubic fourfold.
We call all cubic fourfolds obtained in this way \emph{Pfaffian cubic fourfolds}, and the K3 surface the \emph{associated K3 surface}.

We define the correspondence
\[
\Gamma := \left\{ (s,w)\in S\times W\,:\, s \cap \ker(w)\neq 0 \right\},
\]
with the natural projections $p_S\colon\Gamma\to S$ and $p_W\colon\Gamma\to W$. The above definition makes sense in view of the observation that $\mathrm{Pf}(2,V) =\mathrm{Gr}(2,V)$ and thus the points of $S$ are actually $2$-dimensional subspaces of $V$. Moreover, we think of $w$ as a point of $\P(\Lambda^2V^\vee)$ so that $\ker(w)$ is also a subspace of $V$. We remark here that even though the expected codimension of $\Gamma$ is $3$ a direct computation shows that it is actually $2$.

\begin{prop}[Kuznetsov]\label{prop:KuznetsovPfaffian}
Let $W$ be a Pfaffian cubic fourfold, and let $S$ be the associated K3 surface.
Then the ideal sheaf $\cI_\Gamma$ induces a Fourier-Mukai equivalence
\[
\Phi_{\cI_{\Gamma}\otimes p_W^*\cO_W(1)}\colon\Db(S)\isomto \Ku(W)\subset\Db(W).
\]
\end{prop}

\begin{proof}
This is the content of \cite[Theorem 2]{Kuz:Pfaffians}.
We follow the presentation given in \cite[Proposition 3]{AL:8fold}, and we refer there for all details.

The argument goes as follows, under the additional assumption that $L$ is general (which is enough for our future purposes). In this case, indeed, $S$ does not contain a line and $W$ does not contain a plane.

Consider two distinct points $p_1,p_2\in S$ and set $\Gamma_i:=p_S^{-1}(p_i)$. Note that $\Gamma_i$ is a quartic scroll, for $i=1,2$. Since $p_1\neq p_2$, we have that $\Gamma_1$ and $\Gamma_2$ are distinct. Indeed, if we identify $p_i$ with the subspace it parametrizes, we have $p_1\cap p_2=\{0\}$ because, otherwise, $S$ would contain a line. This implies that if $\Gamma_1=\Gamma_2$, then the maps $\pi_i\colon\Gamma_i\to\P(p_i)$ mapping $w$ to $p_i\cap\ker(w)$ would define two different rulings on $\Gamma_1=\Gamma_2$. This is not possible.

To show that $\Phi:=\Phi_{\cI_{\Gamma}\otimes p_W^*\cO_W(1)}$ is fully faithful, one just applies the standard criterion due to Bondal and Orlov (see, for example, \cite[Proposition 7.1]{Huy:FM}). In particular, for $p_1,p_2\in S$, we have to prove that
\[
\dim\Hom(\Phi(\cO_{p_1}),\Phi(\cO_{p_2})[i])=\dim\Hom(\cO_{p_1},\cO_{p_2}[i]).
\]
A simple computation shows that $\Phi(\cO_{p_i})=\cI_{\Gamma_i}(1)$. Thus, the equality above can be rewritten as 
\[
\dim\Hom(\cI_{\Gamma_1}(1),\cI_{\Gamma_2}(1)[i])=\dim\Hom(\cO_{p_1},\cO_{p_2}[i]).
\]
The equality is clearly trivial when $i<0$. On the other hand, $\Gamma_i$ has codimension $2$ and $\Gamma_1$ and $\Gamma_2$ are distinct if $p_1\neq p_2$. Hence the equality holds for $i=0$ as well. Since the Serre functor of $\Ku(W)$ is the shift by $2$, the same results holds true for $i=2$. Since $\chi(\cI_{\Gamma_1},\cI_{\Gamma_2})=0$, the case $i=1$ follows as well.

This implies that $\Phi$ is an equivalence, since we observed that $\Ku(W)$ is a connected Calabi-Yau category of dimension $2$ and thus cannot have a proper admissible subcategory.
\end{proof}

These cubic fourfolds are rational, as proven in \cite[Proposition 5 ii)]{BD:cubic}. The argument goes as follows. Take $V'$ a general codimension $1$ linear subspace in $V$. The assignement that sends $w\in W$, to $\ker(w)\cap V'$ defines a birational map
\[
W\dashrightarrow\P(V'),
\]
which gives the rationality of $W$.

\subsubsection*{Cubic fourfolds and K3 surfaces}

In Section \ref{sec:CubicFourfolds}, we will develop the theory of moduli spaces for the Kuznetsov component of a cubic fourfold.
This will allow us to give a complete answer to Question \ref{ques:K3} for cubic fourfolds:

\begin{theo}[Addington--Thomas, Bayer--Lahoz--Macr\`i--Nuer--Perry--Stellari]\label{thm:AT}
Let $W$ be a cubic fourfold.
Then $\Ku(W)$ is equivalent to the derived category of a K3 surface if and only if there is a primitive embedding of the hyperbolic lattice $U\hra K_\mathrm{num}(W)$ in the numerical Grothendieck group of $W$.
\end{theo}

At the lattice level, the condition $U\hra K_\mathrm{num}(\Ku(W))$ implies that $W$ is \emph{special}, in the sense of Hassett \cite{Has:special}. Roughly speaking, a cubic fourfold $W$ is \emph{special} if $H^4(W,\Z)\cap H^{2,2}(W)$ contains the class of a surface which is not homologous to the self-intersection $H^2$ of a hyperplane class in $W$. These special cubic fourfolds organize themselves in divisors of Noether-Lefschetz type.

Theorem \ref{thm:AT} was first proved by Addington and Thomas in \cite{AT:cubic} generically on these divisors. The completion of their result is in \cite{BLMNPS:families}.
Both results though rely on Proposition \ref{prop:KuznetsovPfaffian} (or a variant of it, for cubic fourfolds containing a plane; see Remark \ref{rmk:twistedK3} below).

\begin{rema}\label{rmk:twistedK3}
An analogous result can be proved, to characterize cubic fourfolds $W$ for which $\Ku(W)$ is equivalent to $\Db(S,\alpha)$, where $S$ is a K3 surface and $\alpha$ is an element in the Brauer group $\mathrm{Br}(S):=H^2(S,\cO_S^*)_\mathrm{tor}$ of $S$. Given a twisted K3 surface, i.e., a pair $(S,\alpha)$ as above, we can define the abelian category $\coh(S,\alpha)$ of $\alpha$-twisted coherent sheaves on $S$ (see \cite[Chapter 1]{Cal:thesis} for an extensive introduction). We set $\Db(S,\alpha):=\Db(\coh(S,\alpha))$.

Following \cite{Kuz:CubicFourfolds}, it is not difficult to construct examples where $\Ku(W)\cong\Db(S,\alpha)$. Indeed, consider a generic cubic fourfold $W$ containing a plane $P$. Let $P'$ another plane in $\P^5$ which is skew with respect to $P$. Let $\pi_P\colon W\dashrightarrow P'$ be the natural projection from $P$. Given $p\in P'$, the preimage $\pi_P^{-1}(p)$ is the union of $P$ and a quadric $Q_p$. By blowing-up $\widetilde{\pi}_P\colon\widetilde{W}\to P'$, we get a quadric fibration. Given the double nature of $\widetilde{W}$ as a blow-up and as a quadric fibration, one can combine Examples \ref{ex:QuadricFibrations} and \ref{ex:BlowUps} and show that $\Ku(W)\cong\Db(P',\cB_0)$.

Back to the geometric setting and due to the genericity assumption on $W$, the quadric $Q_p$ is singular if and only if $p$ belongs to a smooth sextic $C\subseteq P'$. The double cover $S$ of $P'$ ramified along $C$ is a smooth K3 surface and the quadric fibration provided by $\pi_P$ yields a natural class $\alpha\in\mathrm{Br}(S)$. Moreover $\Db(P',\cB_0)\cong\Db(S,\alpha)$.

The rephrasing of Theorem \ref{thm:AT} in the twisted setting is the following. Let $W$ be a cubic fourfold.
Then $\Ku(W)$ is equivalent to the derived category of a twisted K3 surface $(S,\alpha)$ if and only if there is a primitive vector $\mathbf{v}\in K_\mathrm{num}(\Ku(W))$ such that $\mathbf{v}^2=0$.
This was proved generically on Hassett divisors by Huybrechts in \cite{Huy:cubics}, and the completion is in \cite{BLMNPS:families}.

It should be noted that the condition of having an isotropic vector in $K_\mathrm{num}(\Ku(W))$ mentioned above is equivalent to the condition of having a primitive embedding $U(n)\hra K_\mathrm{num}(\Ku(W))$. This shows the analogy with the untwisted case considered in the theorem above. We conclude by observing that a partial result in the case of cubics containing a plane is in \cite{Mosc:Cubics}.
\end{rema}

\subsection{Gushel-Mukai manifolds}
\label{subsec:GM}

In this section, we assume $\mathrm{char}(\K)=0$.
Gushel-Mukai manifolds were introduced and studied in a series of papers \cite{IM:EPW,DIM:GM,DK1,DK2,DK3}, based on earlier classification results in \cite{Gus:GM,Muk:GM}.

\begin{defi}\label{def:GM}
A \emph{Gushel-Mukai (GM) manifold} is a smooth $n$-dimensional intersection
\[
X := \mathrm{Cone}(\mathrm{Gr}(2,5))\cap \P^{n+4}\cap Q,\qquad 2\leq n\leq 6,
\]
where $\mathrm{Cone}(\mathrm{Gr}(2,5))\subset\P^{10}$ is the cone over the Grassmannian $\mathrm{Gr}(2,5)\subset\P^9$ in its Pl\"ucker embedding, $\P^{n+4}\subset\P^{10}$ is a linear subspace, and $Q\subset\P^{n+4}$ is a quadric hypersurface.
\end{defi}

The geometry of GM manifolds is very rich and it is the subject of much interest recently, also due to their similarity (and connections) with cubic fourfolds.
We will only shortly recall the definition of Kuznetsov component for $n=4,6$, and mention a few results towards Question \ref{ques:rationality} and Question \ref{ques:K3} in these cases, mostly without proofs.
Our main reference is \cite{KP:GM}, and we refer there for all details.

\subsubsection*{The Kuznetsov component}

We assume $n=4,6$.\footnote{If $n=2$, then $X$ is a K3 surface. If $n$ is odd, everything goes through in the same way, but the Kuznetsov component is an Enriques-type category, with $S^2=[4]$.}
In this case $X$ is a Fano manifold.
Since $X$ is smooth, the intersection $\mathrm{cone}(\mathrm{Gr}(2,5))\cap Q$ does not contain the vertex of the cone.
Hence, we can consider the projection from the vertex of the cone in the Grassmannian
\[
f\colon X \to \mathrm{Gr}(2,5),
\]
which is called the \emph{Gushel map}.

There are two possibilities for the Gushel map.
Either $f$ is an embedding and its image is a quadric section of a smooth linear section of $\mathrm{Gr}(2,5)$ (in such a case, we say that the GM manifold is \emph{ordinary}), or $f$ is a double covering onto a smooth linear section of $\mathrm{Gr}(2,5)$, ramified along a quadric section (in such a case, we say that the GM manifold is \emph{special}, and we denote by $\tau$ the involution). In either case, we denote the smooth linear section by $M_X$.

\begin{lemm}\label{lem:LefschetzDecoLinearSectionGrass(2,5)}
Let $\iota\colon M\hra \mathrm{Gr}(2,5)$ be a smooth linear section of dimension $N\geq 3$. Then $M$ has a rectangular Lefschetz decomposition with respect to $\cO_M(1):=\cO_{\mathrm{Gr}(2,5)}(1)|_M$:
\[
\Db(M) = \langle \cB_M,\cB_M(1),\ldots,\cB_M(N-2) \rangle,
\]
where $\cB_M=\{\cO_M,\cU_M^\vee \}$, $\cU_M:=\cU_{\mathrm{Gr}(2,5)}|_M$.
\end{lemm}

\begin{proof}
This is \cite[Theorem 1.2 \& Section 6.1]{Kuz:hyperplane} (see also \cite[Lemma 2.2]{KP:GM}). It can also be obtained, in an indirect way, from Theorem \ref{thm:KuznetsovMain}; we briefly sketch the argument.
Indeed, assume, for simplicity, that $M$ has dimension $5$.
Then, the functor $\iota_*$ is spherical and compatible with the rectangular Lefschetz decomposition of Example \ref{ex:Grass}.
The corresponding Kuznetsov component is Calabi-Yau of dimension $-3$.
By Proposition \ref{prop:HochschildSO} and Proposition \ref{prop:HochschildCY}, this is impossible since the Hochschild homology of $M$ can be computed and it is concentrated in degree $0$.
\end{proof}

\begin{lemm}\label{lem:sphericalGM}
The functor $f_*\colon\Db(X)\to \Db(M_X)$ is spherical.
The associated spherical twists are $T_X=\cO_X(-2)[2]$, $T_{M_X}=\cO_{M_X}(-2)$, if $X$ is ordinary, and $T_X=\tau\circ \cO_X(-1)[1]$, $T_{M_X}=\cO_{M_X}(-1)[-1]$, if $X$ is special.
\end{lemm}

\begin{proof}
This is a direct check; see \cite[Proposition 3.4]{Kuz:CY} for the details.
\end{proof}

By using Lemma \ref{lem:sphericalGM} and the rectangular Lefschetz semiorthogonal decomposition of $M_X$ in Lemma \ref{lem:LefschetzDecoLinearSectionGrass(2,5)}, we can check that the compatibilities of Setup \ref{setup:KuznetsovCY} are met.\footnote{If $n=4$, we have $d=2$ and $m=4$, for $X$ ordinary, and $d=1$ and $m=3$, for $X$ special; if $n=6$, we have $d=1$ and $m=5$.}
Hence, we have a semiorthogonal decomposition
\[
\Db(X) = \langle \Ku(X), \cB_X, \cB_X(1), \ldots, \cB_X(n-3)\rangle,
\]
where $\cB_X:=f^*\cB=\langle \cO_X, \cU_X^\vee\rangle$.
By Theorem \ref{thm:KuznetsovMain}, the Serre functor $S_{\Ku(X)}=[2]$.

The Hochscild homology of $\Ku(X)$ can been computed, again by using Proposition \ref{prop:HochschildSO} and the Hochschild-Kostant-Rosenberg Theorem, since the Hodge diamond of $X$ is known (see \cite{IM:EPW,DK2}; in particular, \cite[Proposition 2.9]{KP:GM}). It coincides with the Hochschild homology of a K3 surface.
Therefore, $\Ku(X)$ is a non-commutative K3 surface.

\subsubsection*{Ordinary GM fourfolds containing a quintic del Pezzo surface}

The analogous result of Proposition \ref{prop:KuznetsovPfaffian} for GM fourfolds is the following (see \cite[Theorem 1.2]{KP:GM}).

\begin{theo}\label{thm:KuzPerryGMvsK3}
Let $X$ be an ordinary GM fourfold containing a quintic del Pezzo surface.
Then there is a K3 surface $S$ such that $\Ku(X)\cong\Db(S)$. 
\end{theo}

The geometric construction of the K3 surface $S$ in Theorem \ref{thm:KuzPerryGMvsK3} is rather concrete. In the language of \cite{KP:GM}, $S$ is a \emph{generalized dual} of the Gushel-Mukai fourfold $X$. We do not need to be explicit here about this. But it is worth mentioning that $S$ is a Gushel-Mukai surface (i.e., $n=2$ in Definition \ref{def:GM}).

With a view toward Question \ref{ques:rationality}, GM fourfolds as in Theorem \ref{thm:KuzPerryGMvsK3} are actually rational (see \cite[Lemma 4.7]{KP:GM}): very roughly, by blowing up a quintic del Pezzo surface, we get a fibration over $\P^2$ whose general fiber is a smooth quintic del Pezzo surface. Since, by a theorem of Enriques, Manin, and Swinnerton-Dyer, a quintic Del Pezzo surface defined over an infinite field $k$ is $k$-rational \cite{SB:RationalityQuintic}, this shows rationality over $\P^2$, and so rationality of $X$.

\subsubsection*{Gushel-Mukai manifolds and K3 surfaces}

In \cite{KP:GM} there are many interesting conjectures on the Kuznetsov components of GM manifolds, in particular related to duality. In the very recent preprint \cite{KP:Joins}, the generalized duality conjecture \cite[Conjecture 3.7]{KP:GM} has been completely solved (see \cite[Corollary 9.21]{KP:Joins}).
This gives an analogue of Theorem \ref{thm:KuzPerryGMvsK3} for GM sixfolds.
On the other hand, we still do not know even a generic answer to Question \ref{ques:K3} for GM fourfolds.

\subsection{Debarre-Voisin manifolds}
\label{subsec:DV}

Also in this section, we assume $\mathrm{char}(\K)=0$.
Debarre-Voisin manifolds were studied in \cite{DV:Examples} with the aim of constructing new examples of locally complete families of polarized hyperk\"ahler fourfolds.
Their derived categories are less studied: indeed much less is known with respect to the two previous examples, and all basic questions are still open.

Let $\iota \colon X\hra \mathrm{Gr}(3,10)$ be a smooth linear section.
It is a Fano manifold of dimension $20$.
We denote by $\cO_X(1)$ the restriction to $X$ of the Pl\"ucker line bundle $\cO_{\mathrm{Gr}(3,10)}(1)$.
A computation similar to what we saw before, gives:

\begin{lemm}\label{lem:DV_examples}
The functor $\iota_*\colon\Db(X)\to\Db(\mathrm{Gr}(3,10))$ is spherical.
The associated spherical twists are $T_X=\cO_X(-1)[2]$ and $T_{\mathrm{Gr}(3,10)}=\cO_{\mathrm{Gr}(3,10)}(-1)$.
\end{lemm}

By using Fonarev's rectangular Lefschetz decomposition of Example \ref{ex:Grass}, the compatibilities of Setup \ref{setup:KuznetsovCY} are met ($d=1$ and $m=10$), and we obtain a semiorthogonal decomposition
\[
\Db(X) = \langle \Ku(X), \cB_X, \ldots, \cB_X(8) \rangle,
\]
where $\cB_X:=\iota^*\cB_{\mathrm{Gr}(3,10)}$ has a strong full exceptional collection of length 12.
By Theorem \ref{thm:KuznetsovMain}, the category $\Ku(X)$ is $2$-Calabi-Yau.
The Hodge numbers of $X$ have been computed in \cite[Theorem 1.1]{DV:Examples}; by using the Hochschild-Kostant-Rosenberg Theorem again, we have
\[
\HH_\bullet (X) = \K[-2] \oplus \K^{130} \oplus \K[2],
\]
and so $\Ku(X)$ is an example of non-commutative K3 surface.

There is no Debarre-Voisin manifold where Question \ref{ques:K3} have been answered yet.

\subsection{The Mukai lattice of the Kuznetsov component}
\label{subsec:Mukai}

We assume throughout this section that the base field is the complex numbers, $\K=\C$.
We introduce a lattice structure in Hochschild homology in the non-commutative K3 surface examples discussed in the previous sections.
This, together with the Hodge structure, corresponds to the usual Mukai structure for (derived categories of) K3 surfaces; we refer to \cite[Chapter 10]{Huy:FM} for a summary of results on K3 surfaces.

\subsubsection*{Topological K-theory}

Let $\cD$ be a non-commutative smooth projective variety.
A general construction of the \emph{topological K-theory} associated to $\cD$ is in \cite{Bla:TopKth}. In our setting, this can be introduced in a way more closely related to the usual K-theory of a complex manifold as follows.

\begin{setup}\label{setup:TopKth}
Let $X$ be a smooth projective variety over $\C$.
We assume that:
\begin{itemize}
\item $H^*(X,\Z)$ is torsion-free and $H^{\mathrm{odd}}(X,\Z)=0$;
\item there is a semiorthogonal decomposition
\[
\Db(X) = \langle \cD_X, E_1, \ldots, E_m\rangle
\]
with $\{E_1,\ldots,E_m\}$ an exceptional collection.
\end{itemize}
\end{setup}

We consider the topological K-theory $K_\mathrm{top}(X)$ of $X$.
In our setup, since cohomology is torsion-free and odd cohomology vanishes, $K_\mathrm{top}(X)=K_\mathrm{top}^0(X)$.
This is defined as the Grothendieck group of topological $\C$-vector bundles on $X$.

The basic properties of topological K-theory in our setup are the following (see \cite{AH:Ktop,AH:Ktop2}; see also \cite[Section 2]{AT:cubic}):
\begin{enumerate}
\item $K_\mathrm{top}(\mathrm{pt})=\Z$.
\item  Any morphism $f\colon X\to Y$ induces a pull-back morphism $f^*\colon K_\mathrm{top}(Y)\to K_\mathrm{top}(X)$ and push-forward morphism $f_*\colon K_\mathrm{top}(X)\to K_\mathrm{top}(Y)$. There is a projection
formula and a Grothendieck-Riemann-Roch formula. One can also take tensor products and duals of classes in topological K-theory.
\item The \emph{Mukai vector}
\[
\vv\colon K_\mathrm{top}(X) \to H^*(X,\Q) \qquad \vv:=\ch.\sqrt{\mathrm{td}_X}
\]
is injective and induces an isomorphism over $\Q$. In particular, $K_\mathrm{top}(X)$ is torsion-free.
\item The \emph{Mukai pairing} $(\blank,\blank)$ on $K_\mathrm{top}(X)$ is defined as follows. Pick a map $p\colon X\to\mathrm{pt}$ to a point and define the topological Euler pairing as
\[
\chi(v_1,v_2):=p_*(v_1^\vee\otimes v_2)\in K_\mathrm{top}(\mathrm{pt}),
\]
for all $v_1,v_2\in K_\mathrm{top}(X)$. Notice that, by (1), $\chi(v_1,v_2)$ is an integer. We can now set $(\blank,\blank):=-\chi(\blank,\blank)$.
\item Consider the following modification of the Hochschild-Konstant-Rosemberg isomorphism introduced in Remark \ref{rmk:HKR}: 
\[
I^X_K:=(\mathrm{td}(X)^{-1/2}\lrcorner(-))\circ I^X_\mathrm{HKR},
\]
where $\mathrm{td}(X)^{-1/2}\lrcorner(-)$ denotes the contraction by $\mathrm{td}(X)^{-1/2}$. We can then take the following sequence of morphisms
\[
K_\mathrm{top}(X)\hookrightarrow K_\mathrm{top}(X)\otimes\C\stackrel{\vv}{\longrightarrow}H^*(X,\C)\stackrel{(I^X_K)^{-1}}{\longrightarrow}\HH_*(X),
\]
where $\vv$ denotes here the $\C$-linear extension of the Mukai vector. The composition above is compatible with the various Mukai pairings defined on topological K-theory, singular cohomology and Hochschild homology. Indeed $\vv$ preserves the Mukai pairing by Grothendieck-Riemann-Roch for complex vector bundles while $I^X_K$ does the same by \cite{Ra:RR}.
\end{enumerate}

Now let $X_1$ and $X_2$ be smooth projective varieties over $\C$ and let $P\in\Db(X_1\times X_2)$. Consider the Fourier-Mukai functor $\Phi_P(\blank):=(p_2)_*(P\otimes p_1^*(\blank))\colon\Db(X_1)\to\Db(X_2)$. Since, by (2), pull-back, push-forward and tensorization induce compatible morphisms $(\Phi_P)_K$ and $(\Phi_P)_H$ at the level of (topological) K-theory and singular cohomology (with $\Q$ coefficients), we can consider the following diagram:
\begin{equation}\label{eqn:bigdiagram}
\xymatrix{
\Db(X_1)\ar[d]^-{\Phi_P}\ar[r]^-{[\blank]}&K(X_1)\ar@{^{(}->}[r]\ar[d]^-{(\Phi_P)_K}&K_\mathrm{top}(X_1)\ar[r]^{\vv}\ar[d]^-{(\Phi_P)_K}&H^*(X_1,\Q)\ar[rr]^{(I^{X_1}_K)^{-1}}\ar[d]^-{(\Phi_P)_H}&&\HH_*(X_1)\ar[d]^-{(\Phi_P)_{\HH}}\\
\Db(X_2)\ar[r]^-{[\blank]}&K(X_2)\ar@{^{(}->}[r]&K_\mathrm{top}(X_2)\ar[r]^{\vv}&H^*(X_2,\Q)\ar[rr]^{(I^{X_2}_K)^{-1}}&&\HH_*(X_2).
}
\end{equation}

\begin{lemm}\label{lem:bigsquare}
All squares in \eqref{eqn:bigdiagram} are commutative.
\end{lemm}

\begin{proof}
The commutativity of the first three squares on the left follows directly  by the definition of the induced morphisms. The commutativity of the rightmost square is proved in \cite[Theorem 1.2]{MS:IDTT}.
\end{proof}

The induced morphism $(\Phi_P)_H$ is compatible with the Hodge structure on $H^*(X_i,\Q)$ induced by the isomorphisms $I^{X_i}_K$ between the total cohomology groups and the Hochschild homologies $\HH_*(X_i)$.

\begin{defi}\label{def:TopKth}
Assume we are in Setup \ref{setup:TopKth}.
We define the \emph{topological K-theory} of $\cD_X$ as
\[
K_\mathrm{top}(\cD_X) := \left\{ u\in K_\mathrm{top}(X)\,:\, ([E_i],u)=0, \text{ for all } i=1,\ldots,m\right\}.
\]
\end{defi}

Now let $X_1$ and $X_2$ be as in Setup \ref{setup:TopKth} and let $\Phi\colon\cD_{X_1}\to\cD_{X_2}$ be a Fourier-Mukai functor. In this setting, it is not hard to show that we can rewrite \eqref{eqn:bigdiagram} in the following way:
\begin{equation*}\label{eqn:bigdiagram1}
\xymatrix{
\cD_{X_1}\ar[r]\ar[d]^-{\Phi}&K(\cD_{X_1})\ar@{^{(}->}[r]\ar[d]^-{\Phi_K}&K_\mathrm{top}(\cD_{X_1})\ar[r]\ar[d]^-{\Phi_K}&\HH_*(\cD_{X_1})\ar[d]^-{\Phi_{\HH}}\\
\cD_{X_2}\ar[r]&K(\cD_{X_2})\ar@{^{(}->}[r]&K_\mathrm{top}(\cD_{X_2})\ar[r]&\HH_*(\cD_{X_2}).
}
\end{equation*}
The Mukai and Hodge structures in the above diagram are compatible as well.

Finally, the Mukai structure is invariant under deformations.
More precisely, let $C$ be a smooth quasi-projective curve over $\C$, and let $g\colon \cX\to C$ be a smooth projective morphism.
We let $\cE_1,\ldots,\cE_m\in\Db(\cX)$ be families of exceptional objects and we assume we have a $C$-linear semiorthogonal decomposition\footnote{In our smooth setting, $C$-linearity simply means that each semiorthogonal factor is closed under tensorization by pull-backs of objects from $\Db(C)$.}
\[
\Db(\cX) = \langle\cD_{\cX}, \cE_1\otimes \Db(C), \ldots, \cE_m\otimes \Db(C)\rangle.
\]
By \cite{Kuz:BaseChange}, for each closed point $c\in C$, we have a semiorthogonal decomposition
\[
\Db(\cX_c) = \langle\cD_{\cX_c}, \cE_1|_{c}, \ldots, \cE_m|_{c}\rangle.
\]
We assume that each closed fiber $\cX_c$ and the above semiorthogonal decomposition are as in Setup \ref{setup:TopKth}.
Then, since topological K-theory is invariant by smooth deformations, we have the following result.

\begin{lemm}\label{lem:TopKthandK3s}
In the above notation and assumptions, we have that the topological K-theory $K_{\mathrm{top}}(\cD_{\cX_c})$ and its Mukai structure are invariant as $c\in C$ varies. In particular, if there exists $c_0\in C$ and a smooth projective K3 surface $S$ such that $\cD_{\cX_{c_0}}\cong \Db(S)$, we have $K_{\mathrm{top}}(\cD_{\cX_c})\cong \widetilde{\Lambda}:=E_8(-1)^{\oplus 2}\oplus U^{\oplus 4}$ as lattice, for all $c\in C$.
\end{lemm}

\subsubsection*{Examples}

Let $W$ be a cubic fourfold defined over $\C$. Since $W$ can be deformed to a Pfaffian cubic fourfold $W'$ and we proved in Proposition \ref{prop:KuznetsovPfaffian} that $\Ku(W')\cong\Db(S)$, for $S$ a smooth projective K3 surface, Lemma \ref{lem:TopKthandK3s} implies that $K_{\mathrm{top}}(\Ku(W))$, endowed with the Mukai pairing, is isometric to the K3 lattice $\widetilde{\Lambda}$.

The lattice $K_{\mathrm{top}}(\Ku(W))$ has a weight-$2$ Hodge structure coming from Hochschild homology; explicitly, it can be defined in terms of the weight-$4$ Hodge structure on $H^4(W,\Z)$:
\begin{align*}
&\widetilde{H}^{2,0}(\Ku(W)):=\vv^{-1}(H^{3,1}(W))\\
&\widetilde{H}^{1,1}(\Ku(W)):=\vv^{-1}\left(\bigoplus_{p=0}^4H^{p,p}(W)\right)\\
&\widetilde{H}^{0,2}(\Ku(W)):=\vv^{-1}(H^{1,3}(W)).
\end{align*}

The lattice together with the Hodge structure is called the \emph{Mukai lattice} of $\Ku(W)$ and denoted by $\widetilde{H}(\Ku(W),\Z)$. We set
\begin{align*}
&\widetilde{H}_\mathrm{Hodge}(\Ku(W),\Z):=\widetilde{H}(\Ku(W),\Z)\cap\widetilde{H}^{1,1}(\Ku(W))\\
&\widetilde{H}_\mathrm{alg}(\Ku(W),\Z):=K_\mathrm{num}(\Ku(W)).
\end{align*}
We will see later in Theorem \ref{thm:YoshiokaMain} that we have $\widetilde{H}_\mathrm{Hodge}(\Ku(W),\Z)=\widetilde{H}_\mathrm{alg}(\Ku(W),\Z)$.
This will imply the integral Hodge conjecture holds for cubic fourfolds (see Proposition \ref{prop:intHodge}).

\begin{exam}\label{ex:twistedK3}
Let $(S,\alpha)$ be a twisted K3 surface. Then the total cohomology $H^*(S,\Z)$ is endowed with a Mukai pairing and a weight-$2$ Hodge structure which depends on a lift to $H^2(S,\Q)$ of $\alpha$ (see, for example, \cite{HS}). This lattice with this Hodge structure is called the Mukai lattice and it is denoted by $\widetilde{H}(S,\alpha,\Z)$. When $\alpha$ is trivial we simply write $\widetilde{H}(S,\Z)$ and we have that $\widetilde{H}^{2,0}(S)=H^{2,0}(S)$, $\widetilde{H}^{0,2}(S)=H^{0,2}(S)$ while $\widetilde{H}^{1,1}(S)=H^{0}(S,\C)\oplus H^{1,1}(S)\oplus H^4(S,\C)$. If $\Ku(W)\cong\Db(S,\alpha)$, then the two integral Hodge structures coincide.
\end{exam}

\begin{rema}\label{rmk:A2}
Consider the projection functor $\delta\colon\Db(W)\to\Ku(W)$ and fix any line $L\subset W$ (as we will see later, there is always a $4$-dimensional family of lines in a cubic fourfold). Define
\[
\llambda_1:=\vv(\delta(\cO_L(1))\qquad\llambda_2:=\vv(\delta(\cO_L(2)).
\]
This vectors generate a primitive positive definite sublattice
\[
A_2 =\begin{pmatrix} 2 & -1 \\ -1 & 2 \end{pmatrix} \subset \widetilde{H}_\mathrm{Hodge}(\Ku(W),\Z)
\]
This embedding of $A_2$ moves in families. We should think of this primitive sublattice as the choice of a lattice polarization on $\Ku(W)$: we will return to this when defining Bridgeland stability conditions on $\Ku(W)$.
By \cite[Proposition 2.3]{AT:cubic} we have a Hodge isometry
\begin{equation}\label{eqn:perpend}
\langle \llambda_1,\llambda_2 \rangle^{\perp} \cong H^4_\mathrm{prim}(W,\Z)(-1),
\end{equation}
where $H^4_\mathrm{prim}(W,\Z)$ is the orthogonal complement of the self-intersection of a hyperplane class. By the Local Torelli theorem (see \cite[Section 6.3.2]{Voi:book2}), it follows that the very general cubic fourfold $W$ has the property that $A_2 =\widetilde{H}_\mathrm{Hodge}(\Ku(W),\Z)$.
\end{rema}

\begin{rema}\label{rmk:Nikulin}
A direct computation shows that the discriminant group of $A_2$ is the cyclic group $\Z/3\Z$. Thus \cite[Theorem 1.6.1 and Corollary 1.5.2]{Ni} implies that any autoisometry of $A_2$ extends to an autoisometry of $\widetilde{H}(\Ku(W),\Z)$. Viceversa, the same results from \cite{Ni} yield that any autoisometry of the orthogonal $A_2^\perp$ in $\widetilde{H}(\Ku(W),\Z)$ extends to an autoisometry of the Mukai lattice.
\end{rema}

\begin{exam}\label{ex:O}
For a cubic fourfold $W$, we can consider the autoequivalence $O_{\Ku(W)}$. Its action on $\widetilde{H}(\Ku(W),\Z)$ was investigated in \cite[Proposition 3.12]{Huy:cubics}. In particular, $(O_{\Ku(W)})_H$ fixes the sublattice $A_2$ and cyclicly permutes the elements $\llambda_1$, $\llambda_2$ and $-\llambda_1-\llambda_2$. On the other hand, $(O_{\Ku(W)})_H$ acts as the identity on $H^4_\mathrm{prim}(W,\Z)$.
\end{exam}

As we observed in Example \ref{ex:twistedK3}, an equivalence $\Db(S)\cong\Ku(W)$ induces a Hodge isometry $\widetilde{H}(S,\Z)\cong\widetilde{H}(\Ku(W),\Z)$.
If such a Hodge isometry exists, we say that $W$ \emph{has a Hodge theoretically associated K3 surface}. Notice that if $W$ has an Hodge theoretically associated K3 surface $S$, then the copy of the hyperbolic lattice generated by $H^0(S,\Z)$ and $H^4(S,\Z)$ embeds primitively in $K_\mathrm{num}(\Ku(W))\subseteq\widetilde{H}_\mathrm{Hodge}(\Ku(W),\Z)$. A simple application of \cite[Theorem 1.6.1 and Corollary 1.5.2]{Ni} shows that the converse is also true. Namely, if there is a primitive embedding $U\hookrightarrow K_\mathrm{num}(\Ku(W))$, then $W$ has a Hodge theoretically associated K3 surfaces. All in all, these two conditions are equivalent.

If $X$ is a Gushel-Mukai manifold, Theorem \ref{thm:KuzPerryGMvsK3} (for fourfolds; in the sixfolds case, this is \cite[Corollary 9.21]{KP:Joins}) implies that $X$ can be deformed to a Gushel-Mukai manifold $X'$ such that $\Ku(X')\cong\Db(S)$, for a $S$ a K3 surface. Hence the discussion above can be repeated for Gushel-Mukai manifolds. For Debarre-Voisin manifolds, this is not yet known even though it is expected to hold true as well.
Moreover, in the case of Gushel-Mukai fourfolds, the concept of Hodge theoretically associated K3 surfaces is slightly different though from the case of cubic fourfolds. Indeed, these are special in the sense of \cite{DIM:GM}, and clearly $\widetilde{H}(S,\Z)\cong\widetilde{H}(\Ku(X),\Z)$ is still equivalent to the condition $U\hra K_{\mathrm{num}}(\Ku(X))$. But, contrary to the cubic fourfolds case, having a Hodge theoretically associated K3 surface may not be divisorial for GM fourfolds, as observed in \cite[Section 3.3]{pertusi:GMK3}.

\subsection{Derived Torelli theorem}
\label{subsec:Torelli}

Let us go back to the case of a twisted K3 surface $(S,\alpha)$. The total cohomology $H^*(S,\R)$ comes with an orientation provided by the four positive (with respect to the Muaki pairing) vectors
\[
v_1:=(0,\omega, 0)\qquad v_2:=\left(1,0,-\frac{\omega^2}{2}\right)\qquad v_3:=\mathrm{Re}\,\psi\qquad v_4:=\mathrm{Im}\,\psi,
\]
where $\omega$ is a positive real multiple of an ample line bundle and $\psi$ is a generator of $H^{2,0}(S)$.

\begin{rema}\label{rmk:staborient}
In the language of stability contitions that will be introduced in Section \ref{sec:Bridgeland}, the choce of $v_1$ and $v_2$ correspond to the choice of a stability condition in a connected component of the space of stability conditions of $\Db(S)$. Note that this space is expected to be connected.
\end{rema}

The following result was first proved by Orlov in his seminal paper \cite{Orl:FM} (with the addition of \cite{HMS:Orient}) and in \cite{HS,HS1,Rei:OrientTwist} for twisted K3 surfaces:

\begin{theo}[Derived Torelli theorem for K3 surfaces]\label{thm:derTorelli}
Let $(S_1,\alpha_1)$ and $(S_2,\alpha_2)$ be twisted K3 surfaces. Then the following are equivalent:
\begin{enumerate}
\item There exists an equivalence $\Db(S_1,\alpha_1)\cong\Db(S_2,\alpha_2)$;
\item There exists an orientation preserving Hodge isometry $\widetilde{H}(S_1,\alpha_1,\Z)\cong\widetilde{H}(S_2,\alpha_2,\Z)$.
\end{enumerate}
\end{theo}

To formulate a similar result in the context of non-commutative K3 surfaces arising in one of the three classes of examples discussed above, we need to specify an orientation on $\widetilde{H}(\Ku(X),\Z)$, where $X$ is either a cubic fourfold or a Gushel-Mukai manifold or a Debarre-Voisin manifold.

Inspired by Remark \ref{rmk:staborient} we could proceed as follows. Assume that $\Ku(X)$ has a stability condition $\sigma$. Then $\widetilde{H}(\Ku(X),\Z)$ contains four positive directions spanned by the real and imaginary part of the central charge $Z$ of $\sigma$ and the real and imaginary part of a generator of $\widetilde{H}^{2,0}(\Ku(X))$. We can then formulate the following natural question:

\begin{ques}[Huybrechts]\label{ques:derTorelli}
Let $X_1$ and $X_2$ be either cubic fourfold or Gushel-Mukai manifold or Debarre-Voisin manifold. Is it true that there exists a Fourier-Mukai equivalence $\Ku(X_1)\cong\Ku(X_2)$ if and only if there is an orientation preserving Hodge isometry $\widetilde{H}(\Ku(X_1),\Z)\cong\widetilde{H}(\Ku(X_2),\Z)$?
\end{ques}

A positive answer would give a non-commutative version of Theorem \ref{thm:derTorelli}. This would lead us to explore the relations between the existence of equivalences between the Kuznetsov components and the birational type of the fourfolds:

\begin{ques}[Huybrechts]\label{ques:birtype}
	Let $X_1$ and $X_2$ be fourfolds as above. Is it true that the existence of a Fourier-Mukai equivalence $\Ku(X_1)\cong\Ku(X_2)$ implies that $X_1$ and $X_2$ are birational?
\end{ques}

In the case of a cubic fourfold $W$, the above discussion about orientation can be made very precise, since $\widetilde{H}_\mathrm{alg}(\Ku(W),\Z)$ always contains a copy of the positivive definite lattice $A_2$. Together with the real and imaginary part of a generator of $\widetilde{H}^{2,0}(\Ku(W))$ this lattice provides a natural orientation on the Mukai lattice of $\Ku(W)$.

\begin{rema}\label{rmk:orient}
It was observed in \cite[Lemma 2.3]{Huy:cubics} that the Mukai lattice $\widetilde{H}(\Ku(W),\Z)$ is always endowed with an orientation reversing Hodge isometry. A way to construct this is by taking the isometry of $A_2$ such that $\llambda_1\mapsto-\llambda_1$ while $\llambda_2\mapsto\llambda_1+\llambda_2$. By Remark \ref{rmk:Nikulin}, this extends to a Hodge isometry of $\widetilde{H}(\Ku(W),\Z)$ which changes the orientation. Another more geometric way of describing such an orientation reversing Hodge isometry is by taking the action induced on the Mukai lattice by the autoequivalence of $\Ku(W)$ obtained by taking the dualizing functor $\mathbb{D}(\blank)=\mathbf{R}\cH om(\blank,\cO_W)$ (post)composed with the tensorization by $\cO_W(1)$.
\end{rema}

For cubic fourfolds, Question \ref{ques:derTorelli} has the following (partial) answer which is a slightly more precise version of items (i) and (ii) of \cite[Theorem 1.5]{Huy:cubics}.

\begin{theo}[Non-commutative Derived Torelli]\label{thm:derNVCorelli}
	Let $W_1$ and $W_2$ be cubic fourfolds such that either $W_1$ is very general or $\widetilde{H}_\mathrm{Hodge}(\Ku(W_1),\Z)$ contains a primitive vector $\vv$ with $\vv^2=0$. Then the following are equivalent:
	\begin{enumerate}
		\item There exists a Fourier-Mukai equivalence $\Ku(W_1)\cong\Ku(W_2)$;
		\item There exists a Hodge isometry $\widetilde{H}(\Ku(W_1),\Z)\cong\widetilde{H}(\Ku(W_2),\Z)$ which is orientation preserving.
	\end{enumerate}
\end{theo}

\begin{proof}
Condition (1) implies (2) in full generality, without the assumptions mentioned in the statement. Indeed, we observed in the previous section that a Fourier-Mukai equivalence $\Ku(W_1)\cong\Ku(W_2)$ induces an Hodge isometry $\widetilde{H}(\Ku(W_1),\Z)\cong\widetilde{H}(\Ku(W_2),\Z)$. In case it is not orientation preserving, we can apply Remark \ref{rmk:orient}.

Assume (2). If $W_1$ is very general, then $\widetilde{H}_\mathrm{Hodge}(\Ku(W_1),\Z)=A_2$, and so the same holds for $W_2$. By equation \eqref{eqn:perpend} the lattices $H^4_\mathrm{prim}(W_i,\Z)$ are the orthogonal complements of $A_2$. Thus, we get a Hodge isometry $H^4_\mathrm{prim}(W_1,\Z)\cong H^4_\mathrm{prim}(W_2,\Z)$. The Torelli theorem for cubic fourfolds mentioned in the introduction and reproved later in the last section (see Theorem \ref{thm:Torelli}), implies then that $W_1\cong W_2$. Hence, in particular, there is a Fourier-Mukai equivalence $\Ku(W_1)\cong\Ku(W_2)$.

If $\widetilde{H}_\mathrm{Hodge}(\Ku(W),\Z)$ contains a primitive vector $\vv$ with $\vv^2=0$, then, by Theorem \ref*{thm:AT} and Remark \ref{rmk:twistedK3}, there is a Fourier-Mukai equivalence $\Ku(W)\cong\Db(S,\alpha)$, for some twisted K3 surface $(S,\alpha)$. The fact that (2) implies (1) is then an easy application of Theorem \ref{thm:derTorelli}, as the orienatation preserving Hodge isometry $\widetilde{H}(\Ku(W_1),\Z)\cong\widetilde{H}(\Ku(W_2),\Z)$ yields an orientation preserving Hodge isometry $\widetilde{H}(S_1,\alpha_1,\Z)\cong\widetilde{H}(S_2,\alpha_2,\Z)$.
\end{proof}

Huybrechts' result \cite[Theorem 1.5]{Huy:cubics} has an additional part proving that the equivalence between (1) and (2) in Theorem \ref{thm:derNVCorelli} holds also for general points in the divisors of the moduli space $\cC$ of cubic fourfolds parameterizing special cubic fourfolds.

\begin{rema}
The tight analogy between Kuznetsov components of cubic fourfolds and K3 surfaces, suggests that the number of isomorphisms classes of cubic fourfolds with equivalent Kuznetsov components (with the equivalence given by a Fourier-Mukai functor) should be finite. Indeed, for K3 surfaces such a finitness result is due to Bridgeland and Maciocia \cite[Corollary 1.2]{BM} (and extended in \cite[Corollary 4.6]{HS} to twisted K3 surfaces). For Kuznetsov components, the same statement is proved in \cite[Theorem 1.1]{Huy:cubics}. Notice that the number of isomorphism classes of cubic fourfolds with equivalent Kuznetsov components can be arbitrarily large \cite{Pert:FM}. The same holds for (twisted) K3 surfaces \cite{Og:FM,HS,St:FM}.
\end{rema}

In the presence of well defined period maps, we could wonder if for two manifolds $X_1$ and $X_2$ which are either cubic or Gushel-Mukai or Debarre-Voisin, the following two conditions are equivalent:
\begin{enumerate}
	\item[(1)] There exists a Fourier-Mukai equivalence $\Ku(X_1)\cong\Ku(X_2)$ commuting with $O_{\Ku(X_1)}$ and $O_{\Ku(X_2)}$;
	\item[(2)] $X_1$ and $X_2$ are points of the same fibre of the period map.
\end{enumerate}
If $X_1$ and $X_2$ are cubic fourfolds, this is true and reduces once more to the Torelli Theorem. This will be explained in Theorem \ref{thm:CategoricalTorelli}.

\section{Bridgeland stability conditions}
\label{sec:Bridgeland}

In this section we give a short review on the theory of Bridgeland stability conditions, with a particular emphasis on the case of non-commutative K3 surfaces. The main references are still the original works \cite{Bri:stability,Bri:K3,KS:stability}.
There are also lecture notes on the subject; see, for example, \cite{Bay:stability,Huy:stability,MS:stability}.

\subsection{Definition and Bridgeland's Deformation Theorem}
\label{subsec:BridgelandMain}

Let $\cD$ be a non-commutative smooth projective variety.
The first ingredient in Bridgeland stability condition is the notion of bounded t-structures. We will actually only define what a heart of a bounded t-structure is; in view of \cite[Lemma 3.2]{Bri:stability}, this definition uniquely determines the bounded t-structure.

\begin{defi}\label{def:heart}
	A \emph{heart of a bounded t-structure} in $\cD$ is a full subcategory $\cA \subset \cD$ such that 
	\begin{enumerate}
		\item[(a)] for $E, F \in \cA$ and $n < 0$ we have $\Hom(E, F[n]) = 0$, and
		\item[(b)] for every $E \in \cD$ there exists a sequence of morphisms
		\[ 0 = E_0 \xrightarrow{\phi_1} E_1 \to \dots \xrightarrow{\phi_m} E_m = E \]
		such that the cone of $\phi_i$ is of the form $A_i[k_i]$ for some sequence
		$k_1 > k_2 > \dots > k_m$ of integers and objects $A_i \in \cA$.
	\end{enumerate}
\end{defi}

If $\cD=\Db(X)$, where $X$ is a smooth projective variety, then $\coh(X)\subset\Db(X)$ satisfies the axioms above and thus it is the heart of a bounded t-structure. Other more elaborate ways of constructing these subcategories are discussed in Section \ref{subsec:Tilt}. The heart of a bounded t-structure is always an abelian category.

\begin{defi}\label{def:Bridgeland}
Let $\Lambda$ be a finite rank free abelian group and let $v\colon K(\cD) \thra \Lambda$ be a surjective group homomorphism.
A \emph{Bridgeland stability condition} on $\cD$ (with respect to the pair $(\Lambda,v)$) is a pair $\sigma= (\cA, Z)$ consisting of the heart of a bounded t-structure $\cA\subset\cD$ and a group homomorphism $Z\colon \Lambda\to\C$ (called \emph{central charge}) such that:

(a) For every $0 \not= A \in \cA$, $Z(A)$\footnote{We abuse notation and denote $Z(v(A))$ by $Z(A)$. We use the identifications $K(\cA)=K(\cD)$.} lies in the extended upper-half plane, i.e., $\Im (Z(A)) \geq 0$ and if $\Im(Z(A))=0$, then $\mathrm{Re} (Z(A)) < 0$ (we say that $Z$ is a \emph{stability function}).

(b) The function $Z$ allows one to define a \emph{slope} by setting $\mu_{\sigma} := - \frac{\mathrm{Re} Z}{\Im Z}$ and a notion of stability:
An object $0 \not= E\in\cA$ is $\sigma$-\emph{semistable} if for every proper subobject $F$, we have $\mu_{\sigma}(F) \leq \mu_{\sigma}(E)$.
We then require any object $A$ of $\cA$ to have a Harder-Narasimhan filtration (HN filtration, for short) in semistable ones. This means that there is a finite sequence of monomorphisms in $\cA$
\[
0=E_0\hookrightarrow E_1\hookrightarrow\cdots\hookrightarrow E_{n-1}\hookrightarrow E_n=A
\]
such that the factors $F_j=E_j/E_{j-1}$ are $\mu_{\sigma}$-semistable and
\[
\mu_{\sigma}(F_1)>\mu_{\sigma}(F_2)\cdots>\mu_{\sigma}(F_n).
\]

(c) Finally, $\sigma$ satisfies the \emph{support property}:
There exists a quadratic form $Q$ on $\Lambda\otimes\R$ such that $Q|_{\Ker Z}$ is negative definite, and $Q(E)\geq0$, for all $\sigma$-semistable objects $E\in\cA$.
\end{defi}

For an object $E\in\cA$, the semistable objects in the filtration in Definition \ref{def:Bridgeland},(b) are called Harder-Narasimhan factors (HN factors, for short).
The definition of semistable object can be extended to objects in $\cD$: we say that an object $F\in\cD$ is $\sigma$-semistable if $F=E[n]$, for $E\in\cA$ $\sigma$-semistable and $n\in\Z$.
We can also define the \emph{phase} of a semistable object as follows: for $E\in\cA$ $\sigma$-semistable, we set $\phi(E):=\frac{1}{\pi}\mathrm{arg}(Z(E))\in(0,1]$ and $\phi(E[n])=\phi(E)+n$. The subcategory given by the union of $\sigma$-semistable objects in $\cD$ gives a \emph{slicing} of $\cD$ (this is important, but we will not need this explicitly in these notes; see \cite[Section 3]{Bri:stability}).
Finally, Harder-Narasimhan filtrations can be defined for any non-zero object $E\in\cD$, by combining the two filtrations in Definition \ref{def:Bridgeland},(b) and in Definition \ref{def:heart},(b).
We set $\phi_\sigma^+(E)$, respectively $\phi_\sigma^-(E)$, as the largest, respectively smallest, phase of the Harder-Narasimhan factors of $E$.

We denote by $\mathrm{Stab}_{(\Lambda,v)}(\cD)$ the set of stability conditions on $\cD$ as in the above definition. To simplify the notation, we often write $\mathrm{Stab}_{\Lambda}(\cD)$ or $\mathrm{Stab}(\cD)$ when $\Lambda$ and/or $v$ are clear. This set is actually a topological space: the topology is given by the coarsest one such that, for any $E \in\cD$, the maps
\[
\mathcal{Z}\colon(\cA,Z) \mapsto Z,
\quad (\cA, Z) \mapsto \phi^+(E),
\quad (\cA, Z) \mapsto \phi^-(E)
\]
are continuous. More explicitly this topology is induced by the generalized (i.e., with values in $[0,+\infty]$) metric
\[
d(\sigma_1,\sigma_2) = \underset{0\neq E \in \cD}{\sup}\left\{|\phi^+_{\sigma_1}(E)-\phi^+_{\sigma_2}(E)|,\, | \phi^-_{\sigma_1}(E) - \phi^-_{\sigma_2}(E)|,\, \|Z_1-Z_2\| \right\},
\]
for $\sigma_1,\sigma_2\in\mathrm{Stab}_{\Lambda}(\cD)$.
Here $\|\blank \|$ denotes the induced operator norm on $\Hom(\Lambda,\C)$, with respect to the choice of any norm in $\Lambda$.

The key result in the theory of stability conditions is Bridgeland Deformation Theorem. This is the main result of \cite{Bri:stability}.

\begin{theo}[Bridgeland]\label{thm:Bridgeland}
The continuous map $\mathcal{Z}$ is a local homemorphism and thus the topological space $\mathrm{Stab}_\Lambda(\cD)$ has a natural structure of a complex manifold of dimension $\rk(\Lambda)$.
\end{theo}

By acting on the central charge by a linear transformation, we have an action of $\C$ (or more generally of the universal cover of $\mathrm{GL}_2^+(\R)$) on $\mathrm{Stab}(\cD)$.
By using this action, the proof of Theorem \ref{thm:Bridgeland} reduces to study deformations of the central charge $Z$ with $\Im(Z)$ constant; in this case, the heart is also constant. Then the result follows from an elementary convex geometry argument. This is explained in full detail in \cite{Bay:short}.
The role of the support property and an effective deformation statement is discussed also in \cite[Appendix A]{BMS}.

If we fix $\mathbf{v}\in\Lambda$, then by using the support property there is a locally-finite set of \emph{walls} (real codimension one submanifolds with boundary) in $\mathrm{Stab}_\Lambda(\cD)$ where the set of semistable objects with class $\mathbf{v}$ changes.

\begin{defi}
Let $\mathbf{v}\in\Lambda$.
A stability condition $\sigma$ is called \emph{generic} with respect to $\mathbf{v}$ (or \emph{$\vv$-generic}) if it does not lie on a wall for $\vv$.
\end{defi}

The main open problem in the theory is the lack of examples.
While the surface case is now well-understood \cite{Bri:K3,BM11:local_p2,AB:Bridgeland,Tod13:extremal}, starting from threefolds the theory becomes quite scarce.
In fact, for a long time no example of a stability condition was known for a projective Calabi-Yau threefold.
The first Calabi-Yau examples were finally produced in \cite{MP1,MP2,BMS}; the case of quintic threefolds has been recently addressed in \cite{li:quintic}.

A conjectural approach to construct stability condition is via the notion of tilt-stability \cite{BMT:3folds-BG,BMS}. This is a weak stability condition, and we recall here the definition and the main example: we will use this to construct Bridgeland stability conditions on non-commutative K3 surfaces.

\begin{defi}
A \emph{weak stability condition} on $\cD$ is a pair $\sigma= (\cA, Z)$ consisting of the heart of a bounded t-structure $\cA\subset\cD$ and a group homomorphism $Z\colon \Lambda\to\C$ satisfying (b), (c) in Definition \ref{def:Bridgeland} and such that
\begin{enumerate}
\item[(a')] for $E\in\cA$, we have $\mathrm{Im} Z(E)\geq 0$, with $\mathrm{Im} Z(E)= 0 \Rightarrow \mathrm{Re}Z(E)\leq 0$.
\end{enumerate}
\end{defi}

\begin{exam}\label{ex:coh}
Let $X$ be a smooth projective variety of dimension $n$ and with an ample class $H$. Consider the lattice $\Lambda_H^1$ generated by the vectors
\[
(H^n\rk(E),H^{n-1}\cdot\mathrm{ch}_{1}(E))\in\Q^{\oplus 2},
\]
for all $E\in\coh(X)$. Set $Z_\mathrm{slope}(E):=-H^{n-1}\cdot\mathrm{ch}_{1}(E)+\mathfrak{i}H^n\rk(E)$. It is easy to verify that the pair $\sigma_\mathrm{slope}:=(\coh(X),Z_\mathrm{slope})$ is a weak stability condition. Note that since $\Lambda_H^1$ has rank $2$, the support property trivially holds.

The slope associated to $\sigma_\mathrm{slope}$ is the classical slope stability for sheaves. Hence, by the Bogomolov-Gieseker inequality, we have
\begin{equation} \label{eq:defDeltaH}
\Delta_H(E) = \left(H^{n-1}\ch_1(E)\right)^2 - 2\, \left(H^n \ch_0(E)\right) \, \left(H^{n-2}\ch_2(E)\right)\geq 0, 
\end{equation}
for all $\sigma_\mathrm{slope}$-semistable sheaves $E$.
\end{exam}

\subsection{Bridgeland's Covering Theorem}
\label{subsec:BridgelandCovering}

Let $\cD$ be a non-commutative K3 surface such that $K_\mathrm{num}(\cD)$ is finitely generated. Set $\Lambda=K_\mathrm{num}(\cD)$. Consider the natural surjection $\vv\colon K(\cD)\twoheadrightarrow\Lambda$ and the Mukai pairing $(\blank,\blank)$ given by
\[
(\vv(E),\vv(F))=-\chi(E,F).
\]

\begin{exam}\label{ex:K3Cubic}
If $S$ is a K3 surface, then we have an identification $\Lambda=\widetilde{H}_\mathrm{alg}(S,\Z)$, and $\vv$ is the Mukai vector.
Similarly, if $\cD=\Ku(W)$, for $W$ a cubic fourfold, then $\Lambda=\widetilde{H}_\mathrm{alg}(\Ku(W),\Z)$, and $\vv$ is as well the Mukai vector.
\end{exam}

Under our assumptions, the pairing $(\blank,\blank)$ yields a natural identification between the vector spaces $\Hom(\Lambda,\C)$ and $\Lambda\otimes\C$. In other words, the continuous maps $\mathcal{Z}\colon\mathrm{Stab}_\Lambda(\cD)\to\Hom(\Lambda,\C)$ can be rewritten as a continuous map
\[
\eta\colon\mathrm{Stab}_\Lambda(\cD)\to\Lambda\otimes\C,
\]
such that, for all $\sigma=(\cA,Z)\in\mathrm{Stab}_\Lambda(\cD)$, we have $Z(\blank)=(\eta(\sigma),\blank)$.

Following \cite{Bri:K3}, we define $\cP\subset \Lambda\otimes\C$ 
as the open subset consisting of those vectors whose real and imaginary parts span positive-definite two-planes in $\Lambda\otimes\R$. Set
\[
\cP_0 := \cP \setminus \bigcup_{\delta\in\Delta} \delta^\perp,
\]
where $\Delta:=\{\delta\in\Lambda:(\delta,\delta)=-2\}$.

\begin{theo}[Bridgeland's Covering]\label{thm:cov}
If $\eta^{-1}(\cP_0)$ is non-empty, then the restriction 
\[
\eta\colon\eta^{-1}(\cP_0)\to \cP_0
\]
is a covering map.
\end{theo}

The proof is an application of Theorem \ref{thm:Bridgeland}. The actual statement is \cite[Corollary 1.3]{Bay:short} (based on \cite[Proposition 8.3]{Bri:K3}).

Notice that $\cP_0$ has two connected components. It is expected that the image $\mathrm{im}(\eta)$ is contained in only one connected component of $\cP_0$ and that $\eta^{-1}(\cP_0)$ is connected and simply-connected as well. This is known only for generic analytic K3 surfaces (and some generic twisted K3 surfaces) \cite{HMS:generic}; in the algebraic case, the strongest evidence is \cite{BB:autoequivalences}. This is related to the choice of an orientation, as discussed in Remark \ref{rmk:staborient}.

\begin{exam}\label{ex:K3}
Let $S$ be a K3 surface.
We let $\cP_0^+$ be the connected component of $\cP_0$ containing vectors of the form $(1,0,-\frac{\omega^2}{2})+ \mathfrak{i} (0,\omega, 0)$, where $\omega$ is a positive real multiple of an ample line bundle.
The main result of \cite{Bri:K3} shows that there exist stability conditions on $\Db(S)$ for which skyscraper sheaves are all stable of the same phase. They are all contained in the same connected component, denoted by $\mathrm{Stab}^\dagger(\Db(S))$. Moreover, $\mathrm{Stab}^\dagger(\Db(S))\subset \eta^{-1}(\cP_0^+)$.
The proof is by using tilting of coherent sheaves: we will recall this construction in the next section.
\end{exam}

\subsection{Tilting and examples}
\label{subsec:Tilt}

The aim of this section is to describe a way to produce (weak) stability condition by an iteration process based on tilting.

Let $\cD$ be a non-commutative smooth projective variety and assume that we are given a weak stability condition $\sigma = (\cA, Z)$ on $\cD$. Let $\mu \in \R$. Consider the subcategories of $\cA$ defined as follows:
\begin{align*}
\cT_\sigma^\mu & = \{E:\text{All HN factors $F$ of $E$ have slope
		$\mu_{\sigma}(F) > \mu $}\} \\
& = \langle E \colon \text{$E$ is $\sigma$-semistable with $\mu_{\sigma}(E) > \mu$} \rangle, \\
\cF_\sigma^\mu & = \{E:\text{All HN factors $F$ of $E$ have slope
		$\mu_{\sigma}(F) \le \mu $}\} \\
& = \langle E \colon \text{$E$ is $\sigma$-semistable with $\mu_{\sigma}(E) \leq \mu$} \rangle,
\end{align*}
where $\langle\blank\rangle$ denotes the extension closure. This notation will be used only in this section where there is no risk to confuse this with semiorthogonal decompositions.

The general theory of torsion pairs and tilting allows us to produce a new abelian category in $\cD$ which is the heart of a bounded t-structure:

\begin{prop}\label{prop:slopetilt}
	Given a weak stability condition $\sigma = (Z, \cA)$ and a choice of slope $\mu \in \R$, the category
	\[
	\cA_\sigma^\mu = \langle \cT_\sigma^\mu, \cF_\sigma^\mu[1] \rangle
    \]
is the heart of a bounded t-structure on $\cD$.
\end{prop}

The proof is a direct check; see \cite{Happel-al:tilting}.
We will refer to $\cA^\mu_\sigma$ as the \emph{tilting} of $\cA$ with respect to the weak stability condition $\sigma$ at the slope $\mu$. When $\sigma$ is clear, we will sometimes just write $\cA^\mu$.

Let us now consider the case $\cD=\Db(X)$, where $X$ is a smooth projective variety, and fix a hyperplane section $H$ on $X$. By Example \ref{ex:coh}, we have the weak stability condition $\sigma_\mathrm{slope}$ and, given $\beta \in \R$, we can consider the tilt
\[
\coh^\beta(X):=(\coh_{\sigma_\mathrm{slope}}(X))^\beta\subseteq \Db(X).
\]

We want now to go further and define a new weak stability condition whose heart is $\coh^\beta(X)$. To this extent, for $E\in\coh(X)$, set
\[
\mathrm{ch}^{\beta}(E) = e^{-\beta H} \mathrm{ch}(E)\in H^*(X,\R).
\]

We take $\Lambda_H^2$ to be the lattice generated by the vectors
\[
(H^n\rk(E),H^{n-1}\cdot\mathrm{ch}_{1}(E),H^{n-2}\cdot\mathrm{ch}_{2}(E))\in\Q^{\oplus 3},
\]
for all $E\in\coh(X)$. The classical Bogomolov inequality \eqref{eq:defDeltaH} defines a quadratic inequality on $\Lambda_H^2$, and it is the key ingredient in the following result:

\begin{prop}\label{prop:tiltstab}
Given $\alpha > 0, \beta \in \R$, the pair $\sigma_{\alpha, \beta} = (\coh^\beta(X), Z_{\alpha, \beta})$
	with $\coh^\beta(X)$ as constructed above, and
	\[
	Z_{\alpha, \beta}(E) := \frac 12 \alpha^2 H^n \mathrm{ch}_0^\beta(E) - H^{n-2}\mathrm{ch}_2^\beta(E) + \mathfrak{i}\, H^{n-1}\mathrm{ch}_1^\beta(E)
	\]
	defines a weak stability condition on $\Db(X)$ with respect to $\Lambda_H^2$. The quadratic form $Q$ can be given by $\Delta_H$. Moreover, these stability conditions vary continuously as $(\alpha, \beta) \in \R_{>0} \times \R$ varies.
    Finally, if $\mathrm{dim}(X)=2$, then $\sigma_{\alpha,\beta}$ is a Bridgeland stability condition on $\Db(X)$.
\end{prop}

Proposition \ref{prop:tiltstab} was observed in the case of surfaces in \cite{Bri:K3,AB:Bridgeland}. The higher dimensional version is in \cite{BMT:3folds-BG,BMS}. When the choice of $\alpha$ and $\beta$ are irrelevant, we will refer to the weak stability condition $\sigma_{\alpha, \beta}$ as \emph{tilt stability}.

In the applications to cubic fourfolds, we will need to tilt $\coh^\beta(X)$ once more, where $X$ will be some non-commutative Fano threefold.
To this end, the only thing we will need is to define the analogue of the action of $\C$ on the central charge for tilt-stability, as for Bridgeland stability.

Consider the weak stability condition $\sigma_{\alpha, \beta}$ on $\Db(X)$ as in Proposition \ref{prop:tiltstab}. Let $\mu \in \R$. 
By Proposition \ref{prop:slopetilt}, we get that
\[ 
\coh_{\alpha, \beta}^\mu(X):= (\coh^\beta(X))_{\sigma_{\alpha, \beta}}^\mu.
\]
is the heart of a bounded t-structure. Let $u \in \C$ be the unit vector in the upper half plane with $\mu = -\frac{\Re(u)}{\Im(u)}$ and consider the function 
\[
Z_{\alpha, \beta}^\mu := \frac 1u Z_{\alpha, \beta}.
\]
Then we have:

\begin{prop}\label{prop:tiltedtiltstability}
	The pair $(\coh_{\alpha, \beta}^\mu(X), Z_{\alpha, \beta}^\mu)$ is a weak stability condition on $\Db(X)$.
\end{prop}

Proposition \ref{prop:tiltedtiltstability} was observed implicitly in \cite{BMT:3folds-BG}; the above statement is \cite[Proposition 2.14]{BLMS:inducing}.

\subsection{Inducing stability conditions}
\label{subsec:inducing}

In this last section, we discuss how stability conditions combine with semiorthogonal decompositions. We follow \cite[Sections 4 \& 5]{BLMS:inducing}.

Let $\cD$ be a non-commutative smooth projective variety.
Let $E_1, \dots, E_m$ be an exceptional collection; set $\cD_2:= \langle E_1, \dots, E_m \rangle$ and $\cD_1=\cD_2^\perp$, so that we have a semiorthogonal decomposition
\[
\cD = \langle \cD_1, \cD_2 \rangle.
\]

\begin{prop} \label{prop:inducestability}
	Let $\sigma = (\cA, Z)$ be a weak stability condition on $\cD$ with the following properties:
	\begin{enumerate}
		\item $E_i \in \cA$,
		\item $S_\cD(E_i) \in \cA[1]$, and
		\item $Z(E_i) \neq 0$,
	\end{enumerate}
for all $i = 1, \dots, m$. Assume moreover that there are no objects $0 \neq F \in \cA_1:=\cA \cap \cD_1$ with
	$Z(F) = 0$ (i.e., $Z_1:=Z|_{K(\cA_1)}$ is a stability function on $\cA_1$).
	Then the pair $\sigma_1 = (\cA_1, Z_1)$ is a stability condition on $\cD_1$.
\end{prop}

This is \cite[Proposition 5.1]{BLMS:inducing}.
In what follows we sketch why $\cA$, under the above assumptions, induces the heart $\cA_1$ of a bounded t-structure on $\cD_1$.

\begin{lemm}\label{lem:induce}
Let $\cA \subset \cD$ be the heart of a bounded t-structure.
Assume that $E_1,\dots,E_m\in\cA$ and $\Hom(E_i,F[p])=0$, for all $F\in\cA$, $i=1,\dots,m$, and $p>1$.
Then $\cA_1 := \cD_1 \cap \cA$ is the heart of a bounded t-structure on $\cD_1$. 
\end{lemm}

\begin{proof}
The category $\cA_1$ satisfies condition (a) in Definition \ref{def:heart}, since it holds for $\cA$; hence, we only need to verify (b).

Consider $F \in \cD_1$. For every $i=1,\dots,m$, there is a spectral sequence (\cite[(3.1.3.4)]{BBD}; see also \cite[Proposition 2.4]{Okada:CY2})
\[
E_2^{p, q} = \Hom(E_i, H^q_\cA(F)[p]) \Rightarrow \Hom(E_i, F[p+q]).
\]
    
By assumption, these terms vanish except for $p = 0, 1$, and thus the spectral sequence degenerates at $E_2$.
On the other hand, since $F \in \cD_1=\cD_2^\perp$ we have $\Hom(E_i,F[p+q]) = 0$. Therefore, $\Hom(E_i, H^q_\cA(F)[p]) = 0$, for all $p \in \Z$.
This gives that $H^q_\cA(F) \in \cA \cap \cD_1 = \cA_1$, and so it proves the claim.
\end{proof}

\begin{proof}[Sketch of the proof of Proposition \ref{prop:inducestability}.]
To induce the heart $\cA_1$ we only need to verify the assumptions of Lemma \ref{lem:induce}.
For all $i=1,\dots,m$, $p>1$, and $F\in\cA$, we have
\[
\Hom(E_i,F[p]) = \Hom(F[p],S_{\cD}(E_i))^\vee = \Hom(F,S_{\cD}(E_i)[-p])^\vee=0,
\]
since $S_{\cD}(E_i)\in\cA[1]$.
\end{proof}

\section{Cubic fourfolds}
\label{sec:CubicFourfolds}

The goal of this section is to present the main results in \cite{BLMS:inducing,BLMNPS:families}.
First of all, for a cubic fourfold $W$, we show the existence of Bridgeland stability conditions on $\Ku(W)$ (Theorem \ref{thm:main1}) and describe a connected component of the space of stability conditions (Theorem \ref{thm:stabconn}).
Then we study moduli spaces of stable objects and we generalize the Mukai theory to the non-commutative setting; the main result is Theorem \ref{thm:YoshiokaMain}.
Finally, in Section \ref{subsec:applications}, we explain how to give a uniform setting to study the various interesting hyperk\"ahler manifolds associated to a cubic fourfolds (e.g., the Fano variety of lines or the hyperk\"ahler manifolds constructed by using twisted cubics). As consequences, we also get new proofs of the Torelli theorem (C.1) and the integral Hodge conjecture for cubic fourfolds.

The basic idea to construct stability conditions is to induce stability conditions as in Section \ref{subsec:inducing}.
The problem is that we are currently not able to do this directly from the derived category of the cubic fourfold.
Instead, we use its structure of conic fibration (see Section \ref{subsec:ConicFibrations}), to reduce to a (non-commutative) projective space, where this inducing procedure works.
The key technical result is a generalization of the Bogomolov-Gieseker inequality (Theorem \ref{thm:Bogomolov}) to this non-commutative setting.

\subsection{Conic fibrations}
\label{subsec:ConicFibrations}

In this section we assume that $\mathrm{char}(\K)\neq2$.
Let $W$ be a cubic fourfold and let $L\subseteq W$ be a line not contained in a plane in $W$\footnote{Note that such a line always exists as the family of lines in a smooth cubic fourfold are $4$-dimensional by \cite{BD:cubic}. On the other hand, such an hypersurface can contain only a finite number of planes.}.
We consider the projection
\[
\pi_L\colon W\dashrightarrow\P^3
\]
from $L$ to a skew $3$-dimensional projective space in $\P^5$. Let $\sigma\colon\widetilde{W}:=\mathrm{Bl}_L(W)\to W$ be the blow-up of $W$ along $L$. The rational map $\pi_L$ yields a conic fibration
\[
\widetilde{\pi}_L\colon\widetilde{W}\to\P^3.
\]
We use the following notation for divisor classes in $\widetilde{W}$: $h$ is the pull-back of a hyperplane section on $\P^3$, $H$ is the pull-back of a hyperplane section on $W$, and the exceptional divisor $D$ of the blow-up has the form $D=H-h$.

As explained in Example \ref{ex:QuadricFibrations}, the conic fibration structure produces a sheaf $\cB_0$ of even parts of Clifford algebras on $\P^3$ and a semiorthogonal decomposition
\begin{equation}\label{eqn:quadrica}
\Db(\widetilde{W}) =\langle \Phi(\Db(\P^3,\cB_0)),\underbrace{\cO_{\widetilde{W}}(-h),\cO_{\widetilde{W}},\cO_{\widetilde{W}}(h),\cO_{\widetilde{W}}(2h)}_{\widetilde{\pi}^*_L\Db(\P^3)}\rangle.
\end{equation}
The Serre functor $S_{\cB_0}$ of the non-commutative smooth projective variety $\Db(\P^3,\cB_0)$ has the form
\[
S_{\cB_0}(\blank)=\blank\otimes_{\cB_0}\cB_{-3}[3]
\]
(see \cite[Section 7]{BLMS:inducing}).

On the other hand, by Example \ref{ex:BlowUps} (and after some simple mutations), the blow-up structure gives a semiorthogonal decomposition
\begin{equation}\label{eqn:blupquadrics}
\Db(\widetilde{W})=\langle\sigma^*\Ku(W),\cO_{\widetilde{W}}(h-H),\cO_{\widetilde{W}},\cO_{\widetilde{W}}(h),\cO_{\widetilde{W}}(H),\cO_D(h), \cO_{\widetilde{W}}(2h),\cO_{\widetilde{W}}(H + h)\rangle.
\end{equation}

If one compares \eqref{eqn:blupquadrics} and \eqref{eqn:quadrica} and perform some elementary mutations and adjunctions, one sees that $\Ku(W)$ embeds into $\Db(\P^3,\cB_0)$ together with three more exceptional objects in \eqref{eqn:blupquadrics}. The precise statement is the following result (see \cite[Proposition 7.7]{BLMS:inducing}):

\begin{prop}\label{prop:twistedP3}
Under the above assumptions,
\[
\Db(\P^3,\cB_0)=\left\langle \Psi ( \sigma^* \Ku(W)), \cB_{1},\cB_{2}, \cB_{3}\right\rangle,
\]
where $\Psi$ is the left adjoint of $\Phi$.
\end{prop}

As a conclusion, the composition $\Psi\circ\sigma^*$ yields a fully faithful embedding of the Kuznetsov component of $W$ into the derived category of a non-commutative Fano manifold of dimension $3$. This is crucial to reduce the complexity of the computations in the next section.

\begin{rema}\label{rmk:quadricinfamilies}
It was observed in \cite{BLMNPS:families} that the above construction works for families of cubic fourfolds over a suitable base. One needs this family to have a section for the relative Fano variety of lines (and the lines should not be contained in planes in the corresponding cubic fourfold).
\end{rema}

To apply the techniques discussed in Section \ref{subsec:inducing} and produce stability conditions on $\Ku(W)$, we need to be able to talk about slope and tilt stability for the abelian category $\coh(\P^3,\cB_0)$ and its tilts, respectively.

Consider the forgetful functor $\mathrm{Forg}\colon\Db(\P^3,\cB_0)\to\Db(\P^3)$ and the \emph{twisted Chern character} defined as
\begin{equation}\label{enq:twistChern}
\mathrm{ch}_{\cB_0} (E) := \mathrm{ch}(\mathrm{Forg}(E)) \left( 1 - \frac{11}{32}\ell\right),
\end{equation}
for all $E\in\Db(\P^3,\cB_0)$, where $\ell$ denotes the class of a line in $\P^3$. We denote by $\mathrm{ch}_{\cB_0,i}$ the degree $i$ component of $\mathrm{ch}_{\cB_0}$, and identify them with rational numbers.

Define on $\coh(\P^3,\cB_0)$ the slope function
\[
\mu^{\cB_0}_h(E):=\left\{\begin{array}{ll}\frac{\mathrm{ch}_{\cB_0, 1}(E)}{\mathrm{ch}_{\cB_0, 0}(E)}&\text{if } \mathrm{ch}_{\cB_0, 0} (E)\neq 0\\ +\infty&\text{otherwise,}\end{array}\right.
\]
This function induces a weak stability condition on $\coh(\P^3,\cB_0)$. Thus it makes sense to talk about $\mu^{\cB_0}_h$-(semi)stable (or, simply, slope-(semi)stable) objects in $\coh(\P^3,\cB_0)$.

The key result is the following generalization of the Bogomolov-Gieseker inequality for slope-semistable sheaves \cite[Theorem 8.3]{BLMS:inducing}:

\begin{theo}\label{thm:Bogomolov}
	For any $\mu^{\cB_0}_h$-semistable sheaf $E\in\coh(\P^3,\cB_0)$, we have
	\[
	\Delta_{\cB_0}(E):= \ch_{\cB_0,1}(E)^2 - 2 \ch_{\cB_0,0}(E) \ch_{\cB_0,2}(E) \geq0.
	\]
\end{theo}

It is not difficult to see that $\cB_i$ is slope-stable and, with the correction given by $\frac{11}{32}$ in \eqref{enq:twistChern}, we have $\Delta_{\cB_0}(\cB_i)=0$, for all $i\in\Z$.
The proof of Theorem \ref{thm:Bogomolov} follows a similar approach as in Langer's proof of the usual Bogomolov-Gieseker inequality \cite{Lan:Bogomolov}.

\subsection{Existence of Bridgeland stability conditions}
\label{subsec:ExistenceCubicFourfolds}

We keep assuming that $\mathrm{char}(\K)\neq2$.
In this section we apply the construction of the previous section to construct Bridgeland stability conditions on $\Ku(W)$ and describe a connected component of $\mathrm{Stab}(\Ku(W))$.

The discussion in Section \ref{subsec:Tilt} works verbatim also for the non-commutative variety $(\P^3,\cB_0)$. In particular, for a given $\beta\in\R$, we consider the modified twisted Chern character
\[
\mathrm{ch}^\beta_{\cB_0}:=e^{-\beta}\cdot\mathrm{ch}_{\cB_0}
\]
and take the lattice $\Lambda_{\cB_0}$ generated by the vectors
\[
\left(\mathrm{ch}_{\cB_0,0}(E),\mathrm{ch}_{\cB_0,1}(E),\mathrm{ch}_{\cB_0,2}(E)\right)\in\Q^{\oplus 3},
\]
for all $E\in\Db(\P^3,\cB_0)$. By Theorem \ref{thm:Bogomolov} we have a quadratic form $\Delta_{\cB_0}$ on $\Lambda_{\cB_0}$.

Let $\Lambda_{\cB_0,\Ku(W)}\subseteq \Lambda_{\cB_0}$ be the lattice which is the image of $K(\Ku(W))$ under the natural composition $K(\Ku(W))\to K(\Db(\P^3,\cB^0))\to\Lambda_{\cB_0}$. Hence we have a surjection
\[
v\colon K(\Ku(W))\twoheadrightarrow\Lambda_{\cB_0,\Ku(W)}.
\]

\begin{rema}\label{rmk:diffcomp}
It is not difficult to see that $v$ has the following alternative description. The composition of $\mathrm{Forg}$ and of the fully faithful functor $\Psi\circ\sigma^*\colon\Ku(W)\to\Db(\P^3,\cB_0)$ induces a morphism at the level of numerical Grtothendieck groups which, composed with the Chern character $\mathrm{ch}_{\cB_0}$ truncated at degree $2$, yields a surjective morphism
\[
u\colon\widetilde{H}_{\mathrm{alg}}(\Ku(W),\Z)\twoheadrightarrow \Lambda_{\cB_0,\Ku(W)}
\]
Given the Mukai vector $\vv\colon K(\Ku(W))\to\widetilde{H}_{\mathrm{alg}}(\Ku(W),\Z)$, we have $v=u\circ\vv$.
\end{rema}

For $\beta\in\R$, we consider the abelian category $\coh^\beta(\P^3,\cB_0)$,
which is the heart of a bounded t-structure obtained by tilting $\coh(\P^3,\cB_0)$ with respect to slope-stability at the slope $\mu^{\cB_0}_h=\beta$. Moreover, for $\alpha\in\R_{>0}$ and all $\beta\in\R$, consider the function
\[
Z_{\alpha, \beta}(E) := \frac 12 \alpha^2 \mathrm{ch}_{\cB_0,0}^\beta(E) - \mathrm{ch}_{\cB_0,2}^\beta(E) + \mathfrak{i}\,\mathrm{ch}_{\cB_0,1}^\beta(E)
\]
defined on $\Lambda_{\cB_0,\Ku(W)}$ and taking values in $\C$. By Proposition \ref{prop:tiltstab}, the pair
\[
\sigma_{\alpha, \beta} := (\coh^\beta(\P^3,\cB_0), Z_{\alpha, \beta})
\]
is a weak stability condition on $\Db(\P^3,\cB_0)$ with respect to $\Lambda_{\cB_0,\Ku(W)}$ (see \cite[Proposition 9.3]{BLMS:inducing}). The support property is provided by the quadratic form given by $\Delta_{\cB_0}$.

We want to use this together with Proposition \ref{prop:inducestability} to prove the following result.

\begin{theo}\label{thm:main1}
	If $W$ is a cubic fourfold, then $\mathrm{Stab}_{\Lambda_{\cB_0,\Ku(W)}}(\Ku(W))$ is non-empty.
\end{theo}

\begin{proof}
This is \cite[Theorem 1.2]{BLMS:inducing}; it is now easy to sketch a proof.
Let us fix a line $L\subset W$ not contained in a plane in $W$. By Proposition \ref{prop:twistedP3}, we have
\[
\Db(\P^3,\cB_0)=\langle\Ku(W),\cB_1,\cB_2,\cB_3\rangle.
\]

Consider the weak stability condition $\sigma_{\alpha, \beta}$ mentioned above and set $\beta = -1$. Define the slope function $\mu_{\alpha,-1}:=\mu_{\sigma_{\alpha,-1}}$ associated to $\sigma_{\alpha,-1}$.
Let us tilt $\coh^{-1}(\P^3,\cB_0)$ again with respect to $\mu_{\alpha,-1}$ at $\mu_{\alpha,-1}=0$, getting the heart $\coh^{0}_{\alpha,-1}(\P^3,\cB_0)$ of a bounded t-structure on $\Db(\P^3,\cB_0)$.
Set $Z_{\alpha,-1}^0:=-\mathfrak{i}Z_{\alpha,-1}$.
By Proposition \ref{prop:tiltedtiltstability}, the pair
\[
\sigma_{\alpha,-1}^0:=(\coh^{0}_{\alpha,-1}(\P^3,\cB_0), Z_{\alpha,-1}^0)
\]
is a weak stability condition on $\Db(\P^3,\cB_0)$. We want to check that the assumptions in Proposition \ref{prop:inducestability} are satisfied for $\cD_1=\Ku(W)$, $\cD_2=\langle\cB_1,\cB_2,\cB_3\rangle$ and $\alpha$ sufficiently small. This would then conclude the proof.

Let us start form (1). A direct computation shows that $\cB_1,\cB_2,\cB_3$, $\cB_{-2}[1],\cB_{-1}[1],\cB_{0}[1]$ belong to $\coh^{-1}(\P^3,\cB_0)$, and they are $\sigma_{\alpha,-1}$-stable for all $\alpha>0$. For $\alpha$ sufficiently small, one directly proves that
\[
\mu_{\alpha,-1}(\cB_{-2}[1])<\mu_{\alpha,-1}(\cB_{-1}[1])<\mu_{\alpha,-1}(\cB_0[1])
<0<\mu_{\alpha,-1}(\cB_1)<\mu_{\alpha,-1}(\cB_2)<\mu_{\alpha,-1}(\cB_3).
\]
Hence $\cB_1$, $\cB_2$ and $\cB_3$ are contained in $\coh^{0}_{\alpha,-1}(\P^3,\cB_0)$, for $\alpha$ sufficiently small.

As for (2), observe that, by \cite[Corollary 3.9]{Kuz:quadrics}, $\cB_i\otimes_{\cB_0}\cB_j\cong\cB_{i+j}$. Thus $S_{\cB_0}(\cB_j)\cong\cB_{j-3}[3]$ and $S_{\cB_0}(\cB_j)\in\coh^{0}_{\alpha,-1}(\P^3,\cB_0)[1]$, for $\alpha$ sufficiently small and $j=1,2,3$. A very simple check yields (3).

Finally, by \cite[Lemma 2.15]{BLMS:inducing}, if $E\in\coh^{0}_{\alpha,-1}(\P^3,\cB_0)$ is such that $Z_{\alpha,-1}^0(E)=0$, then $\mathrm{Forg}(E)$ is a torsion sheaf supported in dimension $0$.
But then $F\notin\Ku(X)$ because 
\[
\Hom_{\cB_0}(\cB_j,F)\cong\Hom_{\cB_0}(\cB_0,F)\cong\Hom_{\P^3}(\cO_{\P^3},\mathrm{Forg}(F)),
\]
where, in the last isomorphism, we used the adjunction between the functors $\blank\otimes_{\cO_{\P^3}}\cB_0$ and $\mathrm{Forg}$ (see again \cite[Section 7]{BLMS:inducing}).
\end{proof}

To finish the section, we enlarge the lattice with respect to which the support property holds to get the analogue of Bridgeland's result for K3 surfaces \cite{Bri:K3} in Example \ref{ex:K3}.

\begin{defi}
A \emph{full numerical stability condition} on $\Ku(W)$ is a Bridgeland stability condition on $\Ku(W)$ whose lattice $\Lambda$ is given by the Mukai lattice $\widetilde{H}_\mathrm{alg}(\Ku(W),\Z)$ and the map $v$ is given by the Mukai vector $\vv$.
\end{defi}

Let $\mathrm{Stab}(\Ku(W))$ be the set of full stability conditions. As explained in Section \ref{subsec:BridgelandCovering}, we have a map
\[
\eta\colon\mathrm{Stab}(\Ku(W))\to\widetilde{H}_\mathrm{alg}(\Ku(W),\C)
\]
together with a period domain $\cP_0\subseteq\widetilde{H}_\mathrm{alg}(\Ku(W),\C)$ such that $\eta|_{\eta^{-1}(\cP_0)}$ is a covering map (see Theorem \ref{thm:cov}).

Let $\sigma=(\cA,Z)$ be the stability condition constructed in the proof of Theorem \ref{thm:main1}. Consider the pair $\sigma':=(\cA,Z')$, where $Z':=Z\circ u$ (see Remark \ref{rmk:diffcomp}). Then $\sigma'$ is a full stability condition and $\eta(\sigma')\in\cP_0$. This is proven in \cite[Proposition 9.10]{BLMS:inducing} and it implies that the open subset $\eta^{-1}(\cP_0)$ is non-empty.
Let $\cP_0^+$ denote the connected component of $\cP_0$ containing $\eta(\sigma')$, and let $\mathrm{Stab}^\dagger(\Ku(W))$ denote the connected component of $\mathrm{Stab}(\Ku(W))$ containing $\sigma'$.

\begin{theo}\label{thm:stabconn}
The connected component $\mathrm{Stab}^\dagger(\Ku(W))$ is contained in $\eta^{-1}(\cP_0^+)$.
In particular, the restriction $\eta\colon\mathrm{Stab}^\dagger(\Ku(W))\to\cP_0^+$ is a covering map.
\end{theo}

Theorem \ref{thm:stabconn} is proved in \cite{BLMNPS:families}.
The key ingredient is Theorem \ref{thm:YoshiokaMain} below and the notion of family of stability conditions to reduce to the K3 surface case (Example \ref{ex:K3}).

\begin{rema}\label{rmK:GM}
(i) If $\Ku(W)\cong\Db(S)$, for $S$ a K3 surface, then by construction $\mathrm{Stab}^\dagger(\Ku(W))$ coincides with the connected component $\mathrm{Stab}^\dagger(\Db(S))$ discussed in Example \ref{ex:K3}.

(ii) We expect Theorems \ref{thm:main1} and \ref{thm:stabconn} to hold also in the case of Gushel-Mukai fourfolds. Indeed, those varieties have a conic fibration that makes it very plausible that the approach in Section \ref{subsec:ConicFibrations} would work also in that case. One of the difficulties consists in proving the analogue of Theorem \ref{thm:Bogomolov} in this new geometric setting.	
\end{rema}

\subsection{Moduli spaces}
\label{subsec:ModuliSpaces}

The most important consequence of Theorem \ref{thm:main1} is to be able to construct and study moduli spaces of stable objects on Kuznetsov components in an analogous way as the Mukai theory for K3 surfaces.

\subsubsection*{General properties of moduli spaces of complexes}

Let $\cD\subset\Db(X)$ be a non-commutative smooth projective variety, where $X$ is a smooth projective variety over $\K$.
In a similar way as in Section \ref{subsec:NC}, given a scheme $B$, locally of finite type over $\K$, we can define a quasi-coherent product category
\[
\cD_{\mathrm{Qcoh}}\boxtimes \mathrm{D}_{\mathrm{Qcoh}}(B)\subset\mathrm{D}_{\mathrm{Qcoh}}(X\times B)
\]
in the unbounded derived category of quasi-coherent sheaves on $X\times B$ (this is the smallest triangulated subcategory closed under arbitrary direct sums and containing $\cD\boxtimes\Db(B)$; see \cite{Kuz:BaseChange}).

\begin{defi}
An object $E\in\mathrm{D}_{\mathrm{Qcoh}}(X\times B)$ is \emph{$B$-perfect} if it is, locally over $B$, isomorphic to a bounded complex of quasi-coherent sheaves on $B$ which are flat and of finite presentation.
\end{defi}

Roughly, complexes which are $B$-perfect are those which can be restricted to fibers over $B$.
We denote by $\mathrm{D}_{B\text{-}\mathrm{perf}}(X\times B)$ the full subcategory of $\mathrm{D}(\mathrm{Qcoh}(X\times B))$ consisting of $B$-perfect complexes, and
\[
\cD_{B\text{-}\mathrm{perf}}:=(\cD_{\mathrm{Qcoh}}\boxtimes \mathrm{D}_{\mathrm{Qcoh}}(B))\cap \mathrm{D}_{B\text{-}\mathrm{perf}}(X\times B).
\]

Consider the $2$-functor
\[
\mathfrak{M}\colon\mathbf{Sch}\to\mathbf{Grp}
\]
which maps a scheme $B$ which is locally of finite type over $\K$ to the groupoid
\[
\mathfrak{M}(B):=\left\{E\in\cD_{B\text{-}\mathrm{perf}}\,:\, \begin{array}{l}\Ext^i(E|_{X\times\{b\}},E|_{X\times\{b\}}) = 0,\text{ for all }i < 0\\
\text{and all geometric points }b \in B\end{array}\right\}.
\]

The following is the main result in \cite{Lie:Mod}:

\begin{theo}[Lieblich]\label{thm:Lieblich}
The functor $\mathfrak{M}$ is an Artin stack, locally of finite type, locally quasi-separated and with separated diagonal.
\end{theo}

To be precise, in \cite{Lie:Mod} only the case $\cD=\Db(X)$ is considered; for the extension to non-commutative smooth projective varieties, simply observe that the property of an object in $\Db(X)$ to be contained in $\cD$ is open. Recall also that a stack is locally quasi-separated if it admits a Zariski covering with substacks which are quasi-separated.

Consider the open substack $\mathfrak{M}_\mathrm{Spl}$ of $\mathfrak{M}$ parameterizing simple objects. Recall that an object $E$ is \emph{simple} if $\Hom(E,E)\cong\K$. This is again an Artin stack, locally of finite type, locally quasi-separated and with separated diagonal. One can take another functor
\[
\underline{M}_\mathrm{Spl}\colon\mathbf{Sch}\to\mathbf{Set}
\]
obtained from $\mathfrak{M}_\mathrm{Spl}$ by forgetting the groupoid structure and quotienting by the equivalence relation obtained by tensoring by pull-backs of line bundles on $B$. A previous result by Inaba \cite{In:Moduli} ensures that $\underline{M}_\mathrm{Spl}$ is represented by an algebraic space $M_\mathrm{Spl}$ which is locally of finite type over $\K$.

\subsubsection*{Bridgeland moduli spaces}

Assume now we have a Bridgeland stability condition $\sigma=(\cA,Z)\in\mathrm{Stab}_\Lambda(\cD)$, and let $\vv\in\Lambda$.
To be precise, we also need to choose a phase $\phi\in\R$ such that $Z(\vv)\in \R_{>0}\,\exp(i\pi\phi)$.
We denote by $\mathfrak{M}_\sigma(\cD,\vv,\phi)$ the substack of $\mathfrak{M}$ parameterizing $\sigma$-semistable objects in $\cD$ with class $\vv$ and phase $\phi$. Often we will use the simplified notation $\mathfrak{M}_\sigma(\cD,\vv)$.

A priori it does not follow from the definition of Bridgeland stability condition that this is an open substack of finite type over $\K$.\footnote{This condition of openness and boundedness of stability should probably be assumed in the definition of Bridgeland stability condition; see indeed \cite{KS:stability,BLMNPS:families}.}
Also, even if this is satisfied, a priori it is not clear that a \emph{good moduli space} (in the sense of Alper \cite{Alp:good}) exists; a positive result on this direction is due to Alper, Halpern-Leistner and Heinloth \cite{AHLH}.
Still in the above assumptions, if the class $\vv$ is primitive in $\Lambda$ and the stability condition $\sigma$ generic with respect to $\vv$, then there are no properly semistable objects and so a good moduli space $M_\sigma(\cD,\vv)$ exists, as a subspace of $M_\mathrm{Spl}$ of finite type over $\K$.

It is a key result by Toda \cite{Toda:K3} that the stability conditions constructed by tilting in Proposition \ref{prop:tiltstab} do satisfy openness and boundedness.

More generally, for a K3 surface $S$, still by \cite{Toda:K3}, this is true for every stability condition in $\mathrm{Stab}^\dagger(\Db(S))$.
Moreover, by \cite{BM:projectivity,MYY}, if $\vv\in\widetilde{H}_{\mathrm{alg}}(S,\Z)$ is a non-zero vector and $\sigma$ is a stability condition which is generic with respect to $\vv$, then $M_\sigma(\Db(S),\vv)$ exists as a projective variety.
If $\vv$ is primitive with $\vv^2+2\geq0$, then $M_\sigma(\Db(S),\vv)$ is a non-empty smooth projective hyperk\"ahler manifold of dimension $\vv^2+2$ which is deformation equivalent to a Hilbert scheme of points on a K3 surface (this is (K3.5) in the introduction; the condition $\vv^2+2\geq0$ is also necessary for non-emptyness).

We will not state here the actual result \cite[Theorem 1.3]{BM:projectivity}, since we will state the analogous version for non-commutative K3 surfaces arising from cubic fourfolds in Theorem \ref{thm:YoshiokaMain} below.
On the other hand, we do need the result on K3 surfaces to prove the result for cubic fourfolds. Hence we give a very short idea of the proof: there exists a Fourier-Mukai partner of $S$ (which may be twisted) such that $M_\sigma(\Db(S),\vv)$ becomes a moduli space of Gieseker semistable (twisted) vector bundles. And so the result follows directly from the sheaf case in \cite{Yo:ModAbVar}.

\subsubsection*{The Kuznetsov component of a cubic fourfold}

The main result for non-commutative K3 surfaces associated to cubic fourfolds is the following theorem from \cite{BLMNPS:families}.
We assume in this section for simplicity that the base field is $\C$; while essentially all results hold true more generally, the fact that moduli spaces are projective relies on an analytic result.

\begin{theo}\label{thm:YoshiokaMain}
Let $W$ be a cubic fourfold.
Then
\[
\widetilde{H}_{\mathrm{Hodge}}(\Ku(W),\Z)=\widetilde{H}_{\mathrm{alg}}(\Ku(W),\Z).
\]
Moreover, assume that $\mathbf{v}\in\widetilde{H}_{\mathrm{alg}}(\Ku(W),\Z)$ is a non-zero primitive vector and let $\sigma\in\mathrm{Stab}^\dagger(\Ku(W))$ be a stability condition on $\Ku(W)$ that is generic with respect to $\mathbf{v}$.
Then
\begin{enumerate}
\item\label{item:YoshiokaMain1} $M_{\sigma}(\Ku(W),\vv)$ is non-empty if and only if $\vv^2+2\geq0$. Moreover, in this case, it is a smooth projective irreducible holomorphic symplectic variety  of dimension $\mathbf{v}^2 + 2$, deformation-equivalent to a Hilbert scheme of points on a K3 surface.

\item\label{item:YoshiokaMain2} If $\vv^2\geq 0$, then there exists a natural Hodge isometry \[\theta\colon H^2(M_\sigma(\Ku(W),\vv),\Z)\xrightarrow{\quad\sim\quad}\begin{cases}\vv^\perp & \text{if }\vv^2>0\\ \vv^\perp/\Z\vv & \text{if } \vv^2=0,\end{cases}\]
where the orthogonal is taken in $\widetilde{H}(\Ku(W),\Z)$.
\end{enumerate}
\end{theo}

\begin{rema}\label{rmk:nonprim}
It would certainly be very interesting to consider the case when $\vv=m\vv_0$, for some $m>1$ (i.e., $\vv$ is not primitive). Then Theorem \ref{thm:YoshiokaMain} shows that $M_{\sigma}(\Ku(W),\mathbf{v})$ is non-empty (i.e., it contains the class of a semistable object) if and only if $\mathbf{v}_0^2\geq-2$ (see \cite[Theorem~2.6]{BM:walls}), for $\sigma$ a $\mathbf{v}$-generic stability condition.
On the other hand, by \cite{AHLH}, $M_{\sigma}(\Ku(W),\mathbf{v})$ admits a good moduli space. If we can prove that $M_\sigma(\Ku(X),\vv)$ is also normal, then we can deduce further that $M_\sigma(\Ku(W),\vv)$ is an irreducible proper algebraic space and either $\dim M_\sigma(\Ku(W),\mathbf{v})=\mathbf{v}^2+2$ and the stable locus is non-empty, or $m>1$ and $\mathbf{v}^2\leq0$.
\end{rema}

The proof of (1), which is the only part of the statement we focus on, is mainly based on a deformation argument which we can sketch as follows. Suppose that $\mathbf{v}$ is such that $\mathbf{v}^2\geq-2$. It is not difficult to see that one can construct a family $\cW\to C$ of smooth cubic fourfolds over a smooth curve $C$ such that $W\cong\cW_p$, for some point $p\in C$ while $\Ku(\cW_q)\cong\Db(S)$, for some other point $q\in C$. To find $q$, it is enough to deform $W$ to one of the divisors mentioned in Section \ref{subsec:CubicFourfolds}, corresponding to cubic fourfolds with homologically associated K3 surface. In view of Proposition \ref{prop:KuznetsovPfaffian}, we can simply consider the Noether-Lefschetz divisor containing Pfaffian cubic fourfolds.

If we assume further that $\mathbf{v}$ is contained in $\widetilde{H}_\mathrm{Hodge}(\Ku(\cW_c),\Z)$, for all $c\in C$, that we can consider the relative moduli space of Bridgeland (semi)stable objects $M(\mathbf{v})\to C$ in the Kuznetsov components such that $M(\mathbf{v})_p\cong M_{\sigma}(\Ku(W),\vv)$ while $M(\mathbf{v})_q$ is a moduli space of stable objects in $\Ku(\cW_q)$. This morphism turns out to be smooth and proper. Thus, to prove that $M(\mathbf{v})_p$ is non-empty, we just need to prove that $M(\mathbf{v})_q\neq\emptyset$. Since $\Ku(\cW_q)\cong\Db(S)$, this follows from the analogous statement for K3 surfaces mentioned before \cite[Theorem 1.3]{BM:projectivity}.
Notice also that, since $\widetilde{H}_\mathrm{Hodge}(\Db(S),\Z)=\widetilde{H}_\mathrm{alg}(\Db(S),\Z)$, the previous deformation argument implies the analogous statement holds for $\Ku(W)$ as well.

The fact that $M_{\sigma}(\Ku(W),\vv)$ is symplectic follows from the fact that $\Ku(W)$ is a non-commutative K3 surface. Indeed, the tangent space of $M_{\sigma}(\Ku(W),\vv)$ at $E$ is identified to $\Hom(E,E[1])$ and Serre duality for $\Ku(W)$ yields a non-degenerate skew-symmetric pairing
\[
\Hom(E,E[1])\times\Hom(E,E[1])\to\Hom(E,E[2])\cong\Hom(E,E)\cong\C.
\]
The last isomorphism follows from the fact that $E$ is stable. Grothendieck-Riemann-Roch allows us to compute the dimension of $\Hom(E,E[1])$ and thus the dimension of $M_\sigma(\Ku(W),\vv)$.
The closedness of this symplectic form follows as in the case for K3 surfaces (proved in \cite[Theorem 3.3]{Inaba2} and \cite{KM:sympl}).

To prove that this symplectic manifold is irreducible and projective (and that the resulting variety is deformation-equivalent to a Hilbert scheme of points), we use a general fact in \cite{BM:projectivity}. The stability condition $\sigma$ induces a nef line bundle on $M_\sigma(\Ku(W),\vv)$ as follows.
Let $E$ be the (quasi-)universal family in $\Db(M_\sigma(\Ku(W),\vv)\times W)$. Since $M_\sigma(\Ku(W),\vv)$ is a moduli space of stable objects in $\Ku(W)$, $E$ is a family of objects in $\Ku(W)$. We then define the numerical Cartier divisor $\ell_\sigma\in\NS(M_\sigma(\vv))_\R$ via the following assignment:
\[
C\mapsto \ell_\sigma. C:=\Im\left(-\frac{Z(\vv(\Phi_E(\cO_C)))}{Z(\vv)}\right),
\]
for every curve $C\subseteq M_\sigma(\Ku(W),\vv)$. The Positivity Lemma of \cite{BM:projectivity} implies that $\ell_\sigma$ is nef. A careful application of the main result in \cite{P:deg} implies that a positive multiple of $\ell_\sigma$ is big. Thus the Base Point Free Theorem (see \cite[Theorem 3.3]{KollarMori}) implies that a positive multiple of $\ell_\sigma$ is ample.

To make this rigorous one needs to develop a new theory of stability conditions in families and this is what is carried out in \cite{BLMNPS:families}. In particular, this includes the crucial construction of the relative moduli spaces with respect to a family of stability conditions.

\begin{rema}\label{rmK:GMmoduli}
Assume we know that, for a Gushel-Mukai fourfold $X$, the Kuznetsov component $\Ku(X)$ carries a stability condition which behaves nicely in family (see Remark \ref{rmK:GM}). Then Theorem \ref{thm:YoshiokaMain} holds for $X$ as well. Indeed, the same proof applies using again a degeneration to divisors in the moduli space of Gushel-Mukai fourfolds parameterizing fourfolds whose Kuznetsov component is equivalent to the derived category of a K3 surface.
\end{rema}

As a consequence, one can construct $20$-dimensional locally complete families of hyperk\"ahler manifolds. Indeed, take a family $\mathcal{W} \to S$ of cubic fourfolds. Let $\vv$ be a primitive section of the local system given by the Mukai lattices $\widetilde{H}(\Ku(\mathcal{W}_s),\mathbb{Z})$ of the fibers over $s \in S$, such that $\vv$ stays
algebraic on all fibers. Assume that for $s \in S$ very general, there exists a stability
condition $\sigma_s \in \mathrm{Stab}^\dagger(\Ku(\mathcal{W}_s)$ that is generic with respect to $\vv$, and
such that the associated central charge $Z \colon
\widetilde{H}_{\mathrm{alg}}(\Ku(\mathcal{W}_s),\Z)\to \mathbb{C}$ is
monodromy-invariant.

\begin{theo} \label{thm:Msigmarelative}
There exists a non-empty open subset $S^0 \subset S$ and a variety $M^0(\vv)$ with a projective morphism $M^0(\vv) \to S^0$ that makes $M^0(\vv)$ a relative moduli space over $S^0$: the fiber over $s \in S^0$ is a moduli space $M_{\sigma_s}(\Ku(\mathcal{W}_s),\mathbf{v})$ of stable objects in the Kuznetsov category of the corresponding cubic fourfold.
\end{theo}

As we mentioned above, this has the following nice application.

\begin{coro}\label{cor:locfam20dim}
For any pair $(a,b)$ of coprime integers, there is a unirational locally complete $20$-dimensional family, over an open subset of the moduli space of cubic fourfolds, of polarized smooth projective irreducible holomorphic symplectic manifolds of dimension $2n+2$, where $n=a^2-ab+b^2$. 
The polarization has divisibility $2$ and degree either $6n$ if $3$ does not divide $n$, or $\frac{2}{3}n$ otherwise.
\end{coro}

In this framework we recover some of the classical families of hyperk\"ahler manifolds associated to cubic fourfolds. This is the content of some of the applications discussed below.

\subsection{Applications}
\label{subsec:applications}

In the remaining part of these lecture notes, we want to focus on some geometric applications of Theorem \ref{thm:YoshiokaMain} and examples.

\subsubsection*{End of the proof of Theorem \ref{thm:AT}}

The easy implication in Theorem \ref{thm:AT} was proved at the end of Section \ref{subsec:Mukai}. Let us now show that if there is a primitive embedding $U\hookrightarrow\widetilde{H}_\mathrm{alg}(\Ku(W),\Z)$, then there is a K3 surface $S$ and an equivalence $\Db(S)\cong\Ku(W)$.

Under our assumption, there is $\mathbf{w}\in\widetilde{H}_{\mathrm{alg}}(\Ku(W),\Z)$ such that $\mathbf{w}^2=0$. Pick a stability condition $\sigma\in\mathrm{Stab}^\dagger(\Ku(W))$ which is generic with respect to $\mathbf{w}$. By Theorem \ref{thm:YoshiokaMain}, the moduli space $M_\sigma(\Ku(W),\mathbf{w})$ of $\sigma$-stable objects in $\Ku(W)$ is non empty, and in fact a K3 surface $S$. Such a moduli space comes with a quasi-universal family $E$ yielding a fully faithful functor $\Phi_E\colon\Db(S,\alpha)\to\Db(W)$, for some $\alpha\in\mathrm{Br}(S)$. Since $S$ parametrizes stable objects in $\Ku(W)$, the functor $\Phi_E$ factors through $\Ku(W)$ and thus provides a fully faithful functor $\Phi_E\colon\Db(S,\alpha)\hookrightarrow\Ku(W)$. As $\Ku(W)$ is connected, this functor is actually an equivalence. Notice that this already proves the non-trivial implication in the generalization of Theorem \ref{thm:AT} in Remark \ref{rmk:twistedK3}.

Since $U$ embeds in $\widetilde{H}_\mathrm{alg}(\Ku(W),\Z)$, we have further a  vector $\mathbf{v}$ such that $(\vv,\mathbf{w})=1$. A standard argument shows that, under these assumptions, the quasi-universal family $E$ is actually universal. Thus $\alpha$ is trivial and we get an equivalence $\Db(S)\cong\Ku(W)$.

\begin{rema}\label{rmk:ratagain}
Theorem \ref{thm:AT} also shows that Kuznetsov's categorical conjectural condition for rationality (i.e., $\Ku(W)\cong\Db(S)$) matches the more classical Hodge theoretical one due to Harris and Hassett \cite{Has:special} (i.e., $W$ has a Hodge theoretically associated K3 surface). At the moment, the list of divisors parameterizing rational cubic fourfolds and verifying Harris-Hassett-Kuznetsov prediction is short but interesting (see \cite{RS:cubics} for recent developments and \cite{NS:StableRationalitySpecialize,KT:RationalitySpecialize} for very recent and interesting results about specialization for stably rationality and rationality).
\end{rema}

\subsubsection*{The integral Hodge conjecture for cubic fourfolds}

The rational Hodge conjecture for cubic fourfolds was proved by Zucker in \cite{Z:Hodge} (see also \cite{CM:Hodge}). On the other hand, on cubic fourfolds the Hodge conjecture holds for integral coefficients as well; this is due to Voisin \cite[Theorem 18]{Voi:aspects} and can be reproved directly as a corollary of Theorem \ref{thm:YoshiokaMain} as follows (see \cite{BLMNPS:families}).\footnote{The argument was also suggested to us by Claire Voisin.}

\begin{prop}[Voisin]\label{prop:intHodge}
The integral Hodge conjecture holds for any cubic fourfold $W$.
\end{prop}

\begin{proof}
Consider a class $v\in H^4(W,\Z)\cap H^{2,2}(W)$. By \cite[Section 2.5]{AH:Ktop} (see also \cite[Theorem 2.1 (3)]{AT:cubic}), there exists $w\in K_\mathrm{top}(W)$ such that $\vv(w)=v+\widetilde{v}$, where $\widetilde{v}\in H^6(W,\Q)\oplus H^8(W,\Q)$.

Take the projection $w'$ of $w$ to $\widetilde{H}(\Ku(W),\Z)$ (induced by the projection functor). Then $w$ differs from $w'$ by a linear combination with integral coefficients $w'=w+a_0[\cO_W]+a_1[\cO_W(1)]+a_2[\cO_W(2)]$ in $K_\mathrm{top}(W)$.
Since the projection preserves the Hodge structure, $w'$ is actually in $\widetilde{H}_\mathrm{Hodge}(\Ku(W),\Z) = \widetilde{H}_\mathrm{alg}(\Ku(W),\Z)$, by Theorem \ref{thm:YoshiokaMain}.

Let $E\in\Ku(W)$ be such that $\vv(E)=w'$ and set
\[
F:=E\oplus\cO_W^{\oplus |a_0|}[\epsilon(a_0)]\oplus\cO_W(1)^{\oplus |a_1|}[\epsilon(a_1)]\oplus\cO_W(2)^{\oplus |a_2|}[\epsilon(a_2)],
\]
where for an integer $a\in\Z$, we define $\epsilon(a)=0$ (resp., $=1$) if $a\geq0$ (resp., $a<0$).
Then $c_2(F)=v$, which is therefore algebraic.
\end{proof}

\begin{rema}\label{rmk:GMintegralHodge}
By Remark \ref{rmK:GMmoduli}, if one could prove that the Kuznetsov component of a GM fourfold carries Bridgeland stability conditions, then the theory of moduli spaces would allow us to repeat verbatim the same argument above and prove the integral Hodge conjecture for GM fourfolds.
\end{rema}

\subsubsection*{The Fano variety of lines and the Torelli theorem}

In the seminal paper \cite{BD:cubic}, Beauville and Donagi showed that the Fano variety of lines $F(W)$ of a cubic fourofold $W$ is a smooth projective hyperk\"ahler manifold of dimension $4$. Moreover, $F(W)$ is deformation equivalent to the Hilbert scheme of length-$2$ zero-dimensional subschemes of a K3 surface. The embedding $F(W)$ inside the Grassmannian of lines in $\P^5$ endowes $F(W)$ with a privileged ample polarization induced by the Pl\"ucker embedding.

The study of $F(W)$ as a moduli space of stable objects was initiated in \cite{MS:Fano} for cubic fourfolds containing a plane and satisfying an additional genericity condition. But the techniques discussed here allow us to prove complete results. Indeed, for any cubic fourfold $W$, in the notation of Theorem \ref{thm:stabconn}, fix a stability condition $\sigma\in\mathrm{Stab}^\dagger(\Ku(W))$ such that $\eta(\sigma)\in (A_2)_\C\cap\cP\subseteq\cP_0^+$. We then get the following general result, which is \cite[Theorem 1.1]{lpz}.

\begin{theo}[Li-Pertusi-Zhao]\label{thm:Fano}
In the assumptions above, the Fano variety of lines in $W$ is isomorphic to the moduli space $M_\sigma(\Ku(W),\llambda_1)$. Moreover, the ample line bundle $\ell_\sigma$ on $M_\sigma(\Ku(W),\llambda_1)$ is identified with a multiple of the Pl\"ucker polarization by this isomorphism.
\end{theo}

\begin{proof}[Sketch of the proof]
Following \cite[Appendix A]{BLMS:inducing}, we outline the proof under the genericity assumption that $\widetilde{H}(\Ku(W),\Z)$ does not contain $(-2)$-classes. This means that there is no class $v\in\widetilde{H}(\Ku(W),\Z)$ such that $v^2=-2$. This will be enough for the application to the Torelli Theorem for cubic fourfold.

Let $L$ be a line in $W$. Following \cite{KM:sympl}, consider the kernel of the evaluation map
\[
F_L := \Ker \left( \cO_W^{\oplus 4} \onto I_L(1) \right),
\]
which is a torsion-free Gieseker-stable sheaf. A direct computation shows that
\[
\Hom(F_L,F_L[i])=0\qquad\Hom(F_L,F_L)\cong\C\qquad\Hom(F_L,F_L[1])\cong\C^4,
\]
for $i<0$. Moreover, $\vv(F_L)=\llambda_1$ and $\vv(F_L)^2=2$.

The point is that $F_L$ needs to be $\sigma$-stable for any $\sigma\in\mathrm{Stab}^\dagger(\Ku(W))$, under our genericity assumptions. Indeed, if this is not the case, then one can show that there must exist a distinguished triangle
\[
A\to F_L\to B,
\]
where $A$ and $B$ are $\sigma$-stable and $\dim\Hom(A,A[1])=\dim\Hom(B,B[1])$. Moreover, a direct computation shows that such dimensions are equal to $2$. This means that $\vv(A)^2=\vv(B)^2=0$ and $(\vv(A)+\vv(B))^2=2$. But then $(\vv(A)-\vv(B))^2=-2$. This is a contradiction

The mapping $L\mapsto F_L$ yields an embedding $F(W)\hookrightarrow M_\sigma(\Ku(W),\llambda_1)$. By Theorem \ref{thm:YoshiokaMain}, the latter space is a smooth projective hyperk\"ahler manifold of dimension $4$. Hence $F(W)\cong M_\sigma(\Ku(W),\llambda_1)$.

We omit here the discussion about the polarization $\ell_\sigma$: it is a straightforward computation (which uses \cite[Lemma 9.2]{BM:projectivity} and \cite[Equation (6)]{Add:TwoConj}).
\end{proof}

We are now ready to answer Question \ref{ques:derTorelli} in the case of cubic fourfolds.

\begin{theo}[Huybrechts-Rennemo]\label{thm:CategoricalTorelli}
Let $W_1$ and $W_2$ be smooth cubic fourfolds.
Then $W_1\cong W_2$ if and only if there is an equivalence
$\Phi \colon \Ku(W_1) \to \Ku(W_2)$ such that $O_{\Ku(W_2)}\circ\Phi=\Phi\circ O_{\Ku(W_1)}$.\footnote{It is actually enough to assume that the action of $\Phi$ on $\widetilde{H}_{\mathrm{alg}}$ commutes with the action of the degree-shift functor.}
\end{theo}

\begin{proof}[Sketch of the proof]
This is proved in \cite[Corollary 2.10]{HR:cubics}, by using the Jacobian ring.
We present a sketch of a different proof, by using Theorem \ref{thm:Fano}: the advantage being that this holds over arbitrary characteristics $\neq2$.
It follows from Example \ref{ex:O} that up to composing $\Phi$ with a suitable power of the autoequivalence $O_{\Ku(W_2)}$ defined in Section \ref{subsec:NCCY} and, possibly, with the shift by $1$, we can assume without loss of generality that $\Phi_H\colon\widetilde{H}(\Ku(W_1),\Z)\to\widetilde{H}(\Ku(W_2),\Z)$ is such that $\Phi_H(\llambda_1)=\llambda_1$. Let $\sigma_1$ be any stability condition in $\mathrm{Stab}^\dagger(\Ku(W_1))$ and let $\sigma_2:=\Phi(\sigma_1)$. Then $\Phi$ induces a bijection between the moduli spaces $M_{\sigma_1}(\Ku(W_1),\llambda_1)$ and $M_{\sigma_2}(\Ku(W_2),\llambda_1)$. A more careful analysis, based on the same circle of ideas as in \cite[Section 5.2]{BMMS:3folds}, shows that we can replace $\Phi$ with a Fourier-Mukai functor with the same properties. Hence, one can show that such a bijection can be replaced by an actual isomorphism of smooth projective varieties preserving the ample polarizations $\ell_{\sigma_1}$ and $\ell_{\sigma_2}$ (here we use the same approach as in \cite[Section 5.3]{BMMS:3folds}).

If we pick $\sigma_1\in\mathrm{Stab}^\dagger(\Ku(W_1))$ and $\eta(\sigma_1)\in(A_2)_\C\cap\cP$ then, by using the results in \cite{BLMNPS:families}, it can be proved that we can choose the functor $\Phi$ in such a way that $\sigma_2\in\mathrm{Stab}^\dagger(\Ku(W_2))$ and $\eta(\sigma_2)\in(A_2)_\C\cap\cP$. Theorem \ref{thm:Fano} implies that
\[
M_{\sigma_i}(\Ku(W_i),\llambda_1)\cong F(W_i)
\]
with this isomorphism sending $\ell_{\sigma_i}$ to the Pl\"ucker polarization. Summarizing, we have an isomorphism $F(W_1)\cong F(W_2)$ preserving the Pl\"ucker polarization. A classical result by Chow (see \cite[Proposition 4]{Charles:Torelli}), implies that $W_1\cong W_2$.
\end{proof}

As a consequence, we present the recent proof of the Torelli theorem for cubic fourfolds by Huybrechts and Rennemo \cite{HR:cubics}. This result was originally proved by Voisin in \cite{Voi:Torelli} and subsequently reproved by Loojienga in \cite{Looijenga:cubics}. Based on the Torelli theorem for hyperk\"ahler manifolds \cite{Verbitsky:torelli}, Charles \cite{Charles:Torelli} gave an elementary proof relying on the Fano variety of lines. The new approach combines Theorem \ref{thm:CategoricalTorelli} and the Derived Torelli Theorem for twisted K3 surfaces (Theorem \ref{thm:derTorelli}). Notice that here we assume the local injectivity of the period map for cubic fourfolds which is a classical result (see, for example, \cite[Section 6.3.2]{Voi:book2}).

\begin{theo}[Voisin]\label{thm:Torelli}
Two smooth complex cubic fourfolds $W_1$ and $W_2$ are isomorphic if and only if there exists a Hodge isometry $H^4_{\mathrm{prim}}(W_1,\Z)\cong H^4_{\mathrm{prim}}(W_2,\Z)$.
\end{theo}

\begin{proof}[Sketch of proof]
Let $\phi\colon H^4_{\mathrm{prim}}(W_1,\Z)\xrightarrow{\simeq} H^4_{\mathrm{prim}}(W_2,\Z)$
be a Hodge isometry.
By \cite[Proposition 3.2]{HR:cubics}, it induces a Hodge isometry $\phi'\colon \widetilde{H}(\Ku(W_1),\Z)\xrightarrow{\simeq}\widetilde{H}(\Ku(W_2),\Z)$ that preserves the natural orientation. This can be explained in the following way. By equation \eqref{eqn:perpend} in Remark \ref{rmk:A2}, $H^4_{\mathrm{prim}}(W_i,\Z)\cong\langle\llambda_1,\llambda_2\rangle^\perp$. But $A_2=\langle\llambda_1,\llambda_2\rangle$ has discriminant group $\Z/3\Z$ (see again Remark \ref{rmk:A2}). Thus, by Remark \ref{rmk:Nikulin}, we see that $\phi$ extends to a Hodge isometry $\widetilde{H}(\Ku(W_1),\Z)\cong\widetilde{H}(\Ku(W_2),\Z)$. If the latter isometry is orientation preserving, then we are done. Otherwise, we compose with an orientation reversing Hodge isometry (see Remark \ref{rmk:orient}). So we may assume that $\phi'$ is orientation preserving since the very beginning.

A general deformation argument based on \cite{Huy:cubics} shows that $\phi'$ extends to a local deformation $\mathrm{Def}(W_1)\cong\mathrm{Def}(W_2)$.
The set $D\subset\mathrm{Def}(W_1)$ of points corresponding to cubic fourfolds $W$ such that $\Ku(W)\cong\Db(S,\alpha)$, for $(S,\alpha)$ a twisted K3 surface, and $\widetilde{H}_{\mathrm{alg}}(\Ku(W),\Z)$ does not contain $(-2)$-classes is dense in the moduli space (this follows from Remark \ref{rmk:twistedK3} and \cite[Lemma 3.22]{HMS:generic}).
As explained in \cite[Section~4.2]{HR:cubics}, for any $t\in D$ there is an orientation preserving Hodge isometry $\phi_t\colon\widetilde{H}(\Ku(W_t),\Z)\to\widetilde{H}(\Ku(W''_t),\Z)$ which lifts to an equivalence $\Phi_t\colon\Ku(W_t)\to\Ku(W''_t)$ such that $O_{\Ku(W''_t)}\circ\Phi=\Phi\circ O_{\Ku(W_t)}$. 
Theorem \ref{thm:CategoricalTorelli} implies that $W_t\cong W''_t$, for any $t\in D$. Since the moduli space of cubic fourfolds is separated, this yields $W_1\cong W_2$.
\end{proof}

\begin{rema}\label{rmk:TorelliGM}
One should not expect that the same argument used in the proof of Theorem \ref{thm:Torelli} might work for GM fourfolds. Indeed, in this case the period map is known to have positive dimensional fibers (see \cite[Theorem 4.4]{DIM:GM}).
\end{rema}

\subsubsection*{Twisted cubics}

We can analyze further the relation between the geometry of moduli spaces of rational curves inside cubic fourfolds and the existence of interesting hyperk\"ahler varieties associated to the cubic.

The next case would be to study moduli of conics inside a cubic $W$. But any such conic would be contained in a plane cutting out a residual line on $W$. On the other hand, for a line $L$ in $W$, we can take all projective planes in $\P^5$ passing through $L$. The residual curve in each of these planes is then a conic. In conclusion, the moduli space of conics would be (at least birationally) a $\P^3$-bundle over $F(W)$. Thus there is no interesting new hyperk\"ahler manifold associated to conics.

Contrary to this, the case of rational normal curves of degree $3$ is extremely interesting and has been studied in \cite{LLSvS:TwistedCubics}. We can summarize their work in the following way. Let us start with a cubic fourfold $W$ which does not contain a plane. The irreducible component $M_3(W)$ of the Hilbert scheme containing twisted cubic curves is a smooth projective variety of dimension $10$.

The curves in $M_3(W)$ always span a $\P^3$, so there is a natural morphism from $M_3(W)$ to the Grassmannian $\mathrm{Gr}(3,\P^5)$ of three-dimensional projective subspaces in $\P^5$.
This morphism induces a fibration $M_3(W) \to Z'(W)$, which turns out to be a $\P^2$-fiber bundle.
With some further work, the authors of \cite{LLSvS:TwistedCubics} prove that the variety $Z'(W)$ is also smooth and projective of dimension $8$.

The geometric nature of $Z'(W)$ is difficult to describe but roughly speaking, $Z'(W)$ is constructed as a moduli space of determinantal representations of cubic surfaces in $W$ (see \cite{B1,dolgachev}).
Finally, in $Z'(W)$ there is an effective divisor coming from non-CM twisted cubics on $W$. This divisor can be contracted and after this contraction we get a new variety denoted by $Z(W)$ and which is a smooth projective hyperk\"ahler manifold of dimension $8$. It contains the cubic fourfold $W$ as a Lagrangian submanifold and $Z'(W)$ is actually the blow-up of $Z(W)$ in $W$.

An approach to the study of this hyperk\"ahler manifold by homological methods was initiated in \cite{LLMS:TwistedCubics} (see also \cite{SS:twisted}).
The main result is the following and shows that the whole picture in \cite{LLSvS:TwistedCubics} has a neat modular interpretation.

\begin{theo}\label{thm:twistcub}
Let $W$ be a smooth cubic fourfold not containing a plane and with hyperplane class $H$.
\begin{enumerate}
\item\label{item:main1} Let $\mathbf{v}_1=\left(0,0,H^2,0,-\frac{1}{4}H^4\right)$. 
Then $Z'(W)$ is isomorphic to an irreducible component of the moduli space of Gieseker stable sheaves on $W$ with Chern character $\mathbf{v}_1$.
\item Let $\mathbf{v}_2=\left(3,0,-H^2,0,\frac{1}{4}H^4\right)$. Then: 
\begin{enumerate}
\item[{\rm (2a)}]\label{item:main2a} $Z'(W)$ is isomorphic to an irreducible component of the moduli space of Gieseker stable torsion free sheaves on $W$ with Chern character $\mathbf{v}_2$.
\item[{\rm (2b)}]\label{item:main2b} If $W$ is very general, both $Z(W)$ and $Z'(W)$ are isomorphic to an irreducible component of the moduli space of tilt-stable objects on $\Db(W)$ with Chern character $\mathbf{v}_2$.
The contraction $Z'(W)\to Z(W)$ is realized as a wall-crossing contraction in tilt-stability.
\item[{\rm (2c)}]\label{item:main2c} If $W$ is very general, then $Z(W)$ is isomorphic to a moduli space of Bridgeland stable objects in $\Ku(W)$ with Chern character $\mathbf{v}_2$.
\end{enumerate}
\end{enumerate}
\end{theo}

Part (2c) of the above result, together with Theorem \ref{thm:YoshiokaMain}, implies that $Z(W)$ is deformation equivalent to a Hilbert scheme of $4$ points on a K3 surface. This result was first proved by Addington and Lehn in \cite{AL:8fold}.

The last part of Theorem \ref{thm:twistcub} has been recently improved in \cite[Theorem 1.2]{lpz}.

\begin{theo}[Li-Pertusi-Zhao]\label{thm:twistLPZ}
Let W be a smooth cubic fourfold not containing a plane and let $\sigma\in\mathrm{Stab}^\dagger(\Ku(W))$ such that $\eta(\sigma)\in (A_2)_\C\cap\cP\subseteq\cP_0$. The smooth projective hyperk\"ahler eightfold $Z(W)$ is isomorphic to the moduli space $M_\sigma(\Ku(W),2\llambda_1+\llambda_2)$.
\end{theo}

It is quite natural to ask what happens to the moduli space $M_\sigma(\Ku(W),2\llambda_1+\llambda_2)$ when $W$ degenerates to a cubic fourfold containing a plane. In this situation, any $\P^3$ containing the plane cuts a non-integral surface of degree $3$. Hence some objects in the moduli space are properly semistable. In \cite{Ouchi:Cubics}, Ouchi described a moduli space of stable objects in $\Ku(W)$, for $W$ a cubic fourfold containing a plane, such that $W$ embeds into it as a Lagrangian submanifold.

\end{document}